\documentclass{amsart}

\usepackage{amssymb}
\usepackage{amsthm}
\usepackage{amsmath}
\usepackage{mathtools}
\usepackage{tikz}
\usepackage{enumitem}

\usepackage[scr=boondox]{mathalpha}

\usepackage{hyperref}

\theoremstyle{plain}
\newtheorem{proposition}{Proposition}[section]
\newtheorem{theorem}[proposition]{Theorem}
\newtheorem{lemma}[proposition]{Lemma}
\newtheorem{corollary}[proposition]{Corollary}
\newtheorem{fact}[proposition]{Fact}
\theoremstyle{definition}

\newtheorem{definition}[proposition]{Definition}
\newtheorem{observation}[proposition]{Observation}
\theoremstyle{remark}
\newtheorem{remark}[proposition]{Remark}

\DeclareMathOperator{\supp}{supp}

\DeclareMathOperator{\id}{id}

\DeclareMathOperator{\dist}{d}

\DeclareMathOperator{\dirac}{\mathcal{D}}

\DeclareMathOperator{\Oc}{\mathcal{O}}

\DeclareMathOperator{\Z}{\mathbb{Z}}
\DeclareMathOperator{\R}{\mathbb{R}}
\DeclareMathOperator{\Rb}{\mathbb{R}}

\newcommand{\abs}[1]{\left|#1\right|}

\newcommand{\norm}[1]{\left\|#1\right\|}

\newcommand{\mc}{\mathcal}
\newcommand{\mr}{\mathrm}
\newcommand{\mb}{\mathbb}

\newcommand{\ms}{\mathsf}
\newcommand{\resp}{resp.\ }

\renewcommand{\emptyset}{\varnothing}

\begin{document}

\title{Patterson--Sullivan theory for coarse cocycles}
\author[Blayac]{Pierre-Louis Blayac}
\address{University of Michigan}
\author[Canary]{Richard Canary}
\address{University of Michigan}
\author[Zhu]{Feng Zhu}
\address{University of Wisconsin-Madison}
\author[Zimmer]{Andrew Zimmer}
\address{University of Wisconsin-Madison}
\thanks{Canary was partially supported by grant DMS-2304636 from the National Science Foundation.
Zhu was partially supported by an AMS-Simons Travel Grant.
Zimmer was partially supported by a Sloan research fellowship and grant DMS-2105580  from the National Science Foundation.}
\date{\today}
\subjclass[2010]{}

\begin{abstract} In this paper we develop a theory of Patterson--Sullivan measures associated to coarse cocycles of convergence groups. This framework includes Patterson--Sullivan measures associated to the Busemann cocycle on the geodesic boundary of a Gromov hyperbolic metric spaces and Patterson--Sullivan measures on flag manifolds associated to Anosov (or more general transverse) subgroups of semisimple Lie groups, as well as more examples. Under some natural geometric assumptions on the coarse cocycle, we prove existence, uniqueness, and ergodicity results. 

\end{abstract}

\maketitle

\setcounter{tocdepth}{1}

\tableofcontents


\section{Introduction} 

Patterson--Sullivan measures were first constructed by Patterson \cite{Patterson} in the setting of Fuchsian groups and by Sullivan \cite{Sullivan1979} for
Kleinian groups. They have been used to study the dynamics of the recurrent part of the geodesic flow of the quotient manifold,
the geometry of the limit set of the group and to obtain counting estimates for both orbit points of the group  and closed geodesics in the quotient manifold.
They have been generalized to many settings, including proper isometric actions on Gromov hyperbolic spaces \cite{roblin,BF2017,CDST2018}
and discrete subgroups of semisimple Lie
groups \cite{Albuquerque,quint}.

In this paper we develop a theory of Patterson--Sullivan measures for coarse-cocycles of convergence group actions, which encompasses many of the previous
situations. When the coarse-cocycle has an expanding property and a finite critical exponent, we show that Patterson--Sullivan measures exist in the critical dimension.
Moreover, we establish a Shadow Lemma in the spirit of Sullivan and show that the action of the convergence group is ergodic with respect to the measure
when the associated Poincar\'e series diverges at its critical exponent.

We also develop the notion of a coarse Gromov--Patterson--Sullivan (GPS) system, which is a pair of coarse-cocycles  with an associated coarse Gromov product, and establish a version of the Hopf--Tsuji--Sullivan ergodic dichotomy in this setting. In a companion paper, we will use this framework to establish mixing, equidistribution
and counting results for relatively Anosov groups (and more generally for certain GPS systems for geometrically finite convergence groups).

\subsection{Main results}
Suppose $\Gamma \subset \mathsf{Homeo}(M)$ is a convergence group (see Section~\ref{sec:compactification} for definitions).
A function $\sigma \colon \Gamma \times M \to \Rb$ 
is called a \emph{$\kappa$-coarse-cocycle} if:
\begin{enumerate}
\item For every $\gamma \in \Gamma$, the function $\sigma(\gamma, \cdot)$ is \emph{$\kappa$-coarsely continuous}: if $x_0 \in M$, then 
$$
\limsup_{x\to x_0} \abs{  \sigma(\gamma, x_0)- \sigma(\gamma, x) } \leq \kappa.
$$
\item $\sigma$ satisfies a coarse version of the cocycle identity: if $\gamma_1, \gamma_2 \in \Gamma$ and $x \in M$, then
$$
\abs{ \sigma(\gamma_1\gamma_2, x) - \Big(\sigma(\gamma_1, \gamma_2 x) +\sigma(\gamma_2, x) \Big)} \leq \kappa.
$$
\end{enumerate}

Notice that  a $0$-coarse-cocycle is simply a continuous cocycle. In the classical hyperbolic setting, one usually considers the Busemann
cocycle.

For such a coarse-cocycle, Patterson--Sullivan measures are naturally defined as follows. 

\begin{definition}\label{defn:patterson-sullivan}
Suppose $\Gamma \subset \mathsf{Homeo}(M)$ is a convergence group and $\sigma \colon \Gamma\times M \to \Rb$ is a coarse-cocycle, then
a probability measure $\mu$ on $M$ is
a \emph{$C$-coarse $\sigma$-Patterson--Sullivan measure} of dimension $\delta$ if, for every $\gamma \in \Gamma$, the measures $\mu$ and $\gamma_*\mu$ are absolutely continuous 
and 
$$
e^{-C-\delta \sigma(\gamma^{-1}, \cdot)} \leq \frac{d\gamma_* \mu}{d\mu} \leq e^{C-\delta \sigma(\gamma^{-1}, \cdot)} 
$$
$\mu$-almost everywhere. 
\end{definition}

In this paper we develop a theory of Patterson--Sullivan measures for coarse-cocycles which have certain geometric properties. In Section~\ref{sec:compactification}, we will show that the set $\Gamma \sqcup M$ has a unique topology which makes it a compact metrizable space and where the natural action of $\Gamma$ on $\Gamma \sqcup M$ is a convergence group action. We call a metric on $\Gamma \sqcup M$ which generates this topology a \emph{compatible metric}. Using such a metric we make the following definitions.

\begin{definition}\label{defn:expanding} Suppose $\Gamma \subset \mathsf{Homeo}(M)$ is a non-elementary convergence group and $\dist$ is a compatible metric on $\Gamma \sqcup M$.  A coarse-cocycle $\sigma \colon \Gamma \times M \to \Rb$ is
\begin{enumerate} 
\item \emph{proper} if $\sigma(\gamma_n,\gamma_n^+) \rightarrow +\infty$ whenever $\{\gamma_n\} \subset \Gamma$ is a sequence of distinct loxodromic elements where the sequence of pairs of repelling/attracting points satisfies  $\liminf_{n \rightarrow \infty} \dist(\gamma_n^-, \gamma_n^+) > 0$. 
\item \emph{expanding} if it is proper and for every $\gamma \in \Gamma$ there is a number $\norm{\gamma}_\sigma \in \Rb$, called the  \emph{$\sigma$-magnitude of $\gamma$}, with the following property:  
\begin{itemize}
\item[] for every $\epsilon > 0$ there exists $C > 0$ such that: whenever $x \in M$, $\gamma \in \Gamma$ and $\dist(x,\gamma^{-1}) > \epsilon$, then 
$$
\norm{\gamma}_\sigma - C \leq \sigma(\gamma, x) \leq \norm{\gamma}_\sigma+C.
$$ 
\end{itemize} 
\end{enumerate} 
\end{definition} 

Given an expanding coarse-cocycle $\sigma \colon \Gamma \times M \to \Rb$, the \emph{$\sigma$-Poincar\'e series} is
$$
Q_\sigma(s) = \sum_{\gamma \in \Gamma} e^{-s \norm{\gamma}_\sigma}\in[0,+\infty]
$$
and the \emph{$\sigma$-critical exponent} is 
$$
\delta_\sigma(\Gamma) = \inf \left\{ s > 0 : Q_\sigma(s) < +\infty\right\}\in [0,+\infty]. 
$$
Although the magnitude function is not uniquely defined, any two choices will differ by a uniformly bounded amount and hence $\delta_\sigma(\Gamma)$ is independent of the choice of particular magnitude function.
Whether $Q_\sigma(\delta_\sigma(\Gamma))=+\infty$ is also independent of the choice of magnitude.

We show that if a coarse-cocycle is expanding and has finite critical exponent $\delta_\sigma(\Gamma)$,  then it admits a coarse Patterson--Sullivan measure of
dimension $\delta_\sigma(\Gamma)$ which is supported on the limit set. Moreover, any coarse Patterson--Sullivan measure has dimension at least $\delta_\sigma(\Gamma)$.

We also establish a Shadow Lemma for coarse Patterson--Sullivan measures and use it to study the associated Patterson--Sullivan measures. In particular,
we establish ergodicity of the action when the Poincar\'e series diverges at its critical exponent.

\begin{theorem}[see Theorem~\ref{thm:ergodicity on single M}]\label{uniqueness and ergodicity}Suppose $\Gamma \subset \mathsf{Homeo}(M)$ is a non-elementary convergence group and $\sigma \colon \Gamma \times M \to \Rb$ is an expanding coarse-cocycle with $\delta:=\delta_\sigma(\Gamma) < +\infty$. If $\mu$ is a $C$-coarse $\sigma$-Patterson--Sullivan measure of dimension $\delta$ and
$$
\sum_{\gamma \in \Gamma} e^{-\delta \norm{\gamma}_\sigma} = + \infty,
$$
 then:
 \begin{enumerate}
 \item $\Gamma$ acts ergodically on $(M, \mu)$.
 \item $\mu$ is coarsely unique in the following sense: if $\lambda$ is a $C$-coarse $\sigma$-Patterson--Sullivan measure of dimension $\delta$, then $e^{-4C} \mu \leq \lambda \leq e^{4C}\mu$.
  \item The conical limit set of $\Gamma$ has full $\mu$-measure.  
  \end{enumerate} 
 \end{theorem} 

As an application of ergodicity in the divergent case, we prove the following rigidity result for Patterson--Sullivan measures. 

\begin{proposition}[see Propositions~\ref{prop:abscts or singular} and~\ref{prop:characterisation of abscts case}]\label{prop:rigidity of PS measures intro} Suppose $\Gamma \subset \mathsf{Homeo}(M)$ is a  non-elementary convergence group and $\sigma_1, \sigma_2 \colon \Gamma \times M \to \Rb$ are expanding coarse-cocycles. For $i=1,2$, let  $\mu_i$ be a coarse $\sigma_i$-Patterson--Sullivan measure of dimension $\delta_i$. 

If  $\sum_{\gamma\in\Gamma} e^{-\delta_1 \norm{\gamma}_{\sigma_1}} = +\infty$, then either: 
\begin{enumerate}
\item $\mu_1 \perp \mu_2$, or
\item $\mu_1 \ll \mu_2$ and $\mu_2 \ll \mu_1$. Moreover, in this case 
$$
\sup_{\gamma \in \Gamma}\abs{ \delta_1\norm{\gamma}_{\sigma_1} - \delta_2\norm{\gamma}_{\sigma_2} } < \infty.
$$
\end{enumerate} 
\end{proposition}

\begin{remark} As described in Section~\ref{subsec:transverse groups} below, the Patterson--Sullivan measures studied in this paper include, as a special case, the Patterson--Sullivan measures associated to transverse discrete subgroups of semisimple Lie groups. In this special case, Kim proved a version of the above proposition for a special class of Zariski-dense transverse groups~\cite{Kim2024}. The above proposition shows that the extra assumptions in~\cite{Kim2024} can be removed.  
\end{remark}

Using this rigidity result we establish a strict convexity result for the critical exponent. 

\begin{theorem}[see Theorem~\ref{thm:convexity of entropy}]\label{thm:convexity of entropy in intro}  Suppose $\Gamma \subset \mathsf{Homeo}(M)$ is a non-elementary convergence group and $\sigma_0, \sigma_1 \colon \Gamma \times M \to \Rb$ are expanding coarse-cocycles with finite critical exponents $\delta_{\sigma_0}(\Gamma)= \delta_{\sigma_1}(\Gamma)=1$. For $0 < \lambda < 1$, let $\sigma_\lambda = \lambda \sigma_0 + (1-\lambda)\sigma_1$. Then 
$$\delta_{\sigma_\lambda}(\Gamma) \leq 1.$$
Moreover, if $\sum_{\gamma \in \Gamma} e^{-\delta_{\sigma_\lambda}(\Gamma) \norm{\gamma}_{\sigma_\lambda}} = +\infty$, then the following are equivalent:
\begin{enumerate}
\item $\delta_{\sigma_\lambda}(\Gamma) = 1$.
\item $\sup_{\gamma \in \Gamma} \abs{ \norm{\gamma}_{\sigma_0} - \norm{\gamma}_{\sigma_1}} < +\infty$.
\end{enumerate}
\end{theorem} 

In the context of Theorem~\ref{thm:convexity of entropy in intro}, if $\sigma_0$ and $\sigma_1$ do not have coarsely equivalent magnitudes, then one obtains a drop in critical exponent when taking a convex combination of $\sigma_0$ and $\sigma_1$. These types of strict convexity results can be used to prove entropy rigidity results, see for instance~\cite{PS2017} and \cite{CZZ2}.

We further study coarse-cocycles which have a well-behaved ``dual cocycle'' and coarse Gromov product.

\begin{definition} \label{def:GPS}
Suppose $\Gamma \subset \mathsf{Homeo}(M)$ is a non-elementary convergence group and let $M^{(2)} = \{ (x,y) \in M^2 : x \neq y\}$.
We say that $(\sigma, \bar{\sigma}, G)$ is a $\kappa$-coarse \emph{Gromov--Patterson--Sullivan system} (or \emph{GPS system}) if   $\sigma, \bar{\sigma} \colon \Gamma \times M \rightarrow \Rb$ are proper $\kappa$-coarse-cocycles, $G \colon M^{(2)} \to \Rb$ is a locally bounded function, and 
$$
\abs{ \Big( \bar\sigma(\gamma, x) + \sigma(\gamma, y) \Big) - \Big( G(\gamma x, \gamma y) - G(x,y) \Big)} \leq \kappa
$$
for all $\gamma \in \Gamma$ and distinct $x,y \in M$. When $\kappa=0$ and $G$ is continuous, we say that $(\sigma, \bar{\sigma}, G)$ is a \emph{continuous GPS system}. 
\end{definition}

We prove that if  $(\sigma, \bar{\sigma}, G)$ is a coarse GPS system, then $\sigma$ and $\bar\sigma$ are expanding (see Proposition~\ref{prop:GPS implies expanding}).
We construct a measurable flow space associated to a GPS system and use the Patterson--Sullivan measures of $\sigma$ and $\bar\sigma$ and the Gromov
product to give it a Bowen--Margulis--Sullivan measure. We will show that the dynamics of this flow space are controlled by the behavior of the Poincar\'e series at the critical exponent and use this to establish the following version of the Hopf--Tsuji--Sullivan dichotomy.

\begin{theorem}[see Section~\ref{sec:proof of dichotomy}]
\label{our dichotomy}
Suppose $(\sigma, \bar\sigma, G)$ is a coarse GPS system and $\delta_\sigma(\Gamma) < +\infty$. Let $\mu$ and $\bar \mu$ be Patterson--Sullivan measures of dimension $\delta$ for $\sigma$ and $\bar{\sigma}$ respectively. 
Then there exists a measurable function $\tilde G$ on $M^{(2)}$ such that 
$$\nu:=e^{\delta \tilde G}\bar\mu\otimes\mu$$
is $\Gamma$-invariant.
Moreover we have the following dichotomy:
\begin{enumerate} 
\item If $\sum_{\gamma \in \Gamma} e^{-\delta \norm{\gamma}_\sigma} = +\infty$, then:
\begin{enumerate}
\item $\delta = \delta_\sigma(\Gamma)$.
\item $\mu( \Lambda^{\rm con}(\Gamma)) = 1 = \bar\mu(\Lambda^{\rm con}(\Gamma))$. 
\item The $\Gamma$ action on $(M^{(2)},\nu)$ is ergodic and conservative. 
\end{enumerate} 
\item If $\sum_{\gamma \in \Gamma} e^{-\delta \norm{\gamma}_\sigma} < +\infty$, then:
\begin{enumerate}
\item $\delta \geq \delta_\sigma(\Gamma)$.
\item $\mu( \Lambda^{\rm con}(\Gamma)) = 0 = \bar\mu(\Lambda^{\rm con}(\Gamma))$. 
\item The $\Gamma$ action on $(M^{(2)}, \nu)$ is non-ergodic and dissipative. 
\end{enumerate} 
\end{enumerate} 

\end{theorem} 

In the theorem above, $\Lambda^{\rm con}(\Gamma)$ denotes the set of conical limit points. We provide the definitions of conservative and dissipative actions, and state their basic properties, in Appendix~\ref{appendix:conservative and dissipative}.

\subsection{Motivating examples}\label{sec:motivating examples}
We now discuss a range of examples which our approach to Patterson--Sullivan theory treats in a unified way.

\subsubsection{Transverse subgroups of semi-simple Lie groups}\label{subsec:transverse groups} In the sequel to this paper~\cite{papertwo} we apply the framework developed here to study Patterson--Sullivan measures for certain classes of discrete subgroups of semi-simple Lie groups. We show, among other things, that the ergodic dichotomy for transverse groups established in~\cite{CZZ2023a,KOW2023} is a particular case of the dichotomy established in this paper. For more details, see ~\cite[Section 11]{papertwo}.

\subsubsection{Proper actions on Gromov hyperbolic spaces}

If $X$ is a proper geodesic Gromov hyperbolic metric space and $\Gamma \subset \mathsf{Isom}(X)$ is discrete, then $\Gamma$ acts on the Gromov boundary $\partial_\infty X$ 
as a convergence group (see \cite[Th.\,3A]{Tukia1994} or \cite{Freden95}).
If we fix a base point $o \in X$, we can define,
and for each $x \in \partial_\infty X$, a {\em Busemann function}
$$b_x  \colon X \rightarrow \Rb \quad\text{by setting}\quad b_x (q) = \limsup_{p \rightarrow x} \dist(p,q)-\dist(p,o).$$
The \emph{Busemann coarse-cocycle} $\beta\colon\Gamma\times\partial_\infty X\to\mathbb R$ is defined by 
$$\beta(\gamma, x) = b_x(\gamma^{-1}(o)).$$
When $X$ is ${\rm CAT}(-1)$ (e.g. $X=\mathbb H^n$) this is a continuous cocycle, but in general it will only be a coarse-cocycle.

The Gromov product $G \colon \partial_\infty X^{(2)} \to \Rb$ is classically defined by 
$$
G(x,y) = \limsup_{p \rightarrow x, q \rightarrow y} \dist(o,p)+\dist(o,q) - \dist(p,q).
$$ 
Then $(\beta, \beta, G)$ is a coarse GPS system, which is not always continuous. Further, one can choose 
$$
\norm{\gamma}_\beta = \dist(o, \gamma(o)). 
$$
Hence, $\delta_\beta(\Gamma)$ is the critical exponent of the usual Poincar\'e series
$$
Q(s) = \sum_{\gamma \in \Gamma} e^{-s \dist(o,\gamma(o))}.
$$

When $X$ is ${\rm CAT}(-1)$, Roblin~\cite{roblin} proved the Hopf--Tsuji--Sullivan dichotomy for the GPS system $(\beta, \beta, G)$, see also work of Burger--Mozes~\cite{BurgerMozes1996}. 
Building upon work of Bader--Furman~\cite{BF2017}, Coulon--Dougall--Schapira--Tapie~\cite{CDST2018} extended this to the case of general proper geodesic Gromov hyperbolic metric spaces. 

Das--Simmons--Urba\'{n}ski have studied the case of improper Gromov hyperbolic metric spaces~\cite{DasSimmonsUrbanski}.

\subsubsection{Coarsely additive potentials} We continue to assume that $X$ is a proper geodesic Gromov hyperbolic metric space and $\Gamma \subset \mathsf{Isom}(X)$ is discrete.

The following definition is adapted from Cantrell--Tanaka~\cite[Def.\ 2.2]{CantrellTanaka_measures}. 

\begin{definition}\label{defn:coarsely additive potential}
A function $\psi \colon X \times X \to \Rb$ is a \emph{coarsely additive potential} if 
\begin{enumerate}
\item\label{item:potential proper} $\lim_{r \rightarrow \infty} \inf_{\dist_X(p,q) \geq r} \psi(p,q) = +\infty$, 
\item\label{item:potential bounded} for any  $r>0$,  
$$\sup_{\dist_X(p,q) \leq r}\abs{ \psi(p,q)} < +\infty,
$$ 
\item\label{item:potential coarsely additive} for every $r > 0$ there exists $\kappa=\kappa(r) > 0$ such that: if $u$ is contained in the $r$-neighborhood of a geodesic in $(X,\dist_X)$ joining $p$ to $q$, then 
$$
\abs{\psi(p,q) - \big(\psi(p,u) + \psi(u,q)\big) } \leq \kappa.
$$
\end{enumerate} 
\end{definition} 

\begin{remark}\label{remarke:CT work} Cantrell--Tanaka consider the case when $\Gamma$ is word hyperbolic and $X=\Gamma$ with a word metric. In this case they introduce \emph{tempered potentials} which are functions $\psi : \Gamma \times \Gamma \rightarrow \Rb$  which satisfy ~\eqref{item:potential coarsely additive} and another property they call (QE). 
In their results they consider the case when $\psi$ is $\Gamma$-invariant (which implies~\eqref{item:potential bounded}) and has finite ``exponent'' (which implies  ~\eqref{item:potential proper}). 
Dilsavor--Thompson proved that condition (QE) is a consequence of ~\eqref{item:potential coarsely additive}  and $\Gamma$-invariance \cite[Rem.\ 3.14]{DT2023} in geodesic hyperbolic spaces, see also Lemma~\ref{lem:formula for Gromov product} for a proof using the language of this paper (which uses also property~\eqref{item:potential proper}).
So, in the case considered by Cantrell--Tanaka the two definitions coincide.
\end{remark} 

We will show that any $\Gamma$-invariant potential gives rise to an expanding coarse-cocycle on $\partial_\infty X$, and that when $\Gamma$ acts co-compactly on $X$, every expanding coarse-cocyle arises in this way. 

\begin{theorem}[see Theorem~\ref{thm:hyperbolic1_potential}]\label{thm: coarsely additive potentials introduction} 
Suppose $\psi$ is a $\Gamma$-invariant coarsely additive potential. 
Define functions $\sigma_{\psi}, \bar\sigma_{\psi} \colon \Gamma \times \partial_\infty X \to [-\infty,\infty]$ and $G_{\psi} \colon \partial_\infty X^{(2)} \to  [-\infty,\infty]$ by 
\begin{align*}
\sigma_{\psi}(\gamma,x) & = \limsup_{p \rightarrow x} \, \psi(\gamma^{-1}o, p) - \psi(o,p), \\ 
\bar\sigma_{\psi}(\gamma,x) & = \limsup_{p \rightarrow x} \, \psi(p,\gamma^{-1}o) - \psi(p,o), \\
G_{\psi}(x,y) & = \limsup_{p \rightarrow x, q \rightarrow y}\,  \psi(p,o) +\psi(o,q) - \psi(p,q).
\end{align*}
Then these quantities are finite, with $G_\psi$ bounded below, $(\sigma_{\psi}, \bar\sigma_{\psi}, G_{\psi})$ is a coarse GPS-system, and one can choose 
$$
\norm{\gamma}_{\sigma_{\psi}} = \psi(o,\gamma o).
$$
\end{theorem} 

\begin{remark} Similar results appear in earlier work of Dilsavor--Thompson~\cite[Section 3.5]{DT2023} using different language. In fact, most of Theorem~\ref{thm: coarsely additive potentials introduction} can be derived from their results, but for the reader's convenience we include the complete argument. 
\end{remark}

\begin{theorem}[see Theorem~\ref{thm:hyperbolic2_potential}] \label{thm:hyperbolic2_potential introduction} Suppose $\Gamma$ acts co-compactly on $X$ and $\sigma\colon \Gamma \times \partial_\infty X \to \Rb$ is an expanding coarse-cocycle. Then there exists a $\Gamma$-invariant coarsely additive potential $\psi: X \times X \rightarrow \Rb$ where
$$
\sup_{\gamma \in \Gamma, x \in \partial_\infty X} \abs{\sigma_{\psi}(\gamma,x)-\sigma(\gamma,x) } < + \infty.
$$
In particular, $\sigma$ is contained in a coarse GPS-system by Theorem~\ref{thm: coarsely additive potentials introduction}.
\end{theorem} 

One can also interpret coarsely additive potentials as $\Gamma$-invariant coarsely-geodesic quasimetrics on $X$, see Section~\ref{sec:metric perspective} below.

The next two subsections highlight two previously studied examples that can be interpreted in terms of GPS systems associated to coarsely additive potentials.

\subsubsection{H\"older potentials and cocycles} Next we describe the setting studied in work of Paulin--Pollicott--Schapira~\cite[Section 3]{PPS}, see also earlier work of Ledrappier~\cite{Ledrappier} and later work of Dilsavor--Thompson~\cite{DT2023}.

Let $X$ be a simply connected complete Riemannian manifold with pinched negative curvature and suppose $\Gamma \subset \mathsf{Isom}(X)$ is discrete.  
Then let $F \colon T^1 X \to \Rb$ be a $\Gamma$-invariant H\"older function.
For $p, q \in X$ define 
$$
\int_p^q F : = \int_0^T F(\ell'(t))dt
$$
where $\ell \colon [0,T] \to X$ is the unit speed geodesic joining $p$ to $q$. 
Among (many) other things, \cite{PPS} consider counting for the ``magnitudes''
$$
\int_o^{\gamma o} F.
$$
To accomplish this they develop a theory of Patterson--Sullivan measures, Busemann cocycles, and Gromov products in this setting.

By~\cite[Lemma 3.2]{PPS}, the function $(p,q) \mapsto \int_p^q F$  satisfies property~\eqref{item:potential coarsely additive} of Definition~\ref{defn:coarsely additive potential}.
To ensure the two other properties~\eqref{item:potential proper} and \eqref{item:potential bounded} of our definition of coarsely additive potentials, we need to make an extra boundedness assumption on the potential not made in \cite{PPS}:
$$
0 < \inf_{v \in T^1 X} F(v) \leq \sup_{v \in T^1 X} F(v) < +\infty.
$$
Note that for many purposes (as in \cite{PPS}), it is harmless to add a constant to $F$, so the positivity assumption above is not much more restrictive that assuming $F$ is bounded below.

Given the additional boundedness assumption, $(p,q) \mapsto \int_p^q F$ defines a coarsely additive potential.
Further the Busemann cocycle and Gromov product introduced in \cite{PPS} (essentially) correspond to the definitions in Theorem~\ref{thm: coarsely additive potentials introduction}.
Finally, one can in this setting improve, using \cite[Lemma 3.2]{PPS}, the proof of Theorem~\ref{thm: coarsely additive potentials introduction} (see Section~\ref{sec:potentials}) to show that the GPS system we construct is continuous.

\subsubsection{Hitting measures of random walks}

Next let $\Gamma$ be a word hyperbolic group and let $\lambda$ be a finitely-supported probability measure on $\Gamma$ with $\langle \supp \lambda \rangle = \Gamma$.
If $g_1,g_2,\dots \in \Gamma$ are random group elements following the distribution $\lambda$, then the location of the random walk $X_n = g_1 \cdots g_n$ follows the distribution $\lambda^{*n}$. 
The Green metric, introduced in~\cite{BB2007}, is the left-invariant function $\dist_\lambda$ on $\Gamma \times \Gamma$ defined by $\dist_\lambda(x,y) = -\log F(x,y)$, where $F(x,y)$ is the probability that the random walk started at $x$ ever hits $y$.

We claim that $\dist_\lambda$  is a coarsely additive potential. Property~\eqref{item:potential bounded} follows from the fact that 
$$
\dist_\lambda(\alpha,\beta) \leq \inf_{n \geq 1} -\log \lambda^{*n}(\alpha^{-1}\beta)<\infty \quad \text{(since $\supp\lambda$ generates $\Gamma$)}.
$$
Property~\eqref{item:potential proper} follows from \cite[Prop.\ 3.1]{BHM8}. Property~\eqref{item:potential coarsely additive} follows from a result of Ancona (see \cite[Thm.\ 27.11]{Woess}):
for any $r \geq 0$, there exists a positive constant $C(r)$ such that
$$F(\id,\gamma) \leq C(r) F(\id,\gamma') F(\gamma',\gamma)$$
whenever $\gamma,\gamma' \in \Gamma$ and $\gamma'$ has (word) distance at most $r$ from a geodesic segment between $\id$ and $\gamma$ in a Cayley graph. Hence $\dist_\lambda$  is a coarsely additive potential. 

Thus Theorem~\ref{thm: coarsely additive potentials introduction} can be applied to conclude: $(\sigma,\bar\sigma,G) := (\sigma_\lambda,\sigma_{\bar\lambda},G_\lambda)$ is a GPS system for $\Gamma \subset \mathsf{Homeo}(\partial_\infty\Gamma)$, where 
\begin{itemize}
	\item $\sigma_\lambda(\gamma,x) = \limsup_{\alpha \to x} \dist_\lambda(\id,\gamma\alpha) - \dist_\lambda(\id,\alpha)$; 
	\item $\bar\lambda$ is the probability measure on $\Gamma$ defined by $\bar\lambda(\gamma) := \lambda(\gamma^{-1})$; \item $G_\lambda(x,y) := \limsup_{\alpha\to x, \beta\to y} \dist_\lambda(\alpha,\id) + \dist_\lambda(\id,\beta) - \dist_\lambda(\alpha,\beta)$.
\end{itemize}

The cocycle $\sigma_\lambda$ also satisfies
$$
\sigma_\lambda(\gamma,\xi) = -\log \frac{d\gamma^{-1}_*\nu}{d\nu}(\xi)
$$
where $\nu$ is the unique $\lambda$-stationary measure on $\partial\Gamma$, i.e.\ harmonic measure or hitting measure associated to $\lambda$
\cite[Prop.\,2.5]{GMM}. Hence the  coarse $\sigma$-Patterson--Sullivan measures are absolutely continuous with respect to the hitting measure $\nu$
by Theorem~\ref{uniqueness and ergodicity} (we are in the divergent case of Theorem~\ref{our dichotomy} since $\Lambda^{\mathrm{con}}(\Gamma)=\Lambda(\Gamma)$).

\subsubsection{Proper actions on ${\rm CAT}(0)$ visibility spaces}

Finally, we briefly discuss another set of examples our results encompass, involving spaces which need not be uniformly hyperbolic.
Let $X$ be a proper ${\rm CAT}(0)$ space with base point $o\in X$ and visual boundary $\partial X$. Suppose $X$ is \emph{visible}, i.e.\ all $\xi\neq\eta\in\partial X$ can be connected by a bi-infinite geodesic in $X$ (this notion was introduced by Eberlein--O'Neill \cite{Visibility}, see also \cite[Def.\,III.9.28]{BridsonHaefliger}).
Let $\Gamma$ be a discrete group of isometries of $X$.

Then $\Gamma$ acts on $\partial X$ as a convergence group~\cite[Th.\,1]{Karlsson}.
The Busemann functions $(x,y,z)\in X^3\mapsto b_z(x,y)=d(x,z)-d(y,z)$ extend continuously to $(x,y,z)\in X^2\times (X\cup \partial X)$ (see \cite[p.\,267]{BridsonHaefliger}), and $\sigma(\gamma,\xi)=b_\xi(\gamma^{-1}o,o)$ defines a continuous cocycle $\Gamma\times\partial X\to\R$. 
Finally, setting $G(\xi,\eta)=-\inf_{x\in X}(b_\xi+b_\eta)(x,o)$, we obtain a continuous GPS system $(\sigma,\sigma,G)$, see e.g.~ \cite[p.\,948]{Ricks}. 
(Because of the visibility assumption, every geodesic in $X$ is rank one in the sense of \cite{Ricks}.)

If $\Gamma$ acts co-compactly on $X$, then $X$ is Gromov hyperbolic (\cite{Visibility}, see also \cite[III.H.1.4]{BridsonHaefliger}).
Otherwise, $X$ may not be Gromov hyperbolic. For example, given geodesics in the hyperbolic plane at distance at least $1$ from one another, the surface obtained by grafting flat strips (of any widths) along these geodesics is always ${\rm CAT}(0)$ and visible.

We note that in the ${\rm CAT}(0)$ setting, Patterson--Sullivan theory has been successfully developed in the more general setting of rank one actions, see for instance~\cite{Knieper1997,Ricks,Link2018}.

\subsection{Outline of the paper} In many theories of Patterson--Sullivan measures, the measures live on the boundary of a metric space and this 
metric space is used in an essential way in the study these measures. The first part of this paper (Sections~\ref{sec:compactification} to~\ref{sec:the shadow lemma}) 
is devoted to developing a perspective for studying these measures without the presence of a metric space. 

We first observe, in Section~\ref{sec:compactification}, that the set $\Gamma \sqcup M$ has a topology which makes it a compact metrizable space
(the existence of this topology is implicit in work of  Bowditch, see  \cite{Bowditch99}). 
In Section~\ref{sec:expanding cocycles}, we study properties of cocycles and prove that the coarse-cocycles in a coarse GPS system are expanding. 

We also establish the following property, which allows us to regard our cocycles as the ``Busemann cocycle'' on the ``Busemann boundary'' associated to the metric-like function $\rho(\alpha, \beta) = \norm{\alpha^{-1} \beta}_\sigma$ on $\Gamma$. Note that a priori $\rho$ is not symmetric and does not satisfy the triangle inequality.

\begin{lemma}[see Lemma~\ref{lem:nice magnitude}]
Suppose $\Gamma \subset \mathsf{Homeo}(M)$ is a non-elementary convergence group and $\sigma$ is an expanding $\kappa$-coarse-cocycle.  
Then there is a choice of magnitude $\norm{\cdot}_\sigma$ such that: if $x \in \Lambda(\Gamma)$ and $\alpha \in \Gamma$, then 
$$
\limsup_{\gamma \rightarrow x} \abs{ \sigma(\alpha,x)-(\norm{\alpha \gamma}_\sigma-\norm{\gamma}_\sigma)  } \leq 2\kappa. 
$$
\end{lemma}

We use this result and Patterson's original argument to show that Patterson--Sullivan measures exist in the critical dimension. 

\medskip\noindent
{\bf Theorem~\ref{thm:PS measures exist}.} {\em If $\sigma$ is an $\kappa$-coarse expanding cocycle for a non-elementary convergence group $\Gamma\subset \mathsf{Homeo}(M)$ and
$\delta:=\delta_\sigma(\Gamma) < + \infty$, then there exists a $2\kappa\delta$-coarse $\sigma$-Patterson--Sullivan measure of dimension $\delta$, which is supported on the limit set $\Lambda(\Gamma)$. 
}

\medskip

One nearly immediate consequence of the existence of a Patterson--Sullivan measure is a result guaranteeing decrease of critical exponent in the
spirit of Dal'bo--Otal--Peign\'e \cite[Prop.\,2]{DOP}.

\medskip\noindent
{\bf Theorem \ref{thm:subgp div implies smaller crit exp}.} {\em
Suppose $\Gamma \subset \mathsf{Homeo}(M)$ is a non-elementary convergence group, $\sigma$ is an expanding coarse-cocycle, and $\delta_\sigma(\Gamma) < +\infty$. If $G \subset \Gamma$ is a subgroup where $\Lambda(G)$ is a strict subset of $\Lambda(\Gamma)$ and 
$$
\sum_{g \in G} e^{-\delta_\sigma(G)\norm{g}_\sigma}=+\infty,
$$ 
then $\delta_\sigma(G) < \delta_\sigma(\Gamma)$. 
}

\medskip

The next key step in the paper is to define shadows in our setting and prove a version of the Shadow Lemma. 
To define shadows we borrow an idea from the theory of Patterson--Sullivan measures associated to Zariski-dense discrete subgroups in semisimple Lie groups (compare the shadows below to the sets $\gamma B_{\theta, \gamma}^\epsilon$ in~\cite[Lem.\ 8.2]{quint}).

\begin{definition} \label{def:shadows}
Suppose $\Gamma \subset \mathsf{Homeo}(M)$ is a non-elementary convergence group and $\dist$ is a compatible metric on $\Gamma \sqcup M$. Given  $\epsilon>0$ and $\gamma\in\Gamma$, the associated \emph{shadow} is
$$
\mc S_\epsilon(\gamma):=\gamma\left(M-B_\epsilon(\gamma^{-1})\right).
$$
where $B_\epsilon(\gamma^{-1})$ is the open ball centered at $\gamma^{-1}$ of radius $\epsilon$ with respect to the metric $\dist$. 
\end{definition}

In Section~\ref{sec:basic properties of fake shadows}, we establish some basic properties of shadows, relate shadows to a notion of uniformly conical limit points, and compare these shadows to the classically-defined shadows in the Gromov hyperbolic setting.
In Section~\ref{sec:the shadow lemma} we prove our version of the Shadow Lemma:

\medskip\noindent
{\bf The Shadow Lemma} (see Theorem \ref{thm:shadow lemma}) {\em Suppose $\Gamma \subset \mathsf{Homeo}(M)$ is a non-elementary convergence group, 
$\sigma \colon \Gamma \times M \rightarrow \Rb$ is an expanding coarse-cocycle, and $\mu$ is a coarse $\sigma$-Patterson--Sullivan measure of dimension $\delta$.
For any sufficiently small $\epsilon>0$ there exists $C=C(\epsilon) >1$ such that 
for all $\gamma \in \Gamma$,
$$
\frac{1}{C} e^{-\delta \norm{\gamma}_{\sigma}} \leq \mu\left( \mc S_\epsilon(\gamma) \right) \leq C e^{-\delta \norm{\gamma}_{\sigma}}.
$$
}

\medskip

We then establish some standard consequences of the Shadow Lemma in our setting.

\medskip\noindent
{\bf Proposition \ref{prop:some consequences of the shadow lemma}.}
{\em Suppose $\Gamma \subset \mathsf{Homeo}(M)$ is a non-elementary convergence group, $\sigma \colon \Gamma \times M \rightarrow \Rb$ is an expanding coarse-cocycle, and $\mu$ is a coarse $\sigma$-Patterson--Sullivan measure of dimension $\beta$. Then:
\begin{enumerate} 
\item $\beta \geq \delta_\sigma(\Gamma)>0$.
\item If $y \in M$ is a conical limit point, then $\mu(\{y\}) = 0$.
\item If 
$$
\sum_{\gamma \in \Gamma} e^{-\beta \norm{\gamma}_\sigma} < + \infty,
$$
(e.g.\  if  $\beta > \delta_\sigma(\Gamma)$)
then $\mu(\Lambda^{\rm con}(\Gamma))= 0$. 
\item There exists $C>0$ such that 
 $$\#\{\gamma\in\Gamma: \norm\gamma_\sigma\leq R\} \leq Ce^{\delta_\sigma(\Gamma) R}$$
 for any $R>0$.
\end{enumerate} 
}

\medskip

In the second part of the paper, we use the framework developed in the first part to study the ergodicity properties of Patterson--Sullivan measures. 
In Sections \ref{sec:conical} and \ref{sec:uniqueness} we prove Theorem \ref{uniqueness and ergodicity}. In Section~\ref{sec: action on M2} we study the action of $\Gamma$ on $M^{(2)}$.  In Section \ref{sec:flowspace} we introduce a flow space
which admits a measurable action of $\Gamma$. In Section \ref{sec:double ergo} we use this flow space to establish ergodicity of the action of $\Gamma$ on $(M^{(2)}, \nu)$ in Theorem \ref{our dichotomy}. Finally, in Section~\ref{sec:proof of dichotomy} we complete the proof of  Theorem~\ref{our dichotomy}. 

The constructions and arguments in Sections \ref{sec:flowspace} and \ref{sec:double ergo} use ideas from the work of Bader--Furman \cite{BF2017}.

For continuous GPS systems (i.e.\ when $\kappa=0$ and $G$ is continuous in Definition~\ref{def:GPS}), there is a well-defined continuous flow space $\psi^t \colon U_\Gamma \to U_\Gamma$ and when the Poincar\'e series diverges at its critical exponent, there is a  unique Bowen--Margulis--Sullivan measure (see Section~\ref{sec:kappa=0 case} for details). The arguments establishing Theorem \ref{our dichotomy} 
show that the flow is conservative and ergodic in this case.

\medskip\noindent
{\bf Theorem \ref{continuous ergodicity}.} {\em
 If $(\sigma, \bar\sigma, G)$ is a continuous GPS system with $\delta:=\delta_\sigma(\Gamma) < +\infty$ and 
$$
\sum_{\gamma \in \Gamma} e^{-\delta \norm{\gamma}_\sigma} = + \infty,
$$
then the flow $\psi^t$ on $(U_\Gamma, m_\Gamma)$ is conservative and ergodic.}

\medskip

In the third part of the paper, we consider applications of our ergodicity results and examine more deeply relations between expanding cocycles and GPS systems. 

In Section~\ref{sec:rigidity of PS measures} we establish Proposition~\ref{prop:rigidity of PS measures intro}, and in Section~\ref{sec:convexity of critical exponent} we establish Theorem \ref{thm:convexity of entropy in intro}.

The results in the next two sections partly answer the question of whether every expanding (coarse-)cocycle is part of a (coarse) GPS system, in addition to describing a systematic way to find expanding cocycles.
In Section~\ref{sec:symmetric cocycles} we define what it means for a coarse-cocycle to be coarsely-symmetric and prove that any expanding coarsely-symmetric coarse-cocycle is part of a GPS system. 
In Section~\ref{sec:potentials} we study the coarsely additive potentials introduced in Definition~\ref{defn:coarsely additive potential} above.  

Finally, in Appendix~\ref{appendix:conservative and dissipative} we define the notions of conservativity, dissipativity and Hopf decompositions for a general (unimodular) group action. We also prove that quotient measures exist when the action is dissipative, which is an essential point in our construction of a measurable flow space. See \cite{hopfdecompo} for a recent account of the Hopf Decomposition in a more general setting.

\subsection{Other approaches and related results}

In recent work Cantrell--Tanaka \cite{CantrellTanaka2021,CantrellTanaka_measures} study general cocycles on the Gromov boundary $\partial_\infty \Gamma$ of a word hyperbolic group. They show that if two cocycles have a corresponding Gromov product, then it is possible to use Patterson--Sullivan measures to build a $\Gamma$-invariant measure on $\partial^{(2)}\Gamma$ \cite[Prop.\ 2.8]{CantrellTanaka_measures} and prove ergodicity of the $\Gamma$ action~\cite[Thm.\ 3.1]{CantrellTanaka_measures}. They also consider a slightly more restrictive notion of the coarsely additive potentials introduced above (see Remark~\ref{remarke:CT work} above) and show that they give rise to coarse-cocycles satisfying these hypotheses. Our definition of GPS systems can be viewed as an extension of some of their ideas to general convergence groups.

A number of recent papers study Patterson--Sullivan theory for metric spaces where the group need not act as a convergence group on the boundary of the metric space. The most general of these investigations are perhaps independent works of Coulon~\cite{CoulonPS,Coulon2023} and Yang~\cite{Yang} which consider the case when $X$ is a proper geodesic metric space and $\Gamma$ is a group acting properly on $X$ by isometries with a contracting element. In this case the boundary is the horoboundary of $X$ and the cocycle is Busemann cocycle. The group action on this boundary may not be a convergence group action, but satisfies certain contracting properties. 

In many ways our approach is orthogonal to Coulon and Yang's. In our approach, we start with a convergence group action and find large classes of cocycles that are amenable to Patterson--Sullivan theory. In Coulon and Yang's approach, one studies large classes of metric spaces where the Busemann cocycle is amenable to Patterson--Sullivan theory. It would also be interesting to develop a uniform framework which contains both theories.


\part{Foundations}

\section{Convergence groups}\label{sec:compactification}

When $M$ is a compact metrizable space, a subgroup $\Gamma \subset \mathsf{Homeo}(M)$ is called a  (discrete) \emph{convergence group} if for every sequence 
$\{\gamma_n\}$ of distinct elements in $\Gamma$, there exist points $x,y \in M$ and a subsequence $\{\gamma_{n_j}\}$  such that $\gamma_{n_j}|_{M \smallsetminus \{y\}}$ converges locally uniformly to $x$. 
This notion was first introduced in \cite{GM87}, and then further developed in \cite{Tukia1994,Bowditch99}.

Given a convergence group, we define the following:
\begin{enumerate}
\item \label{item:limit set} The \emph{limit set} $\Lambda(\Gamma)$ is the set of points $x \in M$ where there exist $y \in M$ and a sequence $\{\gamma_n\}$ in $\Gamma$ so that $\gamma_n|_{M \smallsetminus \{y\}}$ converges locally uniformly to $x$.
\item A point $x \in \Lambda(\Gamma)$ is a \emph{conical limit point} if there exist distinct points $a,b \in M$ and a sequence of elements $\{\gamma_n\}$ in $\Gamma$ where $\lim_{n \to \infty} \gamma_n(x) = a$ and $\lim_{n \rightarrow \infty} \gamma_n(y) = b$ for all $y \in M \smallsetminus \{x\}$. 
\end{enumerate}
We say that a convergence group $\Gamma$ is \emph{non-elementary} if $\Lambda(\Gamma)$ contains at least 3 points.
  In this case $\Lambda(\Gamma)$ is the smallest $\Gamma$-invariant closed subset of $M$ (see \cite[Th.\,2S]{Tukia1994}).
{\bf For the remainder of our paper, we will assume all of our convergence groups are non-elementary.}

The elements in a convergence group can be characterized as follows.

\begin{fact}[{\cite[Th.\,2B]{Tukia1994}}]
\label{fact:classif loxparell}
Suppose $\Gamma \subset \mathsf{Homeo}(M)$ is a convergence group, then every element $\gamma\in \Gamma$ is either
 \begin{itemize}
  \item \emph{loxodromic}: it has two fixed points $\gamma^+$ and $\gamma^-$ in the limit set $\Lambda(\Gamma)\subset M$ such that $\gamma^{\pm n}|_{M \smallsetminus \{\gamma^{\mp}\}}$ converges locally uniformly to $\gamma^{\pm}$, 
  \item \emph{parabolic}: it has one fixed point $p\in \Lambda(\Gamma)$ such that $\gamma^{\pm n}|_{M \smallsetminus \{p\}}$ converges locally uniformly to $p$, or 
  \item \emph{elliptic}: it has finite order.
 \end{itemize}

\end{fact}

We next observe that $\Gamma \sqcup M$ admits a metrizable compact topology. 
This topology plays a similar role in our work as the topology
on the union of a transverse group and its limit set did in \cite{CZZ2023a}. 
Our argument is similar to a construction of Bowditch \cite[p.\,4 \& Prop.\,1.8]{Bowditch99} 
which produces a natural compact topology on $M^{(3)}\sqcup M$, where $M^{(3)} \subset M^3$ is the space of ordered triples of pairwise distinct elements of $M$, by viewing $M^{(3)}\sqcup M$ as a quotient of $M^3$. 

\begin{definition}\label{defn:compactification} Given a convergence group $\Gamma \subset \mathsf{Homeo}(M)$, a \emph{compactifying topology} on $\Gamma \sqcup M$ is a topology such that:
\begin{itemize} 
\item $\Gamma \sqcup M$ is a compact metrizable space.
\item The inclusions $\Gamma \hookrightarrow \Gamma \sqcup M$ and $M \hookrightarrow \Gamma \sqcup M$ are embeddings (where in the first embedding $\Gamma$ has the discrete topology). 
\item $\Gamma$ acts as a convergence group on $\Gamma \sqcup M$. 
\end{itemize} 
A metric $\dist$ on $\Gamma \sqcup M$ is called \emph{compatible} if it induces a compactifying topology.
\end{definition}

\begin{proposition}\label{prop:compactifying} If $\Gamma \subset \mathsf{Homeo}(M)$ is a convergence group, then there exists a unique compactifying topology.
Moreover, with respect to this topology the following hold:
\begin{enumerate}
\item If $\{\gamma_n\} \subset \Gamma$ is a sequence where $\gamma_n\to a \in M$ and $\gamma_n^{-1}\to b\in M$, then $\gamma_{n}|_{M \smallsetminus \{b\}}$ converges locally uniformly to $a$. 
\item A sequence $\{\gamma_n\} \subset \Gamma$ converges to $a \in M$ if and only if 
for every subsequence $\{\gamma_{n_j}\}$ there exist $b \in M$ and a further subsequence $\{\gamma_{n_{j_k}}\}$ such that $\gamma_{n_{j_k}}|_{M \smallsetminus \{b\}}$ converges locally uniformly to $a$.

\item\label{item:contraction of shadows} For any compatible metric $\dist$ and any $\epsilon > 0$ there exists a finite set $F\subset\Gamma$ such that 
$$\gamma \left(M\smallsetminus B_\epsilon(\gamma^{-1})\right) \subset B_\epsilon(\gamma)
$$
for every $\gamma\in\Gamma\smallsetminus F$ (where $B_r(x)$ is the open ball of radius $r$ centered at $x$ with respect to $\dist$).

\item $\Gamma$ is open in $\Gamma \sqcup M$ and its closure is $\Gamma \sqcup \Lambda(\Gamma)$.
\end{enumerate}
\end{proposition}

\begin{proof}
We first show that $\Gamma \sqcup M$ has a compactifying topology.

Fix three distinct points $x_1,x_2,x_3\in M$. For any open set $U\subset M$, let $\Gamma_U\subset \Gamma$ be the set of $\gamma$ such that 
$\#(\{\gamma x_1,\gamma x_2,\gamma x_3\}\cap U)\geq2$. Fix a countable basis $\mathcal B$ of open sets of $M$. Let $\mathcal B'$ be the set 
of singletons of $\Gamma$ and subsets of $\Gamma \sqcup M$ of the form $\Gamma_U\cup U$ for some $U\in\mathcal B$. It is straightforward to 
check that the topology generated by $\mathcal B'$ is compact Hausdorff and second-countable. Hence,
it is metrizable by Urysohn's metrization theorem. It remains to show that $\Gamma$ acts on $\Gamma \sqcup M$ as a convergence group.

Suppose $\{\gamma_n\} \subset \Gamma$ is a sequence of distinct elements. Then  there exist points $a,b \in M$ and a subsequence $\{\gamma_{n_j}\}$  such that $\gamma_{n_j}|_{M \smallsetminus \{b\}}$ converges locally uniformly to $a$. We claim that $\gamma_{n_j}|_{\Gamma \sqcup  M \smallsetminus \{b\}}$ converges locally uniformly to $a$. Suppose not. Then after passing to a subsequence we can find $\{z_j\} \subset \Gamma \sqcup M$ where $z_j \rightarrow z \neq b$ and $\gamma_{n_j}(z_j) \rightarrow c \neq a$.  Passing to a further subsequence we can consider two cases: 

\medskip

\noindent \emph{Case 1:} Assume $\{z_j\} \subset M$. Then by the choice of $\{\gamma_{n_j}\}$ we have $\gamma_{n_j}(z_j) \rightarrow a$, which contradicts our assumptions. 

\medskip 

\noindent \emph{Case 2:} Assume $\{z_j\}\subset \Gamma$. First suppose that $z \in \Gamma$, then passing to a subsequence we can suppose that $z_j = z$ for all $j$. At least two $zx_1$, $zx_2$, $zx_3$ do not equal $b$. So after relabelling we can suppose that $zx_1 \neq b$ and $zx_2 \neq b$. Then $(\gamma_{n_j}z_j)(x_1) \rightarrow a$ and $(\gamma_{n_j}z_j)(x_2) \rightarrow a$. So by the definition of the topology  $\gamma_{n_j}z_j \rightarrow a$. So we have a contradiction. 

Next suppose that $z \in M$. Then by definition of the topology and passing to a subsequence we can assume that $z_j|_{M \smallsetminus \{b'\}}$  converges locally uniformly to $z$. Since $z \neq b$, then $(\gamma_{n_j}z_j)|_{M \smallsetminus \{b'\}}$ converges locally uniformly to $a$, which implies that $\gamma_{n_j}z_j \rightarrow a$. So we have a contradiction. 

Thus $\Gamma$ acts on $\Gamma \sqcup M$ as a convergence group and hence $\Gamma \sqcup M$ has a compactifying topology. 

Next we consider $\Gamma \sqcup M$ with some compactifying topology and prove the assertions in the ``moreover'' part of the proposition. Notice that part (2) will imply that there is a unique compactifying topology. Let $\dist$ be a metric which induces this topology.

(1) Assume $\gamma_n \rightarrow a$ and $\gamma_n^{-1} \rightarrow b$. Suppose for a contradiction that $\gamma_{n}|_{M \smallsetminus \{b\}}$ does not converge locally uniformly to $a$. Then after passing to a subsequence  there exist $\epsilon > 0$ and $\{c_n\} \subset M \smallsetminus B(b,\epsilon)$ such that $\{\gamma_n(c_n)\} \subset M \smallsetminus B(a, \epsilon)$. Since $\Gamma$ acts as a convergence group on $\Gamma \sqcup M$, passing to a further subsequence we can suppose that $\gamma_{n}|_{\Gamma \sqcup M \smallsetminus \{b'\}}$ converges locally uniformly to $a'$ for some $a',b' \in \Gamma \sqcup M$. Since $\Gamma$ acts by homeomorphisms on $M$, we must have $b' \in M$ (otherwise when $n$ is large $\gamma_n|_M$ would not map onto $M$). So 
$$
a = \lim_{n \rightarrow \infty} \gamma_n=\lim_{n \rightarrow \infty} \gamma_n(\id) = a'.
$$
Also notice that $\gamma_n(\gamma_n^{-1}) =\id$ for all $n$ and so we must have $b=b'$. Then $\gamma_n(c_n) \rightarrow a$ and we have a contradiction.

(2) $(\Rightarrow)$: Suppose $\gamma_n \rightarrow a$ and fix a subsequence $\{\gamma_{n_j}\}$. Since $\Gamma \sqcup M$ is compact, there exists a subsequence with $\gamma_{n_{j_k}}^{-1} \rightarrow b$. Then by (1), $\gamma_{n_{j_k}}|_{M \smallsetminus \{b\}}$ converges locally uniformly to $a$. 

$(\Leftarrow)$: Suppose $a \in M$ and $\{\gamma_n\} \subset \Gamma$ has the property that for every subsequence $\{\gamma_{n_j}\}$ there exist $b \in M$ and a further subsequence $\{\gamma_{n_{j_k}}\}$ such that $\gamma_{n_{j_k}}|_{M \smallsetminus \{b\}}$ converges locally uniformly to $a$. Since $\Gamma \sqcup M$ is compact, to show that $\gamma_n$ converges to $a$ it suffices to show that every convergent subsequence converges to $a$. So suppose that $\gamma_{n_j} \rightarrow a'$. Passing to a subsequence we can suppose that $\gamma_{n_j}^{-1} \rightarrow b$. Then by (1), $\gamma_{n_{j}}|_{M \smallsetminus \{b\}}$ converges locally uniformly to $a'$. So by hypothesis, $a = a'$.

(3) Fix $\epsilon > 0$ and suppose not. Then there exist a sequence $\{\gamma_n\}$ of distinct elements and a sequence $\{ x_n\} \subset \Gamma \sqcup M$ such that 
$$
\dist(\gamma_n(x_n), \gamma_n) \geq \epsilon \quad \text{and} \quad \dist(x_n, \gamma_n^{-1}) \geq \epsilon.
$$
Passing to a subsequence, we can suppose that $\gamma_n \rightarrow a \in M$ and $\gamma_n^{-1} \rightarrow b \in M$. Then by (1), $\gamma_{n}|_{M \smallsetminus \{b\}}$ converges locally uniformly to $a$. Since  
$$
\lim_{n \rightarrow \infty} \dist(x_n, b) = \lim_{n \rightarrow \infty} \dist(x_n, \gamma_n^{-1}) \geq \epsilon, 
$$
then $\gamma_n(x_n) \rightarrow a$. So 
$$
\epsilon \leq \lim_{n \rightarrow \infty} \dist(\gamma_n(x_n), \gamma_n)= \dist(a,a) = 0
$$
and we have a contradiction.

(4) Since $M$ is compact, it must be closed in $\Gamma \sqcup M$. Hence $\Gamma$ must be open. Part (2) implies that the closure of $\Gamma$ in $\Gamma \sqcup M$ is $\Gamma \sqcup \Lambda(\Gamma)$. 
\end{proof} 

The following lemma will allow us to relate general elements to loxodromic ones and can be viewed as a convergence group action version of the fact (see \cite{AMS}) that elements in a strongly irreducible linear group are uniformly close to well-behaved proximal elements.

\begin{lemma}\label{lem:AMS_Gromov_hyp} Suppose $\Gamma \subset \mathsf{Homeo}(M)$ is a convergence group and $\dist$ is a compatible metric on $\Gamma \sqcup M$. Then there exist $\epsilon > 0$ and a finite set $F \subset \Gamma$ with the following property: for any $\gamma \in \Gamma$ there is some $f \in F$ where $\gamma f$ is loxodromic and 
$$
\min\left\{ \dist( (\gamma f)^+, (\gamma f)^-), \, \dist( \gamma f, (\gamma f)^-), \,  \dist( (\gamma f)^+, (\gamma f)^{-1}) \right\} > \epsilon.
$$
\end{lemma} 

To prove the lemma we use the following result of Tukia. 

\begin{lemma}[{\cite[Cor.\ 2E]{Tukia1994}}]\label{lem:char of loxodromic} Suppose $\Gamma \subset \mathsf{Homeo}(M)$ is a convergence group.
If $\{\gamma_n\}$ is a sequence of distinct elements, $\gamma_n\rightarrow a$, $\gamma_n^{-1} \rightarrow b$, and $a \neq b$, then for $n$ sufficiently large $\gamma_n$ is loxodromic and $\gamma_n^+\rightarrow a$, $\gamma_n^- \rightarrow b$. 
\end{lemma}

\begin{lemma}\label{lem:AMStech}
Suppose $\Gamma \subset \mathsf{Homeo}(M)$ is a convergence group and $\dist$ is a compatible metric on $\Gamma \sqcup M$.
Then for all $x,y\in \Lambda(\Gamma)$ and $\epsilon>0$ there is $\gamma \in\Gamma$ such that $\gamma(M-B_\epsilon(x))\subset B_\epsilon(y)$.
\end{lemma}
\begin{proof}
Fix an escaping sequence $\{\gamma_n\}\subset\Gamma$.
Passing to a subsequence we may assume  $\gamma_n|_{M \smallsetminus \{a\}}$ converges locally uniformly to $b$ for some $a,b\in\Lambda(\Gamma)$.
Since $\Gamma$ acts minimally on $\Lambda(\Gamma)$ there are $\alpha,\beta\in\Gamma$ such that $\dist (\alpha a,x)<\epsilon/2$ and $\dist(\beta b,y)<\epsilon/2$.
Then $\beta\gamma_n\alpha^{-1}|_{M \smallsetminus \{\alpha a\}}$  converges locally uniformly to $\beta b$, so for $n$ large enough we have 
\[\beta\gamma_n\alpha^{-1}(M-B_\epsilon(x))\subset \beta\gamma_n\alpha^{-1}(M-B_{\epsilon/2}(\alpha^{-1}a))\subset B_{\epsilon/2}(\beta b)\subset B_\epsilon(y).\qedhere \]
\end{proof}

\begin{proof}[Proof of Lemma~\ref{lem:AMS_Gromov_hyp}] Fix four distinct points $x_1, x_2, x_3,x_4 \in M$ in the limit set of $\Gamma$. Let 
$$
\epsilon := \frac{1}{4} \min_{1 \leq i < j \leq 4} \dist(x_i, x_j).
$$
By Lemma~\ref{lem:AMStech}, for every $i\ne j\in\{1,2,3,4\}$ we can find an element $g_{i,j} \in \Gamma$ such that 
$$
g_{i,j}\Big( M \smallsetminus B_\epsilon\left(x_j\right)\Big) \subset B_\epsilon\left(x_i\right) \quad \text{and} \quad g_{i,j}^{-1}\Big( M \smallsetminus B_\epsilon\left(x_i\right)\Big) \subset B_\epsilon\left(x_j\right).
$$

We claim that there exists a finite set $F_0 \subset \Gamma$ such that: if $\gamma \in \Gamma \smallsetminus F_0$, then there exist $i\ne j\in\{1,2,3,4\}$ 
such that $\gamma g_{i,j}$ is loxodromic and 
$$
\min\left\{ \dist( (\gamma g_{i,j})^+, (\gamma g_{i,j})^-) , \, \dist( \gamma g_{i,j}, (\gamma g_{i,j})^-), \, \dist( (\gamma g_{i,j})^+, (\gamma g_{i,j})^{-1})\right\}> \epsilon. 
$$
Suppose not. Then there exists an escaping sequence $\{\gamma_n\}$ in $\Gamma$ where each $\gamma_n$ does not have this property. 
Passing to a subsequence we can suppose that there is $a,b \in M$ such that $\gamma_n \to a$ and $\gamma_n^{-1} \to b$.

Since the balls $\{ B_{2\epsilon}(x_i)\}_{1 \leq i \leq 4}$ are pairwise disjoint we can pick $i\ne j\in\{1,2,3,4\}$ such that 
$a,b \notin B_{2\epsilon}(x_i) \cup B_{2\epsilon}(x_j)$.  Then $\gamma_n g_{i,j} \to a$ and 
\hbox{$(\gamma_n g_{i,j})^{-1} \rightarrow g_{i,j}^{-1}b\in B_\epsilon(x_j)$.} Then, by our choice of $i$ and $j$,
$$
\dist( a, g_{i,j}^{-1}(b)) > \epsilon.
$$
Thus Lemma~\ref{lem:char of loxodromic} implies that $ \gamma_n g_{i,j}$ is loxodromic for $n$ sufficiently large. Further, $( \gamma_n g_{i,j})^+ \to a$ and $(\gamma_n g_{i,j} )^- \to g_{i,j}^{-1}(b)$. So
for $n$ sufficiently large we have 
$$
\min\left\{ \dist( (\gamma_n g_{i,j})^+, (\gamma_n g_{i,j})^-) , \, \dist( \gamma_n g_{i,j}, (\gamma g_{i,j})^-), \, \dist( (\gamma_n g_{i,j})^+, (\gamma_n g_{i,j})^{-1})\right\}> \epsilon. 
$$
Thus we have a contradiction. Thus there exists a finite set $F_0 \subset \Gamma$ with the desired property. 

Now fix a loxodromic element $h$ with 
$$
\min\left\{ \dist(h^+, h^-), \,  \dist(h, h^-), \,  \dist(h^+, h^{-1}) \right\} > \epsilon.
$$ 
Then the set 
$$
F :=\big\{ g_{i,j} : i\ne j\in\{1,2,3,4\} \big\} \cup \{ f^{-1}h : f \in F_0\}
$$
satisfies the lemma. 
\end{proof}

\section{Cocycles and GPS systems}\label{sec:expanding cocycles}

In this subsection, we record basic properties of coarse-cocycles and GPS systems. We begin with a few simple properties shared by all coarse-cocycles.

\begin{observation}\label{obs:triangle inequality}
Suppose $\Gamma \subset \mathsf{Homeo}(M)$ is a convergence group and $\sigma$ is a $\kappa$-coarse-cocycle. Then:
\begin{enumerate}
\item\label{item:identity}  $\abs{\sigma(\id,x)} \leq \kappa$ for any $x\in M$. 
\item\label{item:inverses} If $\gamma\in \Gamma$ and $x \in M$, then 
$$
\abs{ \sigma(\gamma, \gamma^{-1}x) + \sigma(\gamma^{-1},x)} \leq 2 \kappa.
$$
\end{enumerate}
\end{observation}

\begin{proof} For part (1), notice that 
 \begin{equation*}
 \abs{\sigma(\id,x)}=\abs{\sigma(\id^2,x)-(\sigma(\id,\id(x))+\sigma(\id,x))}\leq \kappa. 
 \end{equation*}
 Part (2) follows from part (1) and the fact that 
\begin{align*}
\abs{ \sigma(\gamma, \gamma^{-1}x) + \sigma(\gamma^{-1},x) - \sigma(\id,x)} & \leq \kappa. \qedhere
\end{align*}
\end{proof} 

In the majority of our work we will require that our coarse-cocycles are expanding, see Definition~\ref{defn:expanding}. We start by observing the following alternative description of these cocycles. 

\begin{proposition}\label{prop:different definition of expanding}  Suppose $\Gamma \subset \mathsf{Homeo}(M)$ is a convergence group, $\dist$ is a compatible metric on $\Gamma \sqcup M$, and $\sigma$ is a $\kappa$-coarse-cocycle. If there exists a function $h \colon \Gamma \to \Rb$ such that 
\begin{itemize}
\item for every $\epsilon > 0$ there exists $C > 0$ such that
$$
h(\gamma) - C \leq \sigma(\gamma, x) \leq h(\gamma)+C
$$ 
whenever $x \in M$ and $\gamma \in \Gamma$ and $\dist(x,\gamma^{-1}) > \epsilon$, and
\item $\lim_{n \to \infty} h(\gamma_n) = + \infty$ for every escaping sequence $\{ \gamma_n\} \subset \Gamma$, 
\end{itemize}
then $\sigma$ is expanding and we may choose 
$$
\norm{\gamma}_\sigma = h(\gamma). 
$$
\end{proposition} 

\begin{proof} We just have to show that $\sigma$ is proper. So fix a sequence $\{\gamma_n\} \subset \Gamma$ of distinct loxodromic elements where the sequence of pairs of repelling/attracting points satisfies  $\liminf_{n \rightarrow \infty} \dist(\gamma_n^-, \gamma_n^+) > 0$. Then 
$$
\lim_{n \to \infty} \dist(\gamma_n^{-1}, \gamma_n^-) = \lim_{n \to \infty} \dist(\gamma_n, \gamma_n^+)=0
$$
and so by the first property there exists $C > 0$ such that for all $n \geq 1$, we have
$$
h(\gamma_n) - C \leq \sigma(\gamma_n, \gamma_n^+) 
$$ 
Then the second property implies $\sigma(\gamma_n, \gamma_n^+) \to +\infty$, so $\sigma$ is proper. 
\end{proof}

The next result establishes a number of useful properties for expanding cocycles.

\begin{proposition}\label{prop:basic properties}  Suppose $\Gamma \subset \mathsf{Homeo}(M)$ is a convergence group, $\dist$ is a compatible metric on $\Gamma \sqcup M$, and $\sigma$ is an expanding $\kappa$-coarse-cocycle.  Then:
\begin{enumerate}
\item\label{item:tri inequality} For any finite subset $F \subset \Gamma$, there exists $C > 0$ such that: if $\gamma \in \Gamma$ and $f \in F$, then 
$$
\norm{\gamma}_\sigma -C \leq \norm{\gamma f}_\sigma \leq \norm{\gamma}_\sigma + C \quad \text{and} \quad \norm{\gamma}_\sigma -C \leq \norm{f\gamma}_\sigma \leq \norm{\gamma}_\sigma + C.
$$
\item\label{item:properness} $\lim_{n \rightarrow \infty} \norm{\gamma_n}_\sigma = + \infty$ for every escaping sequence $\{ \gamma_n\} \subset \Gamma$. 
\item\label{item:loxodromic periods} 
If $\gamma\in\Gamma$ is loxodromic, then 
$$
-\kappa + \limsup_{n \rightarrow \infty}  \frac{1}{n} \norm{\gamma^n}_\sigma  \leq \sigma(\gamma,\gamma^+) \leq  \kappa+\liminf_{n \rightarrow \infty}  \frac{1}{n} \norm{\gamma^n}_\sigma
$$
and
$$
-\kappa < \sigma(\gamma, \gamma^+).
$$
\item\label{item:parabolic periods}  If $\gamma\in\Gamma$ is parabolic with fixed point $p \in M$, then
$$-2\kappa\le\sigma(\gamma,p)\leq 4\kappa.$$
\item\label{item:a technical fact}   If $\{\gamma_n\} \subset \Gamma$ is an escaping sequence, $\{y_n\} \subset M$ and $\{ \sigma(\gamma_n, y_n) \}$ is bounded below, then 
$$
\lim_{n \rightarrow \infty} \dist(\gamma_n y_n, \gamma_n) =0.
$$

\item\label{item:multiplicative estimate} For any $\epsilon > 0$ there exists $C > 0$ such that: if $\alpha, \beta \in \Gamma$ and $\dist(\alpha^{-1}, \beta) \geq \epsilon$, then 
$$
\norm{\alpha}_\sigma+\norm{\beta}_\sigma-C \leq \norm{\alpha \beta}_\sigma \leq \norm{\alpha}_\sigma+\norm{\beta}_\sigma+C.
$$

\item\label{item:estimate on difference of transformations} For any $\epsilon>0$ there exists a finite subset $F\subset \Gamma$ such that: if  $\alpha,\beta\in\Gamma$, $\norm\alpha_\sigma\leq \norm\beta_\sigma$ and $\beta^{-1}\alpha\not\in F$, then
$$
\dist (\beta^{-1},\beta^{-1}\alpha)\leq \epsilon.
$$

\end{enumerate}

\end{proposition} 

\begin{proof}[Proof of~\eqref{item:tri inequality}] For every $\epsilon > 0$ fix $C_\epsilon > 0$ such that: if $x \in M$, $\gamma \in \Gamma$ and $\dist(x,\gamma^{-1}) > \epsilon$, then 
$$
\norm{\gamma}_\sigma - C_\epsilon \leq \sigma(\gamma, x) \leq \norm{\gamma}_\sigma+C_\epsilon.
$$ 

Fix $\epsilon > 0$ such that $M \not\subset B_{\epsilon}(p)$ for any $p \in \Gamma \sqcup M$. 
Since $F$ is finite, there exists $\epsilon' > 0$ such that: if $\dist(p,q) > \epsilon$ and $f \in F$, then $\dist(fp, fq) > \epsilon'$.
Using the fact that $\Gamma$ acts as a convergence group on $\Gamma\sqcup M$ (Definition~\ref{defn:compactification}), one can find a finite subset $F' \subset \Gamma$ such that: if $f \in F$ and $\gamma \in \Gamma \smallsetminus F'$, then $\dist(\gamma^{-1},\gamma^{-1} f^{-1}) \leq \epsilon/2$. Let
$$
A : = \sup\{ \abs{\sigma(\gamma, x)} : x \in M,\, \gamma \in F \cup F' \cup F\cdot F'\}<\infty.
$$

Now fix $\gamma \in \Gamma$ and $f \in F$. For the first set of inequalities, fix $y \in M$ with $\dist(y, f^{-1}\gamma^{-1}) >\epsilon$. Then $\dist(fy, \gamma^{-1}) > \epsilon'$. So 
\begin{align*}
\abs{\norm{\gamma f}_\sigma-\norm{\gamma}_\sigma} & \leq C_\epsilon + C_{\epsilon'}+\abs{\sigma(\gamma f, y)-\sigma(\gamma, fy)} \leq C_\epsilon + C_{\epsilon'}+\kappa + \abs{\sigma(f, y)}  \\
& \leq C_\epsilon + C_{\epsilon'}+\kappa + A.
\end{align*} 
For the second set of inequalities, fix $z \in M$ with $\dist(z, \gamma^{-1}) > \epsilon$. If $\gamma \notin F'$, then $\dist(\gamma^{-1} f^{-1}, z) > \epsilon/2$ and so  
\begin{align*}
\abs{\norm{f\gamma}_\sigma-\norm{\gamma}_\sigma} & \leq C_{\epsilon/2} + C_{\epsilon}+\abs{\sigma(f\gamma, z)-\sigma(\gamma, z)} \leq C_{\epsilon/2} + C_{\epsilon}+\kappa + \abs{\sigma(f, \gamma z)}  \\
& \leq C_{\epsilon/2} + C_{\epsilon}+\kappa + A.
\end{align*} 
If $\gamma \in F'$, then 
\begin{align*}
\abs{\norm{f\gamma}_\sigma-\norm{\gamma}_\sigma} & \leq 2C_\epsilon + 2A. \qedhere
\end{align*} 
\end{proof} 

\begin{proof}[Proof of~\eqref{item:properness}] Fix a finite set $F \subset \Gamma$ and $\epsilon > 0$ as in Lemma~\ref{lem:AMS_Gromov_hyp}. For each $n \geq 1$ fix $f_n \in F$ such that 
$$
\dist( (\gamma_n f_n)^+, (\gamma_n f_n)^{-}) > \epsilon \quad \text{and} \quad \dist( (\gamma_n f_n)^+, (\gamma_n f_n)^{-1}) > \epsilon. 
$$
Then the expanding property implies that  there exists $C > 0$ such that 
$$
\sigma(\gamma_n f_n, (\gamma_n f_n)^+) \leq \norm{\gamma_n f_n}_\sigma + C. 
$$
Since $\sigma$ is proper, then $\norm{\gamma_n f_n}_\sigma \rightarrow + \infty$. So by part~\eqref{item:tri inequality}, $\norm{\gamma_n}_\sigma \rightarrow + \infty$. 
\end{proof} 

\begin{proof}[Proof of~\eqref{item:loxodromic periods}] 

Suppose $\gamma\in \Gamma$ is loxodromic.
Since $\gamma^{-n}\to\gamma^-$ as $n\to\infty$ and $\gamma^+\ne \gamma^-$, the expanding property implies that  there exists $C > 0$ such that 
$$
\norm{\gamma^n}_\sigma - C \leq \sigma(\gamma^n, \gamma^+) \leq \norm{\gamma^n}_\sigma+C
$$
for all $n \geq 1$. Since $\sigma$ is a $\kappa$-coarse-cocycle,
$$
n\sigma(\gamma, \gamma^+) - (n-1)\kappa \leq \sigma(\gamma^n, \gamma^+) \leq n\sigma(\gamma, \gamma^+) + (n-1)\kappa.
$$
Combining the two estimates and sending $n$ to infinity yields the first set of inequalities.

By part \eqref{item:properness}, there exists $N \geq 1$ such that $\norm{\gamma^N}_\sigma > C$. Then 
\begin{equation*}
\sigma(\gamma, \gamma^+)\geq \frac{1}{N} \Big(\sigma(\gamma^N, \gamma^+)-(N-1)\kappa\Big) \geq \frac{1}{N} \left(  \norm{\gamma^N}_\sigma-C \right) -\kappa> -\kappa. \qedhere
\end{equation*}
\end{proof}

\begin{proof}[Proof of~\eqref{item:parabolic periods}] 
Suppose $\gamma\in\Gamma$ is parabolic with fixed point $p\in M$.
Fix $y\in M \smallsetminus \{p\}$.
Since $\gamma^{\pm n}\to p$ when $n\to\infty$, by the expanding property there exists $C>0$ such that 
$$
\sigma(\gamma^{\pm n}, y) \geq \norm{\gamma^{\pm n}}_\sigma-C
$$
for all $n\geq 1$. Since, by part ~\eqref{item:properness}, $\norm{\gamma^{\pm n}}_\sigma\to+\infty$, both $\sigma(\gamma^n,y)$ and $\sigma(\gamma^{-n},y)$ are nonnegative for $n$ large. Since
$$\left|\sigma(\gamma^{\pm m},y)-\left(\sigma(\gamma^{\pm 1},\gamma^{\pm (m-1)}y)+\sigma(\gamma^{\pm (m-1)},y)\right)\right|\le \kappa$$ 
for all $m\ge 2$, we see that
\begin{align*}
\limsup_{n \rightarrow \infty} & \frac{1}{n} \left( \sigma(\gamma^{\pm1}, \gamma^{\pm(n-1)}y)+ \sigma(\gamma^{\pm1}, \gamma^{\pm(n-2)}y)+\dots + \sigma(\gamma^{\pm1},y)\right) \\
& \geq \limsup_{n \rightarrow \infty} \frac{1}{n}\left( \sigma(\gamma^{\pm n},y) - (n-1)\kappa\right) \geq - \kappa. 
\end{align*} 
Since $\sigma$ is $\kappa$-coarsely continuous and $\gamma^{\pm n}y\to p$, we see that 
$$\limsup_{n \rightarrow \infty} \frac{1}{n} \left( \sigma(\gamma^{\pm1}, \gamma^{\pm(n-1)}y)+ \sigma(\gamma^{\pm1}, \gamma^{\pm(n-2)}y)+\dots + \sigma(\gamma^{\pm1},y) \right) \le\sigma(\gamma^{\pm1}, p)+\kappa.$$
Thus, $\sigma(\gamma^{\pm1},p)\ge -2\kappa.$

Finally, by the coarse-cocycle identity, see Observation~\ref{obs:triangle inequality}\eqref{item:inverses},
$$
\sigma(\gamma,p) + \sigma(\gamma^{-1},p)  \leq 2\kappa.
$$
Hence, $\sigma(\gamma^{\pm1},p) \leq 4\kappa$. 
\end{proof} 

\begin{proof}[Proof of~\eqref{item:a technical fact}] 
 We prove the contrapositive: if $\{\dist(\gamma_ny_n,\gamma_n)\}$ does not converge to $0$, then 
 $$
\liminf_{n \rightarrow \infty} \sigma(\gamma_n, y_n) =-\infty.
$$
Passing to a subsequence, we may suppose that there exists $\epsilon > 0$ such that 
$$
\dist(\gamma_{n} y_{n},\gamma_{n}) \geq \epsilon
$$
for all $n \geq 1$. Then by the expanding property, there exists $C > 0$ such that 
$$
\sigma(\gamma_n^{-1}, \gamma_n y_n) \geq \norm{\gamma_n^{-1}}_\sigma - C
$$
for all $n \geq 1$. In particular, by part~\eqref{item:properness}, we have
$$  \liminf_{n \rightarrow \infty} \sigma(\gamma_n^{-1}, \gamma_n y_n) = +\infty.
$$
Since 
$$\abs{\sigma(\gamma_n, y_n) -\sigma(\gamma_n^{-1}, \gamma_n y_n)} \leq 2\kappa$$
for all $n$ (see Observation~\ref{obs:triangle inequality}), this implies that 
$\liminf_{n \rightarrow \infty} \sigma(\gamma_n, y_n) = -\infty$.
\end{proof}

\begin{proof}[Proof of~\eqref{item:multiplicative estimate}] Suppose not. Then there exist sequences  $\{\alpha_n\}$ and $\{\beta_n\}$ in $\Gamma$ where $\dist(\alpha_n^{-1}, \beta_n) \geq \epsilon$ and 
$$
\lim_{n \rightarrow \infty} \abs{\norm{\alpha_n}_\sigma+\norm{\beta_n}_\sigma- \norm{\alpha_n \beta_n}_\sigma} = +\infty.
$$
If $\norm{\beta_n}_\sigma$ is bounded, then $\{\beta_n\}$ is a finite set by part~\eqref{item:properness}. 
Part~\eqref{item:tri inequality} then implies that $\norm{\alpha_n\beta_n}_\sigma-\norm{\alpha_n}_\sigma$
 is bounded,
and hence that  $\norm{\alpha_n}_\sigma+\norm{\beta_n}_\sigma- \norm{\alpha_n \beta_n}_\sigma$ is bounded. So $\{\beta_n\}$ must be an escaping sequence. Similarly, $\{\alpha_n\}$ must be an escaping sequence.

Fix $\epsilon_0> 0$ such that $M\not\subset B_{\epsilon_0}(p)\cup B_{\epsilon_0}(q)$ for  any $p,q\in \Gamma \sqcup M$. 
We may assume without loss of generality that $\epsilon_0\ge\epsilon$. Fix $C  > 0$ such that: if $x \in M$, $\gamma \in \Gamma$ and $\dist(x,\gamma^{-1}) > \epsilon$, then 
$$
\norm{\gamma}_\sigma - C \leq \sigma(\gamma, x) \leq \norm{\gamma}_\sigma+C.
$$ 

For each $n$, fix $x_n \in M \smallsetminus \left(  B_{\epsilon}(\beta_n^{-1})\cup B_{\epsilon}(\beta_n^{-1}\alpha_n^{-1})\right)$. Then 
$$
\norm{\beta_n}_\sigma -C \leq \sigma(\beta_n, x_n) \leq \norm{\beta_n}_\sigma +C
$$
and
$$
\norm{\alpha_n\beta_n}_\sigma -C \leq  \sigma(\alpha_n\beta_n, x_n) \leq \norm{\alpha_n\beta_n}_\sigma +C.
$$
Since $\dist(x_n, \beta_n^{-1}) > \epsilon$,  the convergence group property implies that $\dist(\beta_n x_n,\beta_n) \rightarrow 0$. So 
$$
\liminf_{n \rightarrow \infty} \dist(\alpha_n^{-1}, \beta_n x_n) \geq \epsilon. 
$$
So by the expansion property, there exists $C^\prime > 0$ such that 
$$
\norm{\alpha_n}_\sigma - C^\prime \leq \sigma(\alpha_n, \beta_n x_n) \leq \norm{\alpha_n}_\sigma + C^\prime
$$
for all $n \geq 1$. But then 
\begin{align*}
\abs{\norm{\alpha_n}_\sigma+\norm{\beta_n}_\sigma- \norm{\alpha_n \beta_n}_\sigma}& \leq 2C+C' + \abs{ \sigma(\alpha_n, \beta_n x_n)+\sigma(\beta_n, x_n)-\sigma(\alpha_n \beta_n, x_n)} \\
& \leq 2C+C' + \kappa
\end{align*}
and we have a contradiction. 
\end{proof} 

\begin{proof}[Proof of~\eqref{item:estimate on difference of transformations}] 
By  Part~\eqref{item:multiplicative estimate} there exists $C > 0$ such that: if $\alpha, \beta \in \Gamma$ and $\dist(\alpha^{-1}, \beta) \geq \epsilon$, then 
$$
\norm{\alpha}_\sigma+\norm{\beta}_\sigma-C \leq \norm{\alpha \beta}_\sigma.
$$
Then let 
$$
F:=\{\gamma\in \Gamma:\norm\gamma_\sigma\leq C\}.
$$
Notice that $F$ is finite by  part~\eqref{item:properness}.

Now if  $\alpha,\beta\in\Gamma$, $\norm\alpha_\sigma\leq \norm\beta_\sigma$, and $\beta^{-1}\alpha \notin F$, then 
$$
\norm{\beta}_\sigma+\norm{\beta^{-1}\alpha}_\sigma-C > \norm{\alpha}_\sigma=\norm{\beta \beta^{-1}\alpha}_\sigma.
$$
So by our choice of $C$ we must have  $\dist (\beta^{-1},\beta^{-1}\alpha) < \epsilon$.
\end{proof}

Finally,  we observe the coarse-cocycles in a coarse GPS system are expanding.

\begin{proposition}\label{prop:GPS implies expanding} If $\Gamma \subset \mathsf{Homeo}(M)$ is a convergence group and $(\sigma, \bar{\sigma}, G)$ is a coarse GPS system, then:
\begin{enumerate} 
\item\label{item:where sigma is large} $\sigma$ and $\bar{\sigma}$ are expanding.

\item\label{item:norm vs dual norm} There exists $C > 0$ such that 
$$
\norm{\gamma^{-1}}_{\bar{\sigma}}-C \leq \norm{\gamma}_\sigma \leq \norm{\gamma^{-1}}_{\bar{\sigma}} + C
$$
for all $\gamma \in \Gamma$. 
\item\label{item:gromprod is what it should be}   If $G$ is coarsely continuous, then there exists $C'>0$ such that 
$$
\limsup_{\alpha \rightarrow a, \beta \rightarrow b}\abs{G(a,b)-\left(\norm{\alpha^{-1}}_\sigma+\norm\beta_\sigma-\norm{\alpha^{-1}\beta}_\sigma\right)} \leq C' 
$$
for all $(a,b)\in \Lambda(\Gamma)^{(2)}$. 

\end{enumerate} 

\end{proposition}

\begin{proof}[Proof of~\eqref{item:where sigma is large}] For each $\gamma \in \Gamma$ fix $y_\gamma \in M$ such that $\epsilon_0: = \inf_{\gamma \in \Gamma} \dist(y_\gamma, \gamma^{-1}) > 0$. Then define $\norm{\gamma}_\sigma : = \sigma(\gamma, y_\gamma)$.

By symmetry it suffices to show that $\sigma$ is expanding. Fix $\epsilon > 0$. We wish to find $C > 0$ such that 
$$
\norm{\gamma}_\sigma - C \leq \sigma(\gamma,x) \leq \norm{\gamma}_\sigma + C
$$
whenever $\dist(\gamma^{-1},x) > \epsilon$. 

Fix $\epsilon' \in (0,\min\{\epsilon,\epsilon_0\})$ such that $M\not\subset B_{\epsilon'}(p)\cup B_{\epsilon'}(q)$ for all $p,q\in \Gamma \sqcup M$. Let 
$$
C_1:=\sup\{ \abs{G(p,q)} : \dist(p,q) \geq \epsilon'/2\}.
$$
Notice that $C_1 < + \infty$ since $G$ is locally bounded. By Proposition~\ref{prop:compactifying} there exists a finite subset $F\subset \Gamma$ such that 
\begin{equation}\label{eqn:gamma inverse action in proof of expanding}
\gamma^{-1}\left(M- B_{\epsilon'/2}(\gamma)\right)\subset B_{\epsilon'/2}(\gamma^{-1}) \quad \text{and} \quad \gamma\left(M- B_{\epsilon'/2}(\gamma^{-1})\right)\subset B_{\epsilon'/2}(\gamma)
\end{equation}
for any $\gamma\in\Gamma \smallsetminus F$. 
Then let
$$
C_2:=\sup\left\{ \abs{\norm \gamma_\sigma - \sigma(\gamma,x)} :  \gamma\in F \text{ and } x\in M\right\}.
$$
We claim that $C : = \max\{C_2, 4C_1+2\kappa\}$ suffices. 

Fix $\gamma \in\Gamma$ and $x\in M-B_\epsilon(\gamma^{-1})$. If $\gamma \in F$, then 
$$
\norm{\gamma}_\sigma - C_2 \leq \sigma(\gamma,x) \leq \norm{\gamma}_\sigma + C_2.
$$
Otherwise, by the definition of $\epsilon'$ there exists $z'\in M- (B_{\epsilon'}(\gamma y_\gamma)\cup B_{\epsilon'}(\gamma))$. Then let $z := \gamma^{-1}(z^\prime)$. By Equation~\eqref{eqn:gamma inverse action in proof of expanding},
$$
z \in B_{\epsilon'/2}(\gamma^{-1}), \ \gamma x \in B_{\epsilon'/2}(\gamma), \ \gamma y_\gamma \in B_{\epsilon'/2}(\gamma).  
$$
Hence
\begin{align*}
\dist(\gamma z,\gamma x) & = \dist(z', \gamma x) \geq \dist(z', \gamma) - \dist(\gamma, \gamma x) > \epsilon'-\epsilon'/2 = \epsilon'/2, \\ 
\dist(z,x) & \geq \dist(\gamma^{-1},x) - \dist(z,\gamma^{-1}) > \epsilon - \epsilon^\prime/2 > \epsilon^\prime/2, \\
\dist(\gamma z, \gamma y_\gamma) &= \dist(z', \gamma y_\gamma) > \epsilon', \\
\dist(z, y_\gamma) &> \dist(\gamma^{-1}, y_\gamma) - \dist(z,\gamma^{-1}) > \epsilon_0 - \epsilon'/2 > \epsilon'/2.
\end{align*}
Then 
\begin{align*}
\abs{ \sigma(\gamma,x) - \norm \gamma_\sigma} & = \abs{\bar\sigma(\gamma,z) +\sigma(\gamma,x)  - ( \bar\sigma(\gamma,z)+\sigma(\gamma,y_\gamma))} \\
 & \leq \abs{G(\gamma z,\gamma x) - G(z,x) - G(\gamma z,\gamma y_\gamma) + G(z,y_\gamma)}+2\kappa \\
 &\leq 4C_1+2\kappa. \qedhere
\end{align*}

\end{proof}

\begin{proof}[Proof of~\eqref{item:norm vs dual norm}]Fix $\epsilon>0$ such that $M\not\subset B_{\epsilon}(x)\cup B_{\epsilon}(y)$ for all $x,y\in \Gamma \sqcup M$.  Since $\sigma$ and $\bar \sigma$ are expanding, there exists $C > 0$ such that 
$$
 \norm{\gamma}_\sigma-C \leq \sigma(\gamma, x) \leq \norm{\gamma}_\sigma+ C \quad \text{and} \quad \norm{\gamma}_{\bar\sigma}-C \leq \bar\sigma(\gamma, x) \leq \norm{\gamma}_{\bar\sigma}+ C
$$
whenever $\dist(x,\gamma^{-1}) > \epsilon$. Also let 
$$
D_1: = \sup \{ \abs{G(x,y)} : \dist(x,y) \geq \epsilon/2\}. 
$$
Notice that $D < + \infty$ since $G$ is locally bounded.  By Proposition~\ref{prop:compactifying} there also exists a finite subset $F\subset \Gamma$ such that 
\begin{equation}\label{eqn:gamma inverse action in proof of expanding 2}
\gamma^{-1}\left(M- B_{\epsilon/2}(\gamma)\right)\subset B_{\epsilon/2}(\gamma^{-1})\quad \text{and} \quad \gamma\left(M- B_{\epsilon/2}(\gamma^{-1})\right)\subset B_{\epsilon/2}(\gamma)
\end{equation}
for any $\gamma\in\Gamma \smallsetminus F$. Let 
$$
D_2 : = \sup\{ \abs{\norm{\gamma}_\sigma - \norm{\gamma^{-1}}_{\bar\sigma}} : \gamma \in F\}. 
$$

Fix $\gamma \in \Gamma$. If $\gamma \in F$, then 
$$
\norm{\gamma^{-1}}_{\bar\sigma}-D_2 \leq \norm{\gamma}_\sigma \leq  \norm{\gamma^{-1}}_{\bar\sigma}+D_2.
$$
So assume $\gamma \notin F$. Fix $y \in M$ with $\dist(y, \gamma^{-1}) > \epsilon$ and fix $x' \in M- (B_{\epsilon}(\gamma y)\cup B_{\epsilon}(\gamma))$. Let $x: = \gamma^{-1}(x')$. Arguing as in the proof of part~\eqref{item:where sigma is large}, 
$$
\dist(\gamma x, \gamma y) = \dist(x', \gamma y) > \epsilon \quad \text{and} \quad \dist(x,y) > \epsilon/2. 
$$
Then using  the GPS system property (Definition~\ref{def:GPS}) in the second inequality and Observation~\ref{obs:triangle inequality}\eqref{item:inverses} in the third,
\begin{align*}
\norm{\gamma}_\sigma & \leq \sigma(\gamma, y)+C \leq G(\gamma x, \gamma y) - G(x,y) - \bar{\sigma}(\gamma, x) +\kappa+C\\
& \leq G(\gamma x, \gamma y) - G(x,y) + \bar{\sigma}(\gamma^{-1}, \gamma x) +3\kappa+C \\
& \leq 2D_1 + \norm{\gamma^{-1}}_{\bar{\sigma}} +3\kappa+2C= \norm{\gamma^{-1}}_{\bar{\sigma}} +D+3\kappa +2C.
\end{align*}

The same reasoning can be used to show that 
\begin{align*}
\norm{\gamma^{-1}}_{\bar{\sigma}} & \leq \norm{\gamma}_{\sigma}+D +3\kappa+2C. \qedhere
\end{align*}
\end{proof}

\begin{proof}[Proof of~\eqref{item:gromprod is what it should be}] Fix $\epsilon>0$ such that $M\not\subset B_{\epsilon}(x)\cup B_{\epsilon}(y)$ for all $x,y\in \Gamma \sqcup M$. For notational convenience, given $\alpha, \beta \in \Gamma$ let 
$$
G(\alpha, \beta) : = \norm{\alpha^{-1}}_\sigma+\norm\beta_\sigma-\norm{\alpha^{-1}\beta}_\sigma.
$$

Fix $a\neq b\in \Lambda(\Gamma)$ and sequences $\{\alpha_n\},\{\beta_n\} \subset \Gamma$ converging to $a,b$ respectively. Passing to a subsequence we can  assume that  $\alpha_n^{-1}\to a_-$ and $\beta_n^{-1}\to b_-$.
Note that $a\neq b$ implies that $\alpha_n^{-1}\beta_n\to a_-$ and $\beta_n^{-1}\alpha_n\to b_-$.

Fix $x,y,z\in M$ such that 
\begin{equation*}
\dist(x,a),\ \dist(y,b_-),\ \dist(z,b),\ \dist(z,a_-) >\epsilon.
\end{equation*}

Passing to a further subsequence and using the facts that $\alpha_n^{-1}x\to a_-$, $\alpha_nz\to a$ and $\beta^{-1}_n\alpha_n z\to b_-$, we can assume that 
\begin{equation*}
\dist(\alpha_n,a),\ \dist(\beta_n^{-1},b_-),\ \dist(\alpha_n^{-1}x,a_-),\ \dist(\alpha_n^{-1}\beta_n,a_-),\ \dist(\alpha_nz,a),\ \dist(\beta_n^{-1}\alpha_nz,b_-)  <\frac\epsilon2.
\end{equation*}
This implies that
\begin{equation*}
 \dist(\alpha_n,x),\ \dist(\beta_n^{-1},y),\ \dist(\alpha_n^{-1}x,z),\ \dist(\alpha_n^{-1}\beta_n,z),\ \dist(\alpha_n z,x),\ \dist(\beta_n^{-1}\alpha_n z,y) \geq\frac\epsilon2.
\end{equation*}

Then using the constant $C$ from part~\eqref{item:norm vs dual norm} we have
\[ G(\alpha_n,\beta_n) \leq \norm{\alpha_n^{-1}}_{\sigma} + \norm{\beta_n}_\sigma - \norm{\beta_n^{-1}\alpha_n}_{\bar\sigma} + C. \]
Since $\sigma, \bar\sigma$ are expanding, there exists $C_\epsilon$ such that 
$$
\norm\gamma_\sigma-C_\epsilon \leq \sigma(\gamma,p)\leq \norm\gamma_\sigma+C_\epsilon \quad \text{and} \quad \norm\gamma_{\bar\sigma}-C_\epsilon \leq \bar\sigma(\gamma,p)\leq \norm\gamma_{\bar\sigma}+C_\epsilon
$$
whenever $\dist(p,\gamma^{-1})\geq \tfrac\epsilon2$.

Then
\[ G(\alpha_n,\beta_n) \leq \sigma(\alpha_n^{-1},x) + \sigma(\beta_n,y) - \bar\sigma(\beta_n^{-1}\alpha_n,z) + C +  3C_\epsilon.\]
Using the Observation~\ref{obs:triangle inequality}\eqref{item:inverses} (in the first inequality below) and the fact that $\bar\sigma$ is a $\kappa$-coarse-cocycle (in the second), it follows that
\begin{align*}
 G(\alpha_n,\beta_n)
 & \leq \sigma(\alpha_n^{-1},x) + \sigma(\beta_n,y) + \bar\sigma(\alpha^{-1}_n\beta_n, \beta_n^{-1}\alpha_n z) + C + 3C_\epsilon + 2\kappa\\
 & \leq \sigma(\alpha_n^{-1},x) + \sigma(\beta_n,y) + \bar\sigma(\alpha_n^{-1},\alpha_n z) + \bar\sigma(\beta_n,\beta_n^{-1}\alpha_n z) + C + 3C_\epsilon + 3\kappa.
\end{align*}
Next we use the GPS system property (Definition~\ref{def:GPS}), which implies
\begin{align*}
G(\alpha_n,\beta_n) & \leq  \left( \sigma(\alpha_n^{-1},x) + \bar\sigma(\alpha_n^{-1},\alpha_n z)\right)+\left(\sigma(\beta_n,y) + \bar\sigma(\beta_n,\beta_n^{-1}\alpha_n z)\right) + C + 3C_\epsilon + 3\kappa\\
&\leq \left(G(z,\alpha_n^{-1}x) - G(\alpha_n z,x)\right) + \left(G(\alpha_n z, \beta_n y) - G(\beta_n^{-1}\alpha_n z,y)\right) + C + 3C_\epsilon + 5\kappa.
\end{align*}
Finally since $G$ is locally bounded there is $C_\epsilon'>0$ such that $\abs{G(p,q)}\leq C_\epsilon'$ whenever $\dist(p,q)\geq\tfrac\epsilon2$.
Then
\[  G(\alpha_n,\beta_n) \leq G(\alpha_n z,\beta_n y) + C + 3C_\epsilon + 5\kappa + 3C'_\epsilon.\]

We get a lower bound for $G(\alpha_n,\beta_n)$ in a similar way:
\begin{align*}
 G(\alpha_n,\beta_n) & \geq \norm{\alpha_n^{-1}}_{\sigma} + \norm{\beta_n}_\sigma - \norm{\beta_n^{-1}\alpha_n}_{\bar\sigma} - C \\
 & \geq \sigma(\alpha_n^{-1},x) + \sigma(\beta_n,y) - \bar\sigma(\beta_n^{-1}\alpha_n,z) - C - 3C_\epsilon \\
 & \geq \sigma(\alpha_n^{-1},x) + \sigma(\beta_n,y) + \bar\sigma(\alpha^{-1}_n\beta_n, \beta^{-1}\alpha_n z) - C - 3C_\epsilon - 2\kappa\\
 & \geq \sigma(\alpha_n^{-1},x) + \sigma(\beta_n,y) + \bar\sigma(\alpha_n^{-1},\alpha_n z) + \bar\sigma(\beta_n,\beta_n^{-1}\alpha_n z) - C - 3C_\epsilon - 3\kappa\\
 & \geq G(z,\alpha_n^{-1}x) - G(\alpha_n z,x) + G(\alpha_n z,\beta_n y) - G(\beta_n^{-1}\alpha_n z,y) - C - 3C_\epsilon - 5\kappa\\ 
 & \geq G(\alpha_n z,\beta_n y) - C - 3C_\epsilon - 5\kappa - 3C'_\epsilon.
\end{align*}

As $\alpha_n z\to a$ and $\beta_n y\to b$, we conclude using the coarse continuity of $G$.
\end{proof}


\section{Patterson--Sullivan measures}


Using the results established in Proposition~\ref{prop:basic properties}, we can carry out a modification of the standard  construction of a Patterson--Sullivan measure 
due to Patterson \cite{Patterson} in the presence of an expanding coarse-cocycle.

\begin{theorem}
\label{thm:PS measures exist} If $\sigma$ is an expanding $\kappa$-coarse-cocycle for a convergence group $\Gamma\subset \mathsf{Homeo}(M)$ and
$\delta:=\delta_\sigma(\Gamma) < + \infty$, then there exists a $2\kappa\delta$-coarse $\sigma$-Patterson--Sullivan measure of dimension $\delta$, 
which is supported on the limit set $\Lambda(\Gamma)$.
\end{theorem} 

To establish the above theorem we will need a special magnitude, given by the following lemma which allows us to view the coarse-cocycle as a ``Busemann cocycle'' associated to the ``distance'' $\rho(\alpha,\beta) = \norm{\alpha^{-1}\beta}_\sigma$ on $\Gamma$. 

\begin{lemma}\label{lem:nice magnitude}
Suppose $\Gamma \subset \mathsf{Homeo}(M)$ is a convergence group and $\sigma$ is an expanding $\kappa$-coarse-cocycle.  
Then there is a choice of magnitude $\norm{\cdot}_\sigma$ such that: if $x \in \Lambda(\Gamma)$ and $\alpha \in \Gamma$, then 
$$
\limsup_{\gamma \rightarrow x} \abs{ \sigma(\alpha,x)-(\norm{\alpha \gamma}_\sigma-\norm{\gamma}_\sigma)  } \leq 2\kappa. 
$$
Hence for any choice of magnitude $\norm{\cdot}_\sigma'$ there exists $C > 0$ such that: if $x \in \Lambda(\Gamma)$ and $\alpha \in \Gamma$, then 
$$
\limsup_{\gamma \rightarrow x} \abs{ \sigma(\alpha,x)-(\norm{\alpha \gamma}_\sigma'-\norm{\gamma}_\sigma')  } \leq C.
$$
\end{lemma}

\begin{proof}
Let $\dist$ be a compatible metric on $\Gamma \sqcup M$.
Fix two distinct points $p\neq q\in M$ and let $\epsilon:=\dist(p,q)/2$.
Next let $\psi \colon \Rb \to [0,1]$ be a continuous function with ${\rm supp}(\psi) \subset (\epsilon/2,+\infty)$ and $\psi \equiv 1$ on $[\epsilon, +\infty)$.
Note that any $x\in \Gamma\sqcup M$ is a distance at least $\epsilon$ from either $p$ or $q$, so  $\psi(\dist(x,p))+\psi(\dist(x,q))>0$.
Let 
\begin{equation*}
  a(x):=\frac{\psi(\dist(x,p))}{\psi(\dist(x,p))+\psi(\dist(x,q))} \quad\text{and}\quad b(x):=\frac{\psi(\dist(x,q))}{\psi(\dist(x,p))+\psi(\dist(x,q))}.
\end{equation*}
Finally, for $\gamma \in \Gamma$ define
\begin{equation*}
\norm{\gamma}_\sigma := a(\gamma^{-1})\sigma(\gamma,p)+b(\gamma^{-1})\sigma(\gamma,q).
\end{equation*}

We claim that $\norm{\cdot}_\sigma$ is a magnitude function. Since $\sigma$ is expanding there is some magnitude function $\norm{\cdot}_\sigma'$ and there is some $C > 0$ such that: if $x \in M$, $\gamma \in \Gamma$ and $\dist(x, \gamma^{-1}) > \epsilon/2$, then 
$$
\norm{\gamma}_\sigma' - C \leq \sigma(\gamma, x) \leq \norm{\gamma}_\sigma' + C.
$$
Notice that if $a(\gamma^{-1}) \neq 0$, then $\dist(p,\gamma^{-1}) > \epsilon/2$ and if $b(\gamma^{-1}) \neq 0$, then $\dist(q,\gamma^{-1}) > \epsilon/2$. Then, since $a+b \equiv 1$, we have 
$$
\sup_{\gamma \in \Gamma} \abs{\norm{\gamma}_\sigma-\norm{\gamma}_\sigma'} \leq 2C.
$$
So $\norm{\cdot}_\sigma$ is a magnitude function.

To show that this magnitude function has the desired property, fix $x\in \Lambda(\Gamma)$ and $\alpha\in \Gamma$.
Then fix $\{\gamma_n\}\subset\Gamma$ such that $\gamma_n \rightarrow x$ and 
\[L:=\limsup_{\gamma\to x}\abs{\norm{\alpha\gamma}_\sigma-\norm{\gamma}_\sigma-\sigma(\alpha,x)}=\lim_{n \rightarrow \infty }\abs{\norm{\alpha\gamma_n}_\sigma-\norm{\gamma_n}_\sigma-\sigma(\alpha,x)}.\]
Passing to a subsequence we can assume that $\gamma_n^{-1}\to y\in M$. Then $\gamma_n^{-1}\alpha^{-1}\to y$ too.

First suppose that $y$ is distinct from both $p$ and $q$. By the coarse cocycle property, $\norm{\alpha\gamma_n}_\sigma$ is $\kappa$-close to
\[a(\gamma_n^{-1}\alpha^{-1})(\sigma(\alpha,\gamma_np)+\sigma(\gamma_n,p))+b(\gamma_n^{-1}\alpha^{-1})(\sigma(\alpha,\gamma_nq)+\sigma(\gamma_n,q))\]
Note that $p,q\neq y$ implies $\gamma_np,\gamma_nq\to x$.
By the $\kappa$-coarse continuity of $\sigma$ (and since $a+b \equiv 1$) we get that
\[\limsup_{n \rightarrow \infty} \abs{a(\gamma_n^{-1}\alpha^{-1})\sigma(\alpha,\gamma_np)+b(\gamma_n^{-1}\alpha^{-1})\sigma(\alpha,\gamma_nq)-\sigma(\alpha,x)}\leq \kappa.\]
Thus $L$ is bounded above by
\[2\kappa + \limsup_n\abs{(a(\gamma_n^{-1}\alpha^{-1})-a(\gamma_n^{-1}))\sigma(\gamma_n,p)+(b(\gamma_n^{-1}\alpha^{-1})-b(\gamma_n^{-1}))\sigma(\gamma_n,q)}.\]
Next observe that since $a+b\equiv 1$ we have
\[a(\gamma_n^{-1}\alpha^{-1})-a(\gamma_n^{-1})=-b(\gamma_n^{-1}\alpha^{-1})+b(\gamma_n^{-1}).\]
Hence  $L$ is bounded above by
\[2\kappa + \limsup_{n \rightarrow \infty} \abs{(b(\gamma_n^{-1}\alpha^{-1})-b(\gamma_n^{-1}))(\sigma(\gamma_n,q)-\sigma(\gamma_n,p))}.\]
Since $y$ is distinct from both $p$ and $q$, the expanding property implies that 
$$
\abs{\sigma(\gamma_n,q)-\sigma(\gamma_n,p)}
$$
 is bounded.
Moreover, since $\gamma_n^{-1}\alpha^{-1}$ and $\gamma_n^{-1}$ both converge to $y$, the continuity of $b$ implies that 
$$
\lim_{n \rightarrow \infty} b(\gamma_n^{-1}\alpha^{-1})-b(\gamma_n^{-1})= 0.
$$
Therefore  
\[L\leq 2\kappa.\]

Next suppose that $y$ coincides with $p$ or $q$, say $q$ (the other case is similar).
Then for $n$ large enough $b(\gamma_n^{-1})=b(\gamma_n^{-1}\alpha^{-1})=0$.
Thus for $n$ large,
\[\norm{\alpha\gamma_n}_\sigma-\norm{\gamma_n}_\sigma-\sigma(\alpha,x)=\sigma(\alpha\gamma_n,p)-\sigma(\gamma_n,p)-\sigma(\alpha,x).\]
By the $\kappa$-coarse cocycle property this is $\kappa$-close to $\sigma(\alpha,\gamma_np)-\sigma(\alpha,x)$.
Using that $\gamma_np\to x$ and the $\kappa$-coarse continuity property, we have 
$$
\limsup_{n \rightarrow \infty} \abs{\sigma(\alpha,\gamma_np)-\sigma(\alpha,x)}\leq \kappa,
$$
and hence $L\leq2\kappa$. 
\end{proof}

\begin{proof}[Proof of Theorem~\ref{thm:PS measures exist}]
Endow $\Gamma \sqcup M$ with the compactifying topology and let $\dist$ be a compatible distance on $\Gamma \sqcup M$ (see Proposition~\ref{prop:compactifying}). 
For $x \in \Gamma \sqcup M$, let $\mathcal D_x$ denote the Dirac measure centered at $x$. 
Finally, fix a magnitude $\norm{\cdot}_\sigma$ satisfying Lemma~\ref{lem:nice magnitude}.

By~\cite[Lemma 3.1]{Patterson}, there exists a non-decreasing function $\chi\colon \mathbb R \to\mathbb R_{\geq1}$ such that 
\begin{enumerate}[label=(\alph*)]
 \item \label{item:PS meas chi subexp} For every $\epsilon>0$ there exists $R>0$ such that $\chi(r+t)\leq e^{\epsilon t}\chi(r)$ for any $r\geq R$ and $t\geq 0$,
 \item \label{item PS meas chi div} $\sum_{g \in \Gamma} \chi(\norm g_\sigma ) e^{-\delta \norm g_\sigma } =+\infty$
\end{enumerate}
(when $\sum_{g \in \Gamma}  e^{-\delta \norm{g}_\sigma} = +\infty$, we can take $\chi \equiv 1$).

For $s >\delta$, define a Borel probability measure on $\Gamma \sqcup M$ by 
$$
\mu_s := \frac{1}{Q^\chi_\sigma(s)} \sum_{g \in \Gamma} \chi(\norm g_\sigma ) e^{-s \norm g_\sigma  } \mathcal D_g,
$$
where $Q^\chi_\sigma(s):=\sum_{g \in \Gamma} \chi(\norm g_\sigma  ) e^{-s \norm g_\sigma  } $, which is finite by  property~\ref{item:PS meas chi subexp}. 
Then fix $s_n \searrow \delta$ such that $\mu_{s_n} \rightarrow \mu$ in the weak-$*$ topology. 
We claim that $\mu$ is a $2\kappa\delta$-coarse $\sigma$-Patterson--Sullivan measure of dimension $\delta$ on $M$.

By property~\ref{item PS meas chi div} of $\chi$,
$$
\lim_{s \searrow \delta} Q_\sigma^\chi(s)=+\infty.
$$
Hence $\mu$ is supported on $\Lambda(\Gamma)$  by Proposition~\ref{prop:compactifying}.

To verify the Radon--Nikodym derivative condition, fix $\gamma \in \Gamma$. Then define $\chi_\gamma \colon \Gamma \sqcup M \to \Rb$ by
$$
\chi_\gamma(x)=\begin{cases} 
\chi(\norm{\gamma^{-1}x}_\sigma)) / \chi(\norm x_\sigma ) & x\in \Gamma, \\
1 & x \in M. 
\end{cases}
$$
By Proposition~\ref{prop:basic properties}\eqref{item:tri inequality} there exists $C' =C'(\gamma)> 0$ such that 
\begin{equation}\label{eqn:estimate for h(gamma g)}
 - C' \leq \norm{\gamma^{-1}g}_\sigma-\norm g_\sigma  \leq C'
\end{equation} 
for all $g \in \Gamma$.
So  Property~\ref{item:PS meas chi subexp} of $\chi$ implies that $\chi_\gamma$ is continuous. 

Next define $f_\gamma^\pm \colon \Gamma \sqcup M \to \Rb$ by 
$$
f_\gamma^\pm(x) = \begin{cases} 
\norm{\gamma^{-1} x}_\sigma - \norm x_\sigma  & x \in \Gamma, \\
\pm\limsup\limits_{\substack{g_m \rightarrow x, \ \{g_m\} \subset \Gamma}}\pm( \norm{\gamma^{-1}g_m}_\sigma - \norm{g_m}_\sigma) & x \in \Lambda(\Gamma), \\
 \mp C' & x \in M \smallsetminus \Lambda(\Gamma).
\end{cases}
$$
 Equation~\eqref{eqn:estimate for h(gamma g)} implies that $f_\gamma^+$ (resp.\ $f_\gamma^-$) is upper (resp.\ lower) semicontinuous and hence Borel measurable (a priori we did not assume $x\mapsto \sigma(\gamma^{-1},x)$ is measurable).

By Lemma~\ref{lem:nice magnitude},
if $x \in \Lambda(\Gamma)$, then 
$$
\abs{ \sigma(\gamma^{-1},x)-f_{\gamma}^\pm(x)  } \leq 2\kappa.
$$
Also, recall that $\chi_\gamma \equiv 1$ on $\Lambda(\Gamma)$.
Then, since $\mu$ and $\gamma_*\mu$ are supported on $\Lambda(\Gamma)$, to show that $\mu$ and $\gamma_*\mu$ are absolutely continuous with
$$
 e^{-2\kappa\delta-\delta \sigma(\gamma^{-1}, \cdot)}\le\frac{d\gamma_*\mu}{d\mu}\le e^{2\kappa\delta-\delta \sigma(\gamma^{-1}, \cdot)}
$$
(almost everywhere),
it suffices to show that
\[\chi_\gamma e^{-\delta f_\gamma^+} \mu \leq \gamma_*\mu\leq \chi_\gamma e^{-\delta f_\gamma^-} \mu.\]

Since $f_\gamma^+$ is upper semicontinuous there exists a sequence $\{f_m\}$ of continuous functions  on $\Gamma \sqcup M$ where $f_1 \geq f_2 \geq \dots$ and $f_m \rightarrow f_\gamma^+$ pointwise. 

Note that
\begin{align*}
\gamma_*\mu_s & = \frac{1}{Q^\chi_\sigma(s)} \sum_{g \in \Gamma} \chi(\norm{g}_\sigma) e^{-s \norm{g}_\sigma} \dirac_{\gamma g}  = \frac{1}{Q^\chi_\sigma(s)} \sum_{g' \in \Gamma} \chi(\norm{\gamma^{-1} g'}_\sigma) e^{-s \norm{\gamma^{-1} g'}_\sigma} \dirac_{g'}  \\
& =  \frac{1}{Q^\chi_\sigma(s)} \sum_{g'\in\Gamma} \chi_\gamma(g') e^{-s f_\gamma^\pm(g')} \chi(\norm{g'}_\sigma) e^{-s \norm{g'}_\sigma} \dirac_{g'}  = \chi_\gamma e^{-s f_\gamma^\pm} \mu_s. 
\end{align*}
So $\gamma_*\mu_s \geq \chi_\gamma e^{-s f_m} \mu_s$ for all $m$. Since $f_m$ is continuous, we have 
$$
\lim_{n \rightarrow \infty} \norm{\chi_\gamma e^{-s_n f_m}-\chi_\gamma e^{-\delta f_m}}_\infty = 0
$$
and hence
$$
\gamma_* \mu = \lim_{n \rightarrow \infty} \gamma_*\mu_{s_n} \geq \lim_{n \rightarrow \infty}\chi_\gamma e^{-s_n f_m} \mu_{s_n} =\chi_\gamma e^{-\delta f_m} \mu
$$
for all $m$. Then the monotone convergence theorem implies 
$$
\gamma_* \mu \geq \lim_{m \rightarrow \infty} \chi_\gamma e^{-\delta f_m} \mu = \chi_\gamma e^{-\delta f_\gamma^+} \mu.
$$
A very similar argument shows that $\gamma_* \mu \leq \chi_\gamma e^{-\delta f_\gamma^-} \mu$.

\end{proof} 

One can use the above Patterson--Sullivan measure to obtain the following classical entropy gap result (see \cite[Prop.\,2]{DOP}).

\begin{theorem}\label{thm:subgp div implies smaller crit exp}
Suppose $\Gamma \subset \mathsf{Homeo}(M)$ is a convergence group, $\sigma$ is an expanding coarse-cocycle, and $\delta_\sigma(\Gamma) < +\infty$. If $G \subset \Gamma$ is a subgroup where $\Lambda(G)$ is a strict subset of $\Lambda(\Gamma)$ and 
$$
\sum_{g \in G} e^{-\delta_\sigma(G)\norm{g}_\sigma}=+\infty,
$$ 
then $\delta_\sigma(G) < \delta_\sigma(\Gamma)$. 
\end{theorem} 

\begin{proof} Fix an open set $U \subset M$ such that $U \cap \Lambda(\Gamma) \neq \emptyset$ and $\overline{U} \cap \Lambda(G) = \emptyset$. 
By the definition of  $\Lambda(G)$, $G$ acts properly discontinuously on $M \smallsetminus \Lambda(G)$. 
Hence there exists $N > 0$ such that every point in $M$ is contained in at most $N$ different $G$-translates of $U$.

By the previous result, there exist a constant $C_2 > 0$ and a $C_2$-coarse $\sigma$-Patterson--Sullivan measure $\mu$ for $\Gamma$ of dimension $\delta_\sigma(\Gamma)$ supported on $\Lambda(\Gamma)$. Hence
$$
e^{-C_2-\delta_\sigma(\Gamma) \sigma(\gamma^{-1}, \cdot)} \leq \frac{d\gamma_* \mu}{d\mu} \leq e^{C_2 -\delta_\sigma(\Gamma) \sigma(\gamma^{-1}, \cdot)} 
$$
for all $\gamma \in \Gamma$

Suppose for a contradiction that $\delta_\sigma(G) = \delta_\sigma(\Gamma)$. 
Since $\Gamma$ acts minimally on $\Lambda(\Gamma)$ we must have $\mu(U) > 0$.
Since the Hausdorff distance between $\overline U$ and $G$ in $\Gamma\sqcup M$ is positive, by the expanding property there is $C_3>0$ such that $\abs{\sigma(g,x)-\norm{g}_\sigma}\leq C_3$ for all $x\in \overline U$ and $g\in G$.
 Then
\begin{align*}
N & \ge \sum_{g \in G} \mu(gU)=  \sum_{g \in G} (g^{-1}_*)\mu(U) \ge e^{-C_2}  \sum_{g \in G} \int_U e^{-\delta_\sigma(\Gamma)\sigma(g, x)} d\mu(x) \\
& \ge  \frac{\mu(U)}{e^{C_2+\delta_\sigma(G)C_3}} \sum_{g \in G} e^{-\delta_\sigma(G)\norm{g}_\sigma}=+\infty.
\end{align*}
So we have a contradiction. 
\end{proof}


\section{Shadows and their properties}\label{sec:basic properties of fake shadows}

In this section we define our shadows, establish some of their basic properties, relate them to a notion of uniformly conical limit points, and compare these shadows to the classically-defined shadows in the Gromov hyperbolic setting. 

\subsection{Basic properties}
Suppose for the rest of the section that $\Gamma \subset \mathsf{Homeo}(M)$ is a convergence group. Fix a compatible metric $\dist$ on $\Gamma \sqcup M$ and let $B_r(x) \subset \Gamma \sqcup M$ denote the open ball of radius $r > 0$ centered at $x$. 
Given $\epsilon>0$ and $\gamma\in\Gamma$, the associated \emph{shadow} is (see Definition~\ref{def:shadows})
$$
\mc S_\epsilon(\gamma):=\gamma\left(M-B_\epsilon(\gamma^{-1})\right).
$$
Notice that if  $\epsilon>\epsilon'>0$, then  $\mc S_\epsilon(\gamma)\subset\mc S_{\epsilon'}(\gamma)$.

We first establish basic properties of our shadows which are analogues of standard properties of classical shadows.

\begin{proposition}\label{prop:properties of fake shadows} If $\epsilon>0$ and $\sigma \colon \Gamma \times M \rightarrow \Rb$ is an expanding coarse-cocycle, then: 
\begin{enumerate}
\item\label{item:cocycle big on the fake shadow} There exists $C_1 > 0$ such that: if $x \in \mc S_\epsilon(\gamma)$, then 
$$
\norm{\gamma}_{\sigma} - C_1 \leq \sigma(\gamma, \gamma^{-1}(x)) \leq \norm{\gamma}_{\sigma} + C_1.
$$
\item\label{item:collapsing fake shadows}  If $\{\gamma_n\} \subset \Gamma$ is an escaping sequence, then 
$$
\lim_{n \rightarrow \infty} {\rm diam} \, \mc S_\epsilon(\gamma_n)= 0 \quad \text{and} \quad \lim_{n \rightarrow \infty} \inf_{x \in \mc S_\epsilon(\gamma_n)} \dist(\gamma_n, x)=0
$$
where the diameter is with respect to $\dist$. In particular, the Hausdorff distance with respect to $\dist$ between the sets $\{\gamma_n\}$ and $\mc S_\epsilon(\gamma_n)$ converges to zero.

\item\label{item:multiplicative property of intersecting fake shadows}

There exists $C_2 > 0$ such that: if $\alpha, \beta \in \Gamma$, $\norm{\alpha}_{\sigma} \leq \norm{\beta}_{\sigma}$, and 
$$ \mc S_\epsilon(\alpha) \cap \mc S_\epsilon(\beta) \neq \emptyset,
$$
 then 
$$
 \norm{\alpha^{-1}\beta}_{\sigma} + \norm{\alpha}_{\sigma} - C_2 \leq \norm{\beta}_{\sigma}\leq  \norm{\alpha^{-1}\beta}_{\sigma} + \norm{\alpha}_{\sigma} + C_2.
$$

\item\label{item:nesting shadows} There exists $0 < \epsilon' < \epsilon$ such that: if  $\alpha, \beta \in \Gamma$,  $\norm\alpha_\sigma\leq\norm\beta_\sigma$, and $\mc S_\epsilon(\alpha)\cap \mc S_\epsilon(\beta)\neq\emptyset$, then
$$
\mc S_\epsilon(\beta) \subset \mc S_{\epsilon'}(\alpha). 
$$
\item\label{item:Vish covering lemma}   There exists $0 < \epsilon' < \epsilon $ such that:  if $I \subset \Gamma$, then there exists $J \subset I$ such that the shadows $\{ \mc S_\epsilon(\gamma) : \gamma \in J\}$ are disjoint and 
$$
\bigcup_{\gamma \in I} \mc S_\epsilon(\gamma) \subset \bigcup_{\gamma \in J} \mc S_{\epsilon'}(\gamma). 
$$
\end{enumerate}

\end{proposition}

\begin{proof} Part~\eqref{item:cocycle big on the fake shadow} follows from the definition of expanding coarse-cocycles and part~\eqref{item:collapsing fake shadows} is a consequence of Proposition~\ref{prop:compactifying}\eqref{item:contraction of shadows}.

Part~\eqref{item:multiplicative property of intersecting fake shadows}: suppose for a contradiction that the claim is false.
Then for each $n \geq 1$ there exist  $\alpha_n, \beta_n \in \Gamma$ such that $\norm{\alpha_n}_{\sigma} \leq \norm{\beta_n}_{\sigma}$,  
$$
 \mc S_\epsilon(\alpha_n) \cap \mc S_\epsilon(\beta_n) \neq \emptyset \quad \text{and} \quad \abs{\norm{\beta_n}_{\sigma} - \norm{\alpha_n^{-1}\beta_n}_{\sigma} - \norm{\alpha_n}_{\sigma}} \geq n.
$$

By Proposition~\ref{prop:basic properties}\eqref{item:tri inequality}, both $\{\alpha_n^{-1}\beta_n\}$ and $\{\alpha_n\}$ must be escaping.
Since $\norm{\beta_n}_{\sigma}\geq \norm{\alpha_n}_{\sigma}$, this implies that $\{\beta_n\}$ is also escaping.
Then by Proposition~\ref{prop:basic properties}\eqref{item:estimate on difference of transformations} we have  
$$
\lim_{n \rightarrow \infty} \dist(\beta_n^{-1}\alpha_n,\beta^{-1}_n) = 0. 
$$

For each $n$, fix $x_n\in \mc S_\epsilon(\alpha_n) \cap \mc S_\epsilon(\beta_n)$.
By definition, $\dist(\beta^{-1}_nx_n,\beta^{-1}_n)\geq \epsilon$ and so 
$$
\dist(\beta^{-1}_nx_n,\beta^{-1}_n\alpha_n)\geq \dist(\beta^{-1}_nx_n,\beta^{-1}_n)-\dist(\beta_n^{-1}\alpha_n,\beta^{-1}_n)\geq \epsilon/2
$$
 for $n$ large enough. Also,  $\dist(\alpha^{-1}_nx_n,\alpha^{-1}_n)\geq \epsilon$ for any $n$.

Since $\sigma$ is expanding, there exists $C>0$ such that 
$$
\abs{\sigma(\gamma,x)- \norm{\gamma}_\sigma}\leq C
$$
 for all $\gamma\in \Gamma$ and $x\in M-B_{\epsilon/2}(\gamma^{-1})$. 
Using the coarse-cocycle property, 
\begin{align*}
 \vert\norm{\beta_n}_{\sigma} &- \norm{\alpha_n^{-1}\beta_n}_{\sigma} - \norm{\alpha_n}_{\sigma}\vert\\
& \leq \abs{\sigma(\beta_n,\beta_n^{-1}x_n) -  \sigma(\alpha_n,\alpha_n^{-1}x_n) -\sigma(\alpha_n^{-1}\beta_n,\beta_n^{-1}x_n)}+3C\\
&=\abs{\sigma(\alpha_n\alpha_n^{-1}\beta_n,\beta_n^{-1}x_n) -  \sigma(\alpha_n,\alpha_n^{-1}x_n) -\sigma(\alpha_n^{-1}\beta_n,\beta_n^{-1}x_n)}+3C\\
 &\leq \kappa+3C
\end{align*}
and we have a contradiction. 
Thus part ~\eqref{item:multiplicative property of intersecting fake shadows}  is true.

Part~\eqref{item:nesting shadows}: suppose for a contradiction that there exist $\{\alpha_n\}$, $\{\beta_n\}\subset\Gamma$ and $\epsilon_n\to 0$ such that $\norm{\alpha_n}_\sigma\leq \norm{\beta_n}_\sigma$, $\mc S_{\epsilon}(\alpha_n)\cap \mc S_\epsilon(\beta_n)\neq\emptyset$, and $\mc S_\epsilon(\beta_n)\not\subset \mc S_{\epsilon_n}(\alpha_n)$. Then
\begin{equation}\label{eqn:key inclusion in nested shadows}
\alpha_n^{-1}\beta_n(M-B_\epsilon(\beta_n^{-1}))=\alpha_n^{-1}\mc S_\epsilon(\beta_n) \not\subset \alpha_n^{-1}\mc S_{\epsilon_n}(\alpha_n)=M-B_{\epsilon_n}(\alpha^{-1}_n)
\end{equation}
for all $n \geq 1$. 

We first show that $\{ \alpha_n^{-1}\beta_n\}$ must be escaping. 
Suppose not. 
Then we can assume that $\gamma : = \alpha_n^{-1}\beta_n$ for all $n$. 
Since $\Gamma \sqcup M$ is compact, there exists $\delta > 0$ such that $\gamma B_\epsilon(x) \supset B_\delta(\gamma x)$ for all $x \in \Gamma\sqcup M$.
Then $\gamma B_\epsilon(\beta_n^{-1}) \supset B_\delta(\alpha_n^{-1})$ for all $n \geq 1$.
Hence 
\begin{equation*}
\alpha_n^{-1}\beta_n(M-B_\epsilon(\beta_n^{-1}))=M-\gamma B_\epsilon(\beta_n^{-1}) \subset M-B_{\delta}(\alpha^{-1}_n)
\end{equation*}
for all $n \geq 1$. So we have a contradiction with Equation~\eqref{eqn:key inclusion in nested shadows}.

Then $\{\beta_n^{-1}\alpha_n\}$ is also escaping and so by Proposition~\ref{prop:basic properties}\eqref{item:estimate on difference of transformations} we have 
$$
\lim_{n \rightarrow \infty} \dist(\beta_n^{-1},\beta_n^{-1}\alpha_n)=0.
$$
Thus for $n$ large enough
$$
\alpha_n^{-1}\mc S_\epsilon(\beta_n) = \alpha_n^{-1}\beta_n\left( M-B_\epsilon(\beta_n^{-1})\right)\subset \alpha_n^{-1}\beta_n\left(M-B_{\epsilon/2}(\beta_n^{-1}\alpha_n)\right) = \mc S_{\epsilon/2}(\alpha_n^{-1}\beta_n). 
$$
Then, applying part~\eqref{item:collapsing fake shadows} to the escaping sequence $\{\alpha_n^{-1}\beta_n\}$, we obtain that the diameter of 
$$
\alpha_n^{-1}\mc S_\epsilon(\beta_n) \subset \mc S_{\epsilon/2}(\alpha_n^{-1}\beta_n)
$$
tends to zero, and hence is less than $\epsilon/2$ for $n$ large enough. Further, by assumption, $\alpha_n^{-1}\mc S_\epsilon(\beta_n)$ intersects $\alpha_n^{-1}\mc S_\epsilon(\alpha_n)=M-B_{\epsilon}(\alpha_n^{-1})$ for all $n$. Hence 
$$
\alpha_n^{-1}\mc S_\epsilon(\beta_n) \subset M-B_{\epsilon/2}(\alpha_n^{-1})
$$
for $n$ large enough, which implies that $\mc S_\epsilon(\beta_n)\subset \mc S_{\epsilon_n}(\alpha_n)$ for $n$ large enough. Thus we have a contradiction. 

Part~\eqref{item:Vish covering lemma}: using part~\eqref{item:nesting shadows}, the proof of this part of the proposition is standard, see e.g.~\cite[Lem.\,3.15]{Folland1999}. 

Let $\epsilon'$ be as in part~\eqref{item:nesting shadows}. Enumerate $I = \{ \gamma_1, \gamma_2, \dots\}$ such that 
$$
\norm{\gamma_1}_\sigma \leq \norm{\gamma_2}_\sigma \leq \dots
$$
(this is possible by Proposition~\ref{prop:basic properties}\eqref{item:properness}.) Inductively define $j_1 < j_2 < \dots$ as follows: let $j_1=1$, then supposing $j_1, \dots, j_k$ have been selected pick $j_{k+1}$ to be the smallest index greater than $j_k$ such that 
$$
\mc S_\epsilon(\gamma_{j_{k+1}}) \cap \bigcup_{i=1}^k \mc S_\epsilon(\gamma_{j_i}) = \emptyset. 
$$

We claim that $J= \{ \gamma_{j_1}, \gamma_{j_2}, \dots \}$ suffices (it is possible for $J$ to be finite). 
By definition the shadows $\{ \mc S_\epsilon(\gamma) : \gamma \in J\}$ are disjoint. 
Further if $\gamma_n \notin J$, then there exists some index $j_k < n$ such that 
$$
\mc S_\epsilon(\gamma_n) \cap \mc S_\epsilon(\gamma_{j_k})  \neq \emptyset
$$
(otherwise we would have $\gamma_n \in J$). Then part~\eqref{item:nesting shadows} implies that 
$$
\mc S_\epsilon(\gamma_n) \subset  \mc S_{\epsilon'}(\gamma_{j_k}) .
$$
So 
 \begin{equation*}
\bigcup_{\gamma \in I} \mc S_\epsilon(\gamma) \subset \bigcup_{\gamma \in J} \mc S_{\epsilon'}(\gamma). \qedhere
\end{equation*}

\end{proof}

\subsection{Uniformly conical limit points}

Next we introduce a notion of uniformly conical limit points and relate them to the shadows defined above.

\begin{definition}\label{defn:of uniform conical points} Given $\epsilon > 0$, the \emph{$\epsilon$-uniform conical limit set}, denoted $\Lambda_\epsilon^{\rm con}(\Gamma)$, is the set of points $x \in M$ such that there exist $a,b \in M$  and a sequence of elements $\{\gamma_n\}$ in $\Gamma$ where $\dist(a,b) \geq \epsilon$, $\lim_{n \to \infty} \gamma_n x = b$, and $\lim_{n \rightarrow \infty} \gamma_n y = a$ for all $y \in M \smallsetminus \{x\}$. 
\end{definition} 

Notice that by definition 
\begin{equation}\label{eqn:decomposition of conical limit set}
\Lambda^{\rm con}(\Gamma)  = \bigcup_{\epsilon > 0} \Lambda^{\rm con}_{\epsilon}(\Gamma)=\bigcup_{n =1}^\infty \Lambda^{\rm con}_{\frac{1}{n}}(\Gamma).
\end{equation}

We also observe that these limit sets are invariant. 

\begin{observation}\label{obs:invariance of uniform conical limit sets} If $\epsilon > 0$, then $\Lambda_\epsilon^{\rm con}(\Gamma)$ is $\Gamma$-invariant. 
\end{observation} 

\begin{proof} Fix $x \in \Lambda_\epsilon^{\rm con}(\Gamma)$ and $\gamma \in \Gamma$. Then there exist $a,b \in M$  and a sequence of elements $\{\gamma_n\}$ in $\Gamma$ where $\dist(a,b) \geq \epsilon$, $\lim_{n \to \infty} \gamma_n x = b$, and $\lim_{n \rightarrow \infty} \gamma_n y = a$ for all $y \in M \smallsetminus \{x\}$. Then $\lim_{n \to \infty} \gamma_n\gamma^{-1}(\gamma x) = b$ and $\lim_{n \rightarrow \infty} \gamma_n\gamma^{-1}(y) = a$ for all $y \in M \smallsetminus \{\gamma x\}$. So $\gamma x \in  \Lambda_\epsilon^{\rm con}(\Gamma)$. 
\end{proof} 

Next we relate the shadows to this notion of uniformly conical limit set.

\begin{lemma}\label{lem:shadows versus uniform conical limit sets} \
\begin{enumerate}
\item If $x \in \Lambda_{\epsilon}^{\rm con}(\Gamma)$ and $0 < \epsilon' < \epsilon$, then there exists an escaping sequence $\{\gamma_n\}\subset\Gamma$  such that 
$x\in \bigcap_n\mc S_{\epsilon'}(\gamma_n)$.

\item If there exists an escaping sequence $\{\gamma_n\}\subset\Gamma$  such that $x\in \bigcap_n\mc S_{\epsilon}(\gamma_n)$, then $x \in \Lambda_{\epsilon}^{\rm con}(\Gamma)$.
\end{enumerate}
\end{lemma} 

\begin{proof} 
First suppose that $x\in\Lambda^{\rm con}_\epsilon(\Gamma)$. Then there exist $a,b \in M$  and a sequence of elements $\{\gamma_n\}$ in $\Gamma$ where $\dist(a,b) \geq \epsilon$, $\lim_{n \to \infty} \gamma_n^{-1}x = b$, and $\lim_{n \rightarrow \infty} \gamma_n^{-1} y = a$ for all $y \in M \smallsetminus \{x\}$. Thus $\gamma_n^{-1} \rightarrow a \neq b$. So if $\epsilon' < \epsilon$, then  $\dist(\gamma_n^{-1} x,\gamma_n^{-1}) >  \epsilon'$ for $n$ sufficiently large. Thus 
$$
x =\gamma_n \gamma_n^{-1}(x) \in \gamma_n( M-B_{\epsilon'}(\gamma_n^{-1})) = \mc S_{\epsilon'}(\gamma_n)
$$
for $n$ sufficiently large.

Next suppose that $x\in \bigcap_n\mc S_\epsilon(\gamma_n)$ for some $\epsilon>0$ and some escaping $\{\gamma_n\}\subset\Gamma$. 
Passing to a subsequence we can assume that $\gamma_n^{-1}x\to a$, $\gamma_n^{-1}\to b$ and $\gamma_n\to c$.
In particular $\gamma_n^{-1}y\to b$ for any $y\in M \smallsetminus \{c\}$.
Since $x\in \mc S_\epsilon(\gamma_n)$ for every $n$, we have by definition $\dist(\gamma_n^{-1}x,\gamma_n^{-1})\geq \epsilon$.
Passing to the limit we get $\dist(a,b)\geq \epsilon$, so  $a\neq b$.
Moreover $x=c$, as otherwise $\{\gamma_n^{-1} x\}$ would have to converge to $b$. Hence $x\in\Lambda^{\rm con}_\epsilon(\Gamma)$.
\end{proof}

\subsection{Comparison to classical shadows} Suppose that $X$ is a proper geodesic Gromov hyperbolic space. Let $\Gamma \subset \mathsf{Isom}(X)$ be a discrete group. Then $\Gamma$ acts as a convergence group on the Gromov boundary $\partial_\infty X$ of $X$. 

Given $b,p \in X$ and $r > 0$ the associated shadow $\Oc_r(b,p) \subset \partial_\infty X$ is the set of all $x \in \partial_\infty X$ where there is some geodesic ray $\ell \colon [0,\infty) \rightarrow X$ where $\ell(0)=b$, $\lim_{t \rightarrow \infty}\ell(t) = x$, and $\ell$ intersects the open ball of radius $r>0$ centered at $p$. 

\begin{proposition} Fix a compatible metric $\dist$ on $\Gamma \sqcup \partial_\infty X$, and for $\epsilon > 0$ and $\gamma \in \Gamma$ let $\mc S_\epsilon(\gamma) \subset \partial_\infty X$ denote the shadow defined above. 
\begin{enumerate}
\item For any $b \in X$ and $r > 0$ there exists $\epsilon > 0$ such that 
$$
\Oc_r(b, \gamma(b)) \subset \mc S_\epsilon(\gamma)
$$
for all $\gamma \in \Gamma$. 
\item For any $b \in X$ and $\epsilon > 0$  there exists $r > 0$ such that 
$$
\mc S_\epsilon(\gamma) \subset \Oc_r(b, \gamma(b))
$$
for all $\gamma \in \Gamma$. 
\end{enumerate}
\end{proposition} 

\begin{proof} (1): Suppose that no such $\epsilon> 0$ exists. Then there exist $\{\gamma_n\}\subset\Gamma$ and $\{\epsilon_n\}$ such that $\epsilon_n\to 0$ and $\Oc_r(b, \gamma_n(b)) \not\subset \mc S_{\epsilon_n}(\gamma_n)$ for all $n$. Equivalently, for each $n$ there exists 
$$
x_n \in \Oc_r(\gamma_n^{-1}(b), b) \smallsetminus \Big( \partial_\infty X - B_{\epsilon_n}(\gamma_n^{-1})\Big) = \Oc_r(\gamma_n^{-1}(b), b) \cap B_{\epsilon_n}(\gamma_n^{-1}). 
$$
Passing to a subsequence we can suppose that $x_n\rightarrow x$ and $\gamma_n^{-1} \rightarrow a$. Then by definition there exists a geodesic line $\ell \colon \Rb \rightarrow X$ where $\lim_{t \rightarrow \infty} \ell(t) = a$, $\lim_{t \rightarrow -\infty} \ell(t) = x$, and $\ell$ intersects the closed ball of radius $r$ centered at $b$. In particular, $a \neq x$ and hence $x_n \notin  B_{\epsilon_n}(\gamma_n^{-1})$ for $n$ sufficiently large. So we have a contradiction. 

(2): This is very similar to the proof of (1). 
\end{proof}


\section{The Shadow Lemma and its consequences}\label{sec:the shadow lemma}


In this section we establish a version of the classical Shadow Lemma. We then derive some of its immediate consequences.

\begin{theorem}[The Shadow Lemma]\label{thm:shadow lemma} Suppose $\Gamma \subset \mathsf{Homeo}(M)$ is a convergence group, $\sigma \colon \Gamma \times M \rightarrow \Rb$ is an expanding coarse-cocycle, and $\mu$ is a coarse $\sigma$-Patterson--Sullivan measure of dimension $\beta$.
For any sufficiently small $\epsilon>0$ there exists $C=C(\epsilon) >1$ such that 
$$
\frac{1}{C} e^{-\beta \norm{\gamma}_{\sigma}} \leq \mu\left( \mc S_\epsilon(\gamma) \right) \leq C e^{-\beta \norm{\gamma}_{\sigma}}
$$
for all $\gamma \in \Gamma$. 
\end{theorem} 

Using the results established in Sections~\ref{sec:expanding cocycles} and ~\ref{sec:basic properties of fake shadows}, the proof of the shadow lemma is essentially the 
same as Sullivan's original argument \cite[Prop.\,3]{Sullivan1979}.

\begin{lemma}\label{lem:big K has positive mass} For every $\eta > \sup_{x \in M} \mu(\{x\})$ there exists  $\epsilon>0$ such that
$$
 \mu\left(\gamma^{-1} \mc S_\epsilon(\gamma) \right) = \mu\left( M-B_\epsilon(\gamma) \right) \geq  1-\eta
$$
for all $\gamma \in \Gamma$. 
\end{lemma}

\begin{proof}
Otherwise there would exist $\{\gamma_n\}\subset\Gamma$ and $\{\epsilon_n\}$ such that $\epsilon_n\to 0$ and $\mu\left( M\cap B_{\epsilon_n}(\gamma_n) \right) \geq  \eta$ for all $n$. Passing to a subsequence we can suppose that $\gamma_n \rightarrow x \in \Gamma \sqcup M$. Let $\delta_n : = \dist(\gamma_n, x)$ and pass to a further subsequence so that $\{ \epsilon_n+\delta_n\}$ is decreasing. Then 
$$
\mu(\{x\}) = \lim_{n \rightarrow \infty} \mu\left( M\cap B_{\epsilon_n+\delta_n}(x) \right) \geq \eta,
$$
which contradicts our choice of $\eta$. 
\end{proof}

\begin{proof}[Proof of Theorem~\ref{thm:shadow lemma}] Notice that 
$$
\sup_{x \in M} \mu(\{x\}) < 1. 
$$
Otherwise, $\mu$ would be a Patterson--Sullivan measure supported on a single point, which is impossible since $\Gamma$ is non-elementary.
Hence by Lemma~\ref{lem:big K has positive mass}, there exists $\epsilon_0>0$ such that 
$$
\delta_0: = \inf_{\gamma \in \Gamma} \mu\Big( \gamma^{-1}S_{\epsilon_0}(\gamma) \Big)
$$
is positive.

Fix $\epsilon<\epsilon_0$.
By Proposition~\ref{prop:properties of fake shadows}\eqref{item:cocycle big on the fake shadow} there exists $C_1 > 1$ such that: if $\gamma \in \Gamma$, then 
$$
\frac{1}{C_1} e^{-\beta\norm{\gamma}_{\sigma} } \leq \frac{d \left(\gamma^{-1}\right)_*\mu}{d\mu} \leq C_1 e^{-\beta\norm{\gamma}_{\sigma} } 
$$
almost everywhere on $\gamma^{-1}\mc S_{\epsilon}(\gamma)$.

Fix $\gamma \in \Gamma$. Then 
\begin{align*}
\mu\big( \mc S_{\epsilon}(\gamma) \big) & =  \left(\gamma^{-1}\right)_*\mu\big(\gamma^{-1}\mc S_{\epsilon}(\gamma) \big)  = \int_{\gamma^{-1}\mc S_{\epsilon}(\gamma)} \frac{d \left(\gamma^{-1}\right)_*\mu}{d\mu} d\mu.
\end{align*}
Hence 
\begin{equation*}
\frac{\delta_0}{C_1}e^{-\beta\norm{\gamma}_{\sigma} }  \leq \mu\big( \mc S_{\epsilon}(\gamma) \big) \leq C_1 e^{-\beta\norm{\gamma}_{\sigma} }. \qedhere
\end{equation*}

\end{proof}

The following results now also follow from the standard arguments from the classical case.

\begin{proposition}\label{prop:some consequences of the shadow lemma} Suppose $\Gamma \subset \mathsf{Homeo}(M)$ is a convergence group, $\sigma \colon \Gamma \times M \rightarrow \Rb$ is an expanding coarse-cocycle, and $\mu$ is a coarse $\sigma$-Patterson--Sullivan measure of dimension $\beta$. Then:
\begin{enumerate} 
\item \label{item:dim>=critexp} $\beta \geq \delta_\sigma(\Gamma)>0$.
\item If $y \in M$ is a conical limit point, then $\mu(\{y\}) = 0$.
\item \label{item:convergent => conical limset null} If 
$$
\sum_{\gamma \in \Gamma} e^{-\beta \norm{\gamma}_\sigma} < + \infty,
$$
(e.g.\  if  $\beta > \delta_\sigma(\Gamma)$)
then $\mu(\Lambda^{\rm con}(\Gamma))= 0$. 
\item There exists $C>0$ such that 
 $$\#\{\gamma\in\Gamma: \norm\gamma_\sigma\leq R\} \leq Ce^{\delta_\sigma(\Gamma) R}$$
 for any $R>0$.
\end{enumerate}

\end{proposition}

The proof of the proposition requires a lemma. 

\begin{lemma}\label{lem:upper counting bound}
 Then there exists $C>0$ such that 
 $$\#\{\gamma\in\Gamma: \norm\gamma_\sigma\leq R\} \leq Ce^{\beta R}$$
 for any $R>0$.
\end{lemma}

\begin{proof}
 By the Shadow Lemma (Theorem~\ref{thm:shadow lemma}) there exist $\epsilon>0$ and $C_1>1$ such that
\begin{equation*}
  \mu(\mc S_{\epsilon}(\gamma)) \geq C_1^{-1} e^{-\beta \norm\gamma_\sigma}
 \end{equation*}
 for all $\gamma \in \Gamma$. By Proposition~\ref{prop:properties of fake shadows}\eqref{item:multiplicative property of intersecting fake shadows}, there exists $C_2$ such that: if $\gamma,\gamma'\in \Gamma$, 
 $$
 \abs{\norm\gamma_\sigma - \norm{\gamma'}_\sigma} \leq 1,
 $$
 and $\mc S_\epsilon(\gamma)\cap\mc S_\epsilon(\gamma')\neq\emptyset$, then
 $\norm{\gamma^{-1}\gamma'}_\sigma \leq C_2.$
 Let
 $$C_3 := \#\{\gamma\in\Gamma: \norm\gamma_\sigma \leq C_2\}$$
 (which is finite by Proposition~\ref{prop:basic properties}\eqref{item:properness}).
 Then, for all $x\in M$ and $R>0$,
 \begin{equation*}
  \#\{\gamma\in \Gamma: R-1\leq \norm\gamma_\sigma\leq R \text{ and } x\in\mc S_\epsilon(\gamma)\} \leq C_3.
 \end{equation*}
 Then
  \begin{align*}
  \#\{\gamma\in \Gamma: R-1\leq \norm\gamma_\sigma & \leq R\} =  \sum_{\substack{\gamma\in\Gamma \\ R-1\leq \norm\gamma_\sigma\leq R}} 1 \leq C_1 e^{\beta R} \sum_{\substack{\gamma\in\Gamma \\ R-1\leq \norm\gamma_\sigma\leq R}}\mu(\mc S_\epsilon(\gamma)) \\
  & \leq C_1C_3 \mu(M) e^{\beta R}= C_1 C_3 e^{\beta R}.
 \end{align*}
We complete the proof by summing this inequality over $\mathbb N$.
\end{proof}

\begin{proof}[Proof of (1)] This follows immediately from  Lemma~\ref{lem:upper counting bound} and the definition of the critical exponent $\delta_\sigma(\Gamma)$.
\end{proof} 

\begin{proof}[Proof of (2)] Suppose $y$ is a conical limit point. 
By Equation \eqref{eqn:decomposition of conical limit set} and Lemma~\ref{lem:shadows versus uniform conical limit sets}, there exist $\epsilon>0$ and an escaping sequence $\{\gamma_n\} \subset \Gamma$ such that $y \in \mc S_{\epsilon}(\gamma_n)$ for all $n$. 
Hence, by the Shadow Lemma (Theorem~\ref{thm:shadow lemma}), there exists $C > 0$ such that
$$
\mu(\{y\}) \leq \mu(\mc S_{\epsilon}(\gamma_n)) \leq C e^{-\beta \norm{\gamma_n}_\sigma}
$$
for all $n$. Since $\sigma$ is expanding, Proposition~\ref{prop:basic properties}\eqref{item:properness} implies that $\norm{\gamma_n}_\sigma \to +\infty$. Hence $\mu(\{y\}) = 0$ since $\beta>0$ by part (1).
\end{proof}

\begin{proof}[Proof of (3)] 
By Lemma \ref{lem:shadows versus uniform conical limit sets}  for every $m_0 > 0$ we have 
$$
\Lambda^{\rm con}(\Gamma) \subset \bigcup_{m \geq m_0} \bigcap_{n \geq 1} \bigcup_{\norm{\gamma}_\sigma \geq n}  \mc S_{1/m}(\gamma).
$$
By the Shadow Lemma (Theorem~\ref{thm:shadow lemma}),  for all $m$ sufficiently large there exists $C_m > 0$ such that 
$$
 \mu(\mc S_{1/m}(\gamma)) \leq C_m e^{-\beta \norm{\gamma}_\sigma}
$$
for all $\gamma \in \Gamma$. Hence for all $m$ sufficiently large, 
$$
\mu \left(  \bigcap_{n \geq 1} \bigcup_{\norm{\gamma}_\sigma \geq n} \mc S_{1/m}(\gamma) \right) \leq \lim_{n \rightarrow \infty} \sum_{\norm{\gamma}_\sigma \geq n} C_m e^{-\beta \norm{\gamma}_\sigma}
$$
which equals zero by assumption. Thus $\mu( \Lambda^{\rm con}(\Gamma) ) = 0$. 
\end{proof}

\begin{proof}[Proof of (4)] By Theorem~\ref{thm:PS measures exist} there exists a Patterson--Sullivan measure with dimension $\delta_\sigma(\Gamma)$. Then part (4) follows immediately from applying Lemma~\ref{lem:upper counting bound} to this measure. 
\end{proof}


\part{Dynamics of Patterson--Sullivan measures}


\section{Conical limit points have full measure in the divergent case}\label{sec:conical}


In this section we show that any Patterson--Sullivan measure with dimension equal to the critical exponent is supported on the conical limit set when the associated Poincar\'e series diverges at its critical exponent. 
The proof is similar to Roblin's argument  for the analogous result for Busemann cocycles in ${\rm CAT}(-1)$ spaces~\cite{roblin}, in that we use a variant of the Borel--Cantelli Lemma. 
However, we use a different variant of the lemma and apply it to a different collection of sets. This approach seems slightly simpler and was also used in~\cite{CZZ2023a}.

\begin{proposition}\label{prop:uniform conical has full measure} Suppose $\Gamma \subset \mathsf{Homeo}(M)$ is a convergence group and $\sigma \colon \Gamma \times M \rightarrow \Rb$ is an expanding coarse-cocycle with $\delta:=\delta_\sigma(\Gamma) < +\infty$. If $\mu$ is a coarse $\sigma$-Patterson--Sullivan measures of dimension $\delta$ and
$$
\sum_{\gamma \in \Gamma} e^{-\delta \norm{\gamma}_\sigma} = + \infty,
$$
 then $\mu\left( \Lambda^{\rm con}(\Gamma) \right) = 1$.
\end{proposition} 

\begin{proof}
We first show that $\mu(\Lambda^{\rm con}(\Gamma)) > 0$. 
To accomplish this we use the following variant of the Borel--Cantelli Lemma.

\begin{lemma}[Kochen--Stone Borel--Cantelli Lemma, \cite{KochenStone}]\label{lem: KS BC lemma}
Let $(X,\nu)$ be a finite measure space. If $\{ A_n\}$ is a sequence of measurable sets where 
$$
\sum_{n =1}^\infty \nu(A_n) = +\infty \quad \text{and} \quad \liminf_{N \rightarrow\infty} \frac{ \sum_{1 \le m,n \le N} \nu(A_n \cap A_m)}{(\sum_{n=1}^N \nu(A_n))^2} < +\infty,
$$
then  
$$
\nu\left( \{ x \in M : x \text{ is contained in infinitely many of } A_1, A_2, \dots \}\right) > 0. 
$$
\end{lemma} 

Fix a compatible metric $\dist$ on $\Gamma \sqcup M$, and for $\epsilon > 0$ and $\gamma \in \Gamma$ let $\mc S_\epsilon(\gamma) \subset M$ denote the shadow defined in Section~\ref{sec:basic properties of fake shadows}. Using the Shadow Lemma (Theorem~\ref{thm:shadow lemma}), fix $\epsilon>0$ and a constant $C_1 > 1$ such that 
$$
\frac{1}{C_1} e^{-\delta \norm{\gamma}_\sigma} \leq \mu\Big( \mc S_\epsilon(\gamma) \Big) \leq C_1 e^{-\delta \norm{\gamma}_\sigma}
$$
for all $\gamma \in \Gamma$. Next fix an enumeration $\Gamma = \{ \gamma_n\}$ such that 
$$
\norm{\gamma_1}_\sigma \leq \norm{\gamma_2}_\sigma \leq \dots 
$$
(this is possible by Proposition~\ref{prop:basic properties}\eqref{item:properness}) and let
$$
A_n :=\mc S_\epsilon(\gamma_n).
$$
We will show that the sets $\{A_n\}$ satisfy the hypothesis of the Kochen--Stone lemma. 

One part is easy: by assumption
$$
\sum_{n=1}^\infty \mu(A_n) \geq \frac{1}{C_1} \sum_{\gamma \in \Gamma} e^{-\delta \norm{\gamma}_\sigma} = +\infty.
$$

The other part is only slightly more involved. Using Proposition~\ref{prop:properties of fake shadows}\eqref{item:multiplicative property of intersecting fake shadows} there exists $C_2 > 0$ such that: if $1 \leq n \leq m$ and $A_n \cap A_m \neq \emptyset$, then 
$$
\norm{\gamma_n}_\sigma +\norm{\gamma_n^{-1}\gamma_m}_\sigma \leq \norm{\gamma_m}_\sigma+C_2. 
$$
Hence, in this case, $\norm{\gamma_n^{-1}\gamma_m}_\sigma \leq \norm{\gamma_m}_\sigma+C_2$ and 
$$
\mu( A_n \cap A_m) \leq \mu(A_m) \leq C_1 e^{-\delta \norm{\gamma_m}_\sigma}\leq  C_3e^{-\delta \norm{\gamma_n}_\sigma}e^{-\delta \norm{\gamma_n^{-1}\gamma_m}_\sigma}
$$
where $C_3 := C_1e^{\delta C_2}$. So, if $f(N) := \max\{ n : \norm{\gamma_n} \leq \norm{\gamma_N}+C_2\}$, then
\begin{align*}
\sum_{m,n=1}^N \mu(A_n \cap A_m) & \leq 2 \sum_{1 \leq n \leq m \leq N} \mu(A_n \cap A_m) \leq 2C_3  \sum_{1 \leq n \leq m \leq N} e^{-\delta \norm{\gamma_n}_\sigma}e^{-\delta \norm{\gamma_n^{-1}\gamma_m}_\sigma} \\
& \leq 2C_3 \sum_{n=1}^N e^{-\delta \norm{\gamma_n}_\sigma} \sum_{n=1}^{f(N)}  e^{-\delta \norm{\gamma_n}_\sigma}.
\end{align*}
Thus to apply the Kochen--Stone lemma, it suffices to observe the following. 

\begin{lemma} There exists $C_4 > 0$ such that: 
$$
\sum_{n=1}^{f(N)}  e^{-\delta \norm{\gamma_n}_\sigma} \leq C_4 \sum_{n=1}^N e^{-\delta \norm{\gamma_n}_\sigma}
$$
for all $N \geq 1$. 
\end{lemma}

\begin{proof} Notice if $N < n \leq m \leq f(N)$ and $A_n \cap A_m \neq \emptyset$, then 
$$
\norm{\gamma_n^{-1}\gamma_m}_\sigma \leq \norm{\gamma_m}_\sigma-\norm{\gamma_n}_\sigma +C_2 \leq 2C_2.  
$$
So if $D:=\#\{ \gamma \in \Gamma: \norm{\gamma}_\sigma \leq 2C_2\}$, then 
$$
\sum_{n=N+1}^{f(N)}  e^{-\delta \norm{\gamma_n}_\sigma} \leq C_1\sum_{n=N+1}^{f(N)}  \mu(A_n) \leq C_1 D \mu\left( \bigcup_{n=N+1}^{f(N)} A_n \right) \leq C_1 D. 
$$
Hence 
\begin{equation*}
 \sum_{n=1}^{f(N)}  e^{-\delta \norm{\gamma_n}_\sigma} \leq \left(1+ C_1D e^{\delta \norm{\gamma_1}_\sigma} \right) \sum_{n=1}^N e^{-\delta \norm{\gamma_n}_\sigma}. \qedhere
\end{equation*}
\end{proof} 

So by the Kochen--Stone lemma the set 
$$
X:=\{ x \in M : x \text{ is contained in infinitely many of } A_1, A_2, \dots \}
$$
has positive $\mu$-measure. By Lemma~\ref{lem:shadows versus uniform conical limit sets}, $X \subset \Lambda^{\rm con}(\Gamma)$.
Hence $\mu(\Lambda^{\rm con}(\Gamma)) > 0$.

Now suppose for a contradiction that $\mu(\Lambda^{\rm con}(\Gamma)) <1 $. 
Let $Y : = M -\Lambda^{\rm con}(\Gamma)$ and define a measure $\hat{\mu}$ on $M$ by 
$$
\hat{\mu}(\cdot) = \frac{1}{\mu(Y)} \mu(Y \cap \cdot).
$$ 
This is also a coarse $\sigma$-Patterson--Sullivan measure of dimension $\delta$ and so the argument above implies that 
$$
0 < \hat{\mu}( \Lambda^{\rm con}(\Gamma)) = \mu( Y \cap \Lambda^{\rm con}(\Gamma)) =0.
$$ 
So we have a contradiction. 
\end{proof}


\section{Ergodicity and  uniqueness of Patterson--Sullivan measures}\label{sec:uniqueness}


In this section we establish uniqueness and ergodicity of Patterson--Sullivan measures in the divergent case. Our argument is similar to the proof of statement (g) in~\cite[p.\,22]{roblin}.

For the rest of the section, suppose $\Gamma \subset \mathsf{Homeo}(M)$ is a convergence group and $\sigma \colon \Gamma \times M \rightarrow \Rb$ is an expanding coarse-cocycle with $\delta:=\delta_\sigma(\Gamma) < +\infty$.

\begin{theorem}\label{thm:ergodicity on single M} If $\mu$ is a $C$-coarse $\sigma$-Patterson--Sullivan measure of dimension $\delta$ and
$$
\sum_{\gamma \in \Gamma} e^{-\delta \norm{\gamma}_\sigma} = + \infty,
$$
 then:
 \begin{enumerate}
 \item $\Gamma$ acts ergodically on $(M, \mu)$.
 \item $\mu$ is coarsely unique in the following sense: if $\lambda$ is a $C$-coarse $\sigma$-Patterson--Sullivan measure of dimension $\delta$, then $e^{-4C} \mu \leq \lambda \leq e^{4C}\mu$.
  \item $\mu(\Lambda_\epsilon^{\rm con}(\Gamma)) = 1$ when $\epsilon >0$ is sufficiently small (recall that $\Lambda_\epsilon^{\rm con}(\Gamma)$ was defined in Definition~\ref{defn:of uniform conical points}).  \end{enumerate} 
 \end{theorem}

The rest of the section is devoted to the proof of Theorem~\ref{thm:ergodicity on single M}. We will prove that $\Gamma$ acts ergodically on $(M, \mu)$ and then use ergodicity to deduce the other claims. To prove ergodicity we will first establish a version of the Lebesgue differentiation theorem (as in ~\cite[p.\,27, Lem.\,2]{roblin}, see also ~\cite[Sublemma 8.7]{DK2022}). 

Fix a compatible metric $\dist$ on $\Gamma \sqcup M$.

\begin{lemma}\label{lem:Leb diff theorem} 
Suppose $\epsilon_0>0$  satisfies the Shadow Lemma (Theorem~\ref{thm:shadow lemma}). If $f \in L^1(M, \mu)$, then for 
$\mu$-almost every $x \in M$ we have
$$
f(x) = \lim_{n \rightarrow \infty} \frac{1}{\mu(\mc S_\epsilon(\gamma_n))} \int_{\mc S_\epsilon(\gamma_n)} f(y) d\mu(y)
$$
for every $0 < \epsilon \leq \epsilon_0$ and escaping sequence $\{\gamma_n\} \subset \Gamma$ with 
$$
x \in \bigcap_{n \geq 1} \mc S_\epsilon(\gamma_n).
$$
\end{lemma} 

\begin{proof} Using Proposition~\ref{prop:properties of fake shadows}\eqref{item:Vish covering lemma}, the proof is very similar to the proof of the Lebesgue differentiation theorem, see e.g. ~\cite[Th.\,3.18]{Folland1999}.

Let $\epsilon_j : =\epsilon_0/j$. For $f \in L^1(M, \mu)$ and $j \geq 1$, define $A_jf, B_jf \colon M \rightarrow [0,+\infty]$ by setting 
$$
A_jf(x) = \lim_{R\to\infty}\sup_{\substack{\norm\gamma_\sigma\geq R\\ x\in\mc S_{\epsilon_j}(\gamma)}} \frac{1}{\mu(\mc S_{\epsilon_j}(\gamma))} \int_{\mc S_{\epsilon_j}(\gamma)} \abs{f(y)-f(x)} d\mu(y)$$
and
$$B_jf(x) =  \lim_{R\to\infty}\sup_{\substack{\norm\gamma_\sigma\geq R\\ x\in\mc S_{\epsilon_j}(\gamma)}} \frac{1}{\mu(\mc S_{\epsilon_j}(\gamma))} \int_{\mc S_{\epsilon_j}(\gamma)} \abs{f(y)} d\mu(y) 
$$
if $x\in\Lambda^{\rm con}_{2\epsilon_j}(\Gamma)$ and $A_jf(x)=B_jf(x)=0$ otherwise. Notice that if $x \in \Lambda^{\rm con}_{2\epsilon_0}(\Gamma)$, then Lemma~\ref{lem:shadows versus uniform conical limit sets} implies that the supremums above involve non-empty sets. Thus the functions $A_jf$, $B_jf$ are indeed non-negative. 

Now fix $f \in L^1(M, \mu)$. We claim that $A_jf(x) = 0$ for $\mu$-almost every $x \in M$. To show this it suffices to fix $\alpha > 0$ and show that 
$$
\mu(\{x \in M : A_jf(x) > \alpha\})=0. 
$$  

Fix $\eta > 0$ and let $g$ be a continuous function with 
$$
\int_M \abs{f-g} d\mu < \eta.
$$
Then 
$$
0 \leq A_jf(x) \leq B_j(f-g)(x) +\abs{f(x)-g(x)} + A_jg(x).
$$
Since $g$ is continuous, Proposition~\ref{prop:properties of fake shadows}\eqref{item:collapsing fake shadows} implies that $A_jg(x) = 0$. Hence 
$$
\{ x \in M : A_jf(x) > \alpha\} \subset N_j \cup \{ x \in M : \abs{f(x)-g(x)} > \alpha/2\}
$$
where 
$$N_j : = \{ x \in M : B_j(f-g)(x) > \alpha/2\}.
$$
The measure of the second set is easy to bound: 
\begin{equation}\label{eqn:Chevy's bound}
\mu( \{ x \in M : \abs{f(x)-g(x)} > \alpha/2\} ) \leq \frac{2}{\alpha} \int_M \abs{f-g} d \mu <  \frac{2\eta}{\alpha}.
\end{equation} 
To bound the measure of the first set, notice that for every $x \in N_j$, there exists $\gamma_x \in \Gamma$ such that $x \in \mc S_{\epsilon_j}(\gamma_x)$ and 
$$
\mu(\mc S_{\epsilon_j}(\gamma_x) )\leq \frac{2}{\alpha} \int_{\mc S_{\epsilon_j}(\gamma_x)} \abs{f(y)-g(y)} d\mu(y).
$$
Using Proposition~\ref{prop:properties of fake shadows}\eqref{item:Vish covering lemma}, we can find $N_j' \subset N_j$ and $\epsilon_j'<\epsilon_j$ such that the shadows 
$\{ \mc S_{\epsilon_j}(\gamma_x) : x \in N_j'\}$ are disjoint (which implies $N_j'$ is countable) and 
$$
N_j \subset \bigcup_{x \in N_j} \mc S_{\epsilon_j}(\gamma_x) \subset \bigcup_{x \in N_j'} \mc S_{\epsilon_j'}(\gamma_x).
$$
Applying the Shadow Lemma (Theorem~\ref{thm:shadow lemma}), there exists a constant $C_j > 1$ such that 
$$
\mu(  \mc S_{\epsilon_j'}(\gamma)) \leq C_j\mu( \mc S_{\epsilon_j}(\gamma))
$$
for all $\gamma \in \Gamma$. Then 
\begin{align*}
\mu(N_j) & \leq \sum_{x \in N_j'} \mu(  \mc S_{\epsilon_j'}(\gamma_x)) \leq C_j \sum_{x \in N_j'} \mu(  \mc S_{\epsilon_j}(\gamma_x)) \\
& \leq \frac{2C_j}{\alpha} \sum_{ x \in N_j'}  \int_{ \mc S_{\epsilon_j}(\gamma_x)} \abs{f(y)-g(y)} d\mu(y) \\
& = \frac{2C_j}{\alpha} \int_{\bigcup_{x \in N_j'}  \mc S_{\epsilon_j}(\gamma_x)} \abs{f(y)-g(y)} d\mu(y) \leq \frac{2C_j\eta}{\alpha}. 
\end{align*}
Then using Equation~\eqref{eqn:Chevy's bound},
$$
\mu(\{x \in M : A_jf(x) > \alpha\}) \leq \frac{2(1+C_j)\eta}{\alpha}.
$$
Since $\eta > 0$ was arbitrary, we have $\mu(\{x \in M : A_jf(x) > \alpha\}) =0$. Since $\alpha > 0$ was arbitrary, we see that $A_jf(x) = 0$ for $\mu$-almost every $x \in M$. 

Next we show that the full $\mu$-measure set $\bigcap_{j \geq 1}\{ x : A_jf(x) = 0\}\cap \Lambda^{\rm con}(\Gamma)$ satisfies the lemma.
To that end, fix $x\in \Lambda^{\rm con}(\Gamma)$ with $A_jf(x) = 0$ for all $j \geq 1$, $\epsilon \in (0,\epsilon_0]$ and an escaping sequence $\{\gamma_n\}$ where 
$$
x \in  \bigcap_{n \geq 1} \mc S_\epsilon(\gamma_n).
$$
Fix $j \geq 1$ such that $\epsilon_j < \epsilon$, so $x\in \Lambda^{\rm con}_{2\epsilon_j}(\Gamma)$. 
Then $\mc S_{\epsilon_j}(\gamma) \supset \mc S_{\epsilon}(\gamma)$ for any $\gamma\in\Gamma$. 
By the Shadow Lemma (Theorem~\ref{thm:shadow lemma}), there exists a constant $c > 0$ such that 
$$
\mu(  \mc S_{\epsilon}(\gamma)) \geq c \mu( \mc S_{\epsilon_j}(\gamma))
$$
for all $\gamma \in \Gamma$. Then 
\begin{align*}
\limsup_{n \rightarrow \infty} & \abs{ f(x) - \frac{1}{\mu(\mc S_\epsilon(\gamma_n))} \int_{\mc S_\epsilon(\gamma_n)} f(y) d\mu(y)} \\
& \leq \limsup_{n \rightarrow \infty} \frac{1}{\mu(\mc S_\epsilon(\gamma_n))} \int_{\mc S_\epsilon(\gamma_n)} \abs{f(x)-f(y)} d\mu(y) \\ 
&\leq  \limsup_{n \rightarrow \infty} \frac{c^{-1}}{\mu(\mc S_{\epsilon_j}(\gamma_n))} \int_{\mc S_{\epsilon_j}(\gamma_n)} \abs{f(x)-f(y)} d\mu(y) \\
& \leq c^{-1} A_jf(x) = 0. \qedhere
\end{align*}

\end{proof} 

Next we use Lemma~\ref{lem:Leb diff theorem} to prove the following. 

\begin{lemma}\label{lem:zooming in on density points} Suppose $\epsilon_0>0$  satisfies the Shadow Lemma (Theorem~\ref{thm:shadow lemma}).
If $E \subset M$ is measurable, then for 
$\mu$-almost every $x \in E$ we have
$$
1 = \lim_{n \rightarrow \infty} \mu(\gamma_n^{-1}E)
$$
for every $0 < \epsilon \leq \epsilon_0$ and  escaping sequence $\{\gamma_n\} \subset \Gamma$ with 
$$
x \in \bigcap_{n \geq 1} \mc S_\epsilon(\gamma_n).
$$
\end{lemma} 

\begin{proof} Applying Lemma~\ref{lem:Leb diff theorem} to $1_E$, there is a full $\mu$-measure set $N \subset M$ such that 
$$
1 =\lim_{n \rightarrow \infty} \frac{1}{\mu(\mc S_\epsilon(\gamma_n))} \int_{\mc S_\epsilon(\gamma_n)} 1_E(y) d\mu(y)= \lim_{n \rightarrow \infty} \frac{\mu(E \cap \mc S_{\epsilon}(\gamma_n)) }{\mu(\mc S_{\epsilon}(\gamma_n))}
$$
whenever $x \in N \cap E$,  $0 < \epsilon \leq \epsilon_0$ and $\{\gamma_n\} \subset \Gamma$ is an escaping sequence with 
$$
x \in \bigcap_{n \geq 1} \mc S_{\epsilon}(\gamma_n).
$$

We claim that the full $\mu$-measure set $N$ satisfies the lemma. To that end, fix $x \in E \cap N$, $0 < \epsilon \leq \epsilon_0$ and  escaping sequence $\{\gamma_n\} \subset \Gamma$ with 
$$
x \in \bigcap_{n \geq 1} \mc S_\epsilon(\gamma_n).
$$

Notice that 
$$
x \in \bigcap_{n \geq 1} \mc S_{\epsilon/j}(\gamma_n).
$$
for all $j \geq 1$, since   $\mc S_{\epsilon}(\gamma) \subset \mc S_{\epsilon/j}(\gamma)$. So we have 
$$
1 = \lim_{n \rightarrow \infty} \frac{\mu(E \cap \mc S_{\epsilon/j}(\gamma_n)) }{\mu(\mc S_{\epsilon/j}(\gamma_n))}
$$
for every $j \geq 1$. Now
\begin{align*}
\mu & (E \cap \mc S_{\epsilon/j}(\gamma_n))  = \mu(\mc S_{\epsilon/j}(\gamma_n))-\mu(E^c \cap \mc S_{\epsilon/j}(\gamma_n)) \\
& =  \mu(\mc S_{\epsilon/j}(\gamma_n))-(\gamma_n^{-1})_*\mu(\gamma_n^{-1}E^c \cap \gamma_n^{-1}\mc S_{\epsilon/j}(\gamma_n)). 
\end{align*}
Hence 
\begin{align*}
0 = \lim_{n \rightarrow \infty} \frac{ (\gamma_n^{-1})_*\mu(\gamma_n^{-1}E^c \cap \gamma_n^{-1}\mc S_{\epsilon/j}(\gamma_n))}{(\gamma_n^{-1})_*\mu(\gamma_n^{-1}\mc S_{\epsilon/j}(\gamma_n))}.
\end{align*}
By Proposition~\ref{prop:properties of fake shadows}\eqref{item:cocycle big on the fake shadow} , there exists $C_j > 1$ (independent of $n$) such that 
$$
\frac{1}{C_j} e^{-\norm{\gamma_n}_\sigma} \leq \frac{d(\gamma_n^{-1})_*\mu}{d\mu} \leq C_j e^{-\norm{\gamma_n}_\sigma} 
$$
almost everywhere on $ \gamma_n^{-1}\mc S_{\epsilon/j}(\gamma_n)$. Thus
\begin{align*}
0 = \lim_{n \rightarrow \infty} \frac{ \mu(\gamma_n^{-1}E^c \cap \gamma_n^{-1}\mc S_{\epsilon/j}(\gamma_n))}{\mu(\gamma_n^{-1}\mc S_{\epsilon/j}(\gamma_n))}.
\end{align*}
Recall that $\gamma_n^{-1}\mc S_{\epsilon/j}(\gamma_n) = M - B_{\epsilon/j}(\gamma_n)$. Further, by Proposition~\ref{prop:uniform conical has full measure} and Proposition~\ref{prop:some consequences of the shadow lemma}, $\mu$ has no atoms. Hence 
$$
\lim_{j \rightarrow \infty} \inf_{n \geq 1} \mu\left( \gamma_n^{-1}\mc S_{\epsilon/j}(\gamma_n)) \right) = 1.
$$
Thus $\mu(\gamma_n^{-1}E^c) \rightarrow 0$, which implies that $\mu(\gamma_n^{-1}E) \rightarrow 1$. 
\end{proof} 

Now we are ready to prove the three assertions in Theorem~\ref{thm:ergodicity on single M}.

\begin{lemma}  $\Gamma$ acts ergodically on $(M, \mu)$. \end{lemma} 

\begin{proof} Lemma~\ref{lem:zooming in on density points} implies that any  $\Gamma$-invariant set  with positive $\mu$-measure has full measure. 
\end{proof} 

\begin{lemma}  If $\lambda$ is a $C$-coarse $\sigma$-Patterson--Sullivan measure of dimension $\delta$, then $e^{-4C} \mu \leq \lambda \leq e^{4C}\mu$.
\end{lemma} 

\begin{proof} For any $t \in [0,1]$ the measure $\mu_t := (1-t)\mu+t\lambda$ is also a $C$-coarse $\sigma$-Patterson--Sullivan measure of dimension $\delta$. 
Indeed, for any $\gamma\in\Gamma$, letting $f(x)=e^{-\delta\sigma(\gamma^{-1},x)}$ we have $C^{-1}f\mu\leq\gamma_*\mu\leq Cf\mu$ and $C^{-1}f\lambda\leq\gamma_*\lambda\leq Cf\lambda$, so 
\[\gamma_*\mu_t=\gamma_*((1-t)\mu+t\lambda)=(1-t)\gamma_*\mu+t\gamma_*\lambda\leq (1-t)Cf\mu+tCf\lambda=Cf\mu_t,\]
and similarly $\gamma_*\mu_t\geq C^{-1}f\mu_t$.

Fix $s,t \in (0,1)$. Then the measures $\mu_s$ and $\mu_t$ are absolutely continuous. Since $\mu_t$ and $\mu_s$ are both coarse Patterson--Sullivan measures of the same dimension, the Radon-Nikodym derivative $\frac{d \mu_t}{d\mu_s}$ is coarsely $\Gamma$-invariant, more precisely: for any $\gamma \in \Gamma$ we have
$$
e^{-2C} \frac{d \mu_t}{d\mu_s} \leq \frac{d \mu_t}{d\mu_s} \circ \gamma \leq e^{2C} \frac{d \mu_t}{d\mu_s}
$$
$\mu_s$-almost everywhere. 

Next fix $\epsilon_j \searrow 0$. Then for each $j$ there exists $r_j \in \Rb$ such that the set $A_j:=\{ r_j \leq \frac{d \mu_t}{d\mu_s} \leq r_j + \epsilon_j\}$ has positive $\mu_s$-measure. Then $\Gamma \cdot A_j$ is $\Gamma$-invariant and hence, by ergodicity, must have full measure. Further,  
$$
\Gamma \cdot A_j \subset \left\{ e^{-2 C}r_j \leq \frac{d \mu_t}{d\mu_s} \leq e^{2C} r_j +e^{2C} \epsilon_j\right\}
$$
and so 
$$
e^{-2C}r_jd\mu_s \leq d \mu_t \leq (e^{2C} r_j +e^{2C}\epsilon_j )d\mu_s.
$$
Since $\mu_t$ and $\mu_s$ are both probability measures, we must have 
$$
e^{-2C}r_j \leq 1 \quad \text{and} \quad e^{2C} r_j +e^{2C} \epsilon_j \geq 1
$$
for all $j$. Thus any limit point of $\{r_j\}$ is in $[e^{-2C}, e^{2C}]$, which implies that 
$$
e^{-4C}d\mu_s \leq d \mu_t \leq e^{4C} d\mu_s.
$$

Since $s,t \in (0,1)$ were arbitrary, we then see that $e^{-4C} \mu \leq \lambda \leq e^{4C}\mu$.
\end{proof}

\begin{lemma} $\mu(\Lambda_\epsilon^{\rm con}(\Gamma)) = 1$ when $\epsilon >0$ is sufficiently small. \end{lemma} 

\begin{proof} Proposition~\ref{prop:uniform conical has full measure} implies that \hbox{$\mu(\Lambda^{\rm con}(\Gamma)) = 1$}. Since $\Lambda^{\rm con}(\Gamma)=\bigcup_{\epsilon>0}\Lambda^{\rm con}_\epsilon(\Gamma)$, this implies that  \hbox{$\mu(\Lambda_\epsilon^{\rm con}(\Gamma)) > 0$} when $\epsilon>0$ is sufficiently small. 
By Observation~\ref{obs:invariance of uniform conical limit sets}, the set $\Lambda_\epsilon^{\rm con}(\Gamma)$ is $\Gamma$-invariant. 
Hence, by ergodicity, $\mu(\Lambda_\epsilon^{\rm con}(\Gamma)) = 1$ for all sufficiently small $\epsilon>0$.
\end{proof}

\section{BMS measures on $M^{(2)}$, conservativity and dissipativity}\label{sec: action on M2}

Suppose $\Gamma\subset \mathsf{Homeo}(M)$ is a convergence group and, as before, let 
$$
M^{(2)} := \{ (x,y) \in M^2 : x \neq y\}.
$$
In this section we study the action of $\Gamma$ on $M^{(2)}$. 

\subsection{BMS measures}\label{sec:BMS on M2}

We first observe that a coarse GPS system can be used to produce a $\Gamma$-invariant measure on $M^{(2)}$. To that end, suppose $(\sigma,\bar\sigma,G)$ is a coarse GPS system, and that  $\mu$ and $\bar\mu$ are 
coarse Patterson--Sullivan measures of dimension $\delta\geq 0$ for $\sigma$ and $\bar\sigma$ respectively.

We use a lemma from~\cite{BF2017} to show that $\bar\mu \otimes \mu$ can be scaled to become $\Gamma$-invariant. Note that this lemma is unnecessary in the continuous case, i.e.\ when $\kappa=0$ and $G$ is continuous in Definition~\ref{def:GPS}.

\begin{lemma}\label{lem:finding the good G} There exists a Borel measurable function $\tilde{G} \colon M^{(2)} \to \Rb$ such that $(\sigma, \bar \sigma, \tilde G)$ is a coarse GPS system and the measure 
\[\nu : = e^{\delta \tilde{G}}\bar\mu \otimes \mu|_{M^{(2)}}.\]
on $M^{(2)}$ is locally finite and $\Gamma$-invariant. 
We call $\nu$ a \emph{Bowen--Margulis--Sullivan (BMS) measure of dimension $\delta$} on $M^{(2)}$ associated to $(\sigma,\bar\sigma,G,\mu,\bar\mu)$. 
\end{lemma} 

\begin{proof} Define $H \colon M^{(2)} \to \Rb$ by  
$$
H(x,y) = \limsup_{p \rightarrow x, q \rightarrow y} G(p,q).
$$
Since $(\sigma, \bar \sigma, G)$ is a coarse GPS system, we see that $(\sigma, \bar \sigma, H)$ is a coarse GPS system. 
By construction $H$ is upper semicontinuous and hence Borel measurable (while $G$ may not be).

Let $\nu_0 : = e^{\delta H} \bar\mu \otimes \mu$ and
$$
\rho(\gamma, x, y) := -\frac{1}{\delta} \log \frac{d\gamma^{-1}_*\nu_0}{d\nu_0}(x,y). 
$$
By uniqueness of Radon--Nikodym derivatives, there is a full $(\bar\mu\otimes\mu)$-measure $\Gamma$-invariant Borel measurable subset $E\subset M^{(2)}$ such that $\rho(\gamma, x, y)$ is defined for all $\gamma\in\Gamma$ and $(x,y)\in E$,
and $\rho(\gamma\gamma',x,y)=\rho(\gamma,\gamma'x,\gamma'y)+\rho(\gamma',x,y)$ for any additional $\gamma'\in\Gamma$.
We extend $\rho$ to a cocycle on the whole set $M^{(2)}$ by setting it to zero on the complement of $E$. Further,  since $(\sigma,\bar\sigma, H)$ is a coarse GPS system and $\mu$ and $\bar\mu$ are coarse Patterson--Sullivan 
measures for $\sigma$ and  $\bar\sigma$, one may check that $\rho$ is bounded on a set of full measure. So up to changing $\rho$ on a null measure set we may assume that
$$
\sup_{\gamma \in \Gamma, (x,y) \in M^{(2)}} \abs{\rho(\gamma,x,y)} < +\infty.
$$
By~\cite[Lem.\,3.4]{BF2017} there exists a bounded Borel function $\phi \colon M^{(2)} \to \Rb$ such that
\begin{equation}\label{eqn:property of rho when defining BMS}
\rho(\gamma, x,y) = \phi(\gamma x,\gamma y) - \phi(x,y)
\end{equation}
for all $\gamma \in \Gamma$ and $(x,y) \in M^{(2)}$. Then let $\tilde G = H + \phi$.

Since $\phi$ is bounded, $\tilde G$ is at bounded distance from $H$, which immediately implies that $(\sigma,\bar\sigma,\tilde G)$ is a coarse GPS system. 
The fact that $\nu: = e^{\delta \tilde{G}}\bar\mu \otimes \mu$ is locally finite on $M^{(2)}$ comes from the fact that $\tilde G$ is locally finite. Using Equation~\eqref{eqn:property of rho when defining BMS} one can show that $\nu$ is $\Gamma$-invariant. 
\end{proof}

\subsection{Conservative--dissipative dichotomy for BMS measures} In this section, we consider the conservativity/dissipativity of the $\Gamma$ action on $M^{(2)}$. 

We say that an orbit  $\Gamma \cdot (x,y)\subset M^{(2)}$ is \emph{escaping} if $\{\gamma\in\Gamma:\gamma(x,y)\in K\}$ is finite for any compact subset $K\subset M^{(2)}$.

\begin{lemma}\label{lem:conical do not escape}
 An orbit $\Gamma \cdot (x,y)\subset M^{(2)}$ is escaping if and only both $x$ and $y$ are not conical limit points.
\end{lemma}
\begin{proof}
Let $\dist$ be a compatible metric on $\Gamma \sqcup M$.

Suppose one of $x$ and $y$ is conical, say $x$.
Then there exists $\{\gamma_n\}\subset \Gamma$ and $a\neq b\in M$ such that $\gamma_nx\to a$ while $\gamma_nz\to b$ for any $z\in M\smallsetminus\{x\}$.
In particular $\gamma_n(x,y)\to (a,b)\in M^{(2)}$, so $\Gamma(x,y)$ is not escaping.

Next, suppose $\Gamma\cdot (x,y)$ is not escaping. Then $\dist(\gamma_nx,\gamma_ny)\ge\epsilon$  for some $\epsilon>0$ and some escaping sequence $\{\gamma_n\}\subset \Gamma$.
Passing to a subsequence, there are $c,b\in M$ such that $\gamma_nz\to b$ for any $z\in M\smallsetminus\{c\}$.
Since  $\{\gamma_nx\}$ and $\{\gamma_ny\}$ cannot both converge to $b$,
one of $x$ and $y$ must be equal to $c$, say $x$, so $\gamma_n y\to b$. 
Then, after  passing to a further subsequence, $\{\gamma_nx\}$ converge to a point  $a$ such that $\dist(a,b)\ge \epsilon$, so $a\in M\smallsetminus\{b\}$. Therefore,  $x$ is conical.
\end{proof}

As a corollary we obtain the following dichotomy, which is a part of our Hopf--Tsuji--Sullivan dichotomy (Theorem~\ref{our dichotomy}).

\begin{corollary}\label{cor:conservativity dissipativity dichotomy}
 Let $(\sigma,\bar\sigma,G)$ be a coarse GPS system and let $\nu$ be a BMS measure of dimension $\delta$ on $M^{(2)}$. 
 \begin{itemize}
 \item If $\sum_{\gamma \in \Gamma} e^{-\delta\norm\gamma_\sigma}=+\infty$, then the action of $\Gamma$ on $(M^{(2)},\nu)$ is conservative.
 \item If $\sum_{\gamma \in \Gamma} e^{-\delta\norm\gamma_\sigma}<+\infty$, then the action of $\Gamma$ on $(M^{(2)},\nu)$  is dissipative.
 \end{itemize}
\end{corollary}

\begin{proof}
By definition, $\nu=e^{\delta \tilde G}\bar\mu\otimes\mu|_{M^{(2)}}$ where $\mu$ and $\bar\mu$ are coarse Patterson--Sullivan measures for $\sigma$ and $\bar\sigma$, and $\tilde G\colon M^{(2)}\to\R$ is as in Lemma~\ref{lem:finding the good G}.
 Suppose $\sum e^{-\delta\norm\gamma_\sigma}=+\infty$ (\resp $<+\infty$).
 By Proposition~\ref{prop:uniform conical has full measure} (\resp Proposition~\ref{prop:some consequences of the shadow lemma}\eqref{item:convergent => conical limset null}) and Proposition~\ref{prop:GPS implies expanding}\eqref{item:norm vs dual norm}, $\mu$ and $\bar\mu$ give full measure to $\Lambda^{\rm con}(\Gamma)$ (\resp $M-\Lambda^{\rm con}(\Gamma)$). Hence $\nu$ gives full measure to $\Lambda^{\rm con}(\Gamma)^{(2)}$ (\resp $(M-\Lambda^{\rm con}(\Gamma))^{(2)}$) in $M^{(2)}$ , and hence gives full measure to the set of $\Gamma$-orbits in $M^{(2)}$ that do not escape (\resp that do escape) by Lemma~\ref{lem:conical do not escape}, which is the conservative (\resp dissipative) part in the Hopf decomposition of Lemma~\ref{lem:topHopf}.
\end{proof}


\section{A flow space}\label{sec:flowspace}


In this section, we use our Patterson--Sullivan measure to define a flow space which admits a measurable action by $\Gamma$. In the presence of a GPS system we construct a $\Gamma$-invariant flow-invariant measure on this flow space. The construction of the measurable action comes from work of Bader--Furman~\cite{BF2017}. 

For the rest of the section suppose $\Gamma \subset \mathsf{Homeo}(M)$ is a convergence group and $\sigma \colon \Gamma \times M \rightarrow \Rb$ is an expanding coarse-cocycle. As in Section~\ref{sec: action on M2}, let 
$$
M^{(2)} := \{ (x,y) \in M^2 : x \neq y\}.
$$
The space $M^{(2)}\times \mathbb R$ has a natural flow defined by 
$$
\psi^t(x,y,s) = (x,y,s+t). 
$$

\subsection{An action of $\Gamma$ on $M^{(2)}\times\R$}\label{sec: action on M2xR} In this section we show that any Patterson--Sullivan measure induces a measurable action of $\Gamma$ on $M^{(2)} \times \Rb$.

Suppose $\mu$ is a coarse $\sigma$-Patterson--Sullivan measure of dimension $\delta$. Then let $\sigma_{\rm PS} \colon \Gamma \times M \rightarrow \Rb$ be the measurable cocycle defined by 
$$
\sigma_{\rm PS}(\gamma, x) = -\frac{1}{\delta} \log \frac{ d \gamma^{-1}_* \mu}{d \mu}(x).
$$

\begin{observation}\label{obs:everywhere defined PS cocycle} We can assume that  $\sigma_{\rm PS}$ is everywhere defined and that  $\sigma_{\rm PS}$ is a cocycle:
$$
\sigma_{\rm PS}(\gamma_1\gamma_2, x)= \sigma_{\rm PS}(\gamma_1, \gamma_2x)+\sigma_{\rm PS}(\gamma_2, x)
$$
for all $\gamma_1, \gamma_2 \in \Gamma$ and $x \in M$. 
\end{observation} 

\begin{proof} By uniqueness of Radon--Nikodym derivatives, there exists a $\Gamma$-invariant set $E \subset M$ where $\mu(E) = 1$ and 
$$
\sigma_{\rm PS}(\gamma_1\gamma_2, x)= \sigma_{\rm PS}(\gamma_1, \gamma_2x)+\sigma_{\rm PS}(\gamma_2, x)
$$
for all $\gamma_1, \gamma_2 \in \Gamma$ and $x \in E$. Since $\mu(E^c) = 0$, 
we may assume that $\sigma_{\rm PS}|_{\Gamma \times E^c} \equiv 0$. Then  $\sigma_{\rm PS}$ is an everywhere defined cocycle.
\end{proof} 

Using Observation~\ref{obs:everywhere defined PS cocycle} we can define a $\Gamma$ action on $M^{(2)} \times \Rb$ by 
$$
\gamma \cdot (x,y, t) = (\gamma x, \gamma y, t+\sigma_{\rm PS}(\gamma, y)). 
$$ 
Notice that this action commutes with the flow $\psi^t$.

\subsection{A measure on the flow space} Now we assume that $\sigma$ is part of a coarse GPS system $(\sigma,\bar\sigma,G)$ and $\bar \mu$ is a coarse $\bar\sigma$-Patterson--Sullivan measure of dimension $\delta$. In this case, we will construct a flow-invariant measure on $M^{(2)} \times \Rb$.

Let $\nu=e^{\delta \tilde G}\bar\mu\otimes\mu|_{M^{(2)}}$ be a $\Gamma$-invariant BMS measure associated to $(\sigma,\bar\sigma,G,\mu,\bar\mu)$ as in Section~\ref{sec:BMS on M2},
where $\tilde G:M^{(2)}\to\Rb$ is measurable and $(\sigma,\bar\sigma,\tilde G)$ is a coarse GPS system.

Then let $\tilde m := \nu \otimes dt$, which is a measure on $M^{(2)}\times\R$. 
Notice that:
\begin{enumerate}
\item Since $\tilde{G}$ is locally bounded on $M^{(2)}$, the measure $\tilde m$ is locally finite on $M^{(2)} \times \Rb$. 
\item $\tilde m$ is $\Gamma$-invariant and $\psi^t$-invariant. 
\end{enumerate}

Next we show that the action of $\Gamma$ on  $(M^{(2)}\times \R, \tilde m)$ is dissipative (see Appendix~\ref{appendix:conservative and dissipative} for the definition).

Since $\mu$ is a coarse Patterson--Sullivan measure, there exists $C > 0$ such that for each $\gamma \in \Gamma$ there is some $M_\gamma \subset M$ with $\mu(M_\gamma)=1$ and 
\begin{equation*}
\sup_{x \in M_{\gamma}} \abs{ \sigma_{\rm PS}(\gamma,x)- \sigma(\gamma,x)} < C.
\end{equation*}
Let 
\begin{equation}\label{eqn:defn of M'}
M' := \bigcap_{\alpha \in \Gamma} \alpha \left( \bigcap_{\gamma \in \Gamma} M_\gamma \right). 
\end{equation}
Then $M'$ is $\Gamma$-invariant, $\mu(M')=1$, and 
\begin{equation}\label{eqn:difference between cocycles}
\sup_{x \in M', \gamma \in \Gamma} \abs{ \sigma_{\rm PS}(\gamma,x)- \sigma(\gamma,x)} < C.
\end{equation}

Finally let
$$
Z := \{ (x,y,t) \in M^{(2)} \times \Rb : y \in M'\}. 
$$
Then $Z$ is $\Gamma$-invariant and $\psi^t$-invariant, and has full $\tilde m$-measure.

The next result implies that if $v \in Z$, then its $\Gamma$-orbit is escaping, i.e.\ $\{\gamma:\gamma v\in K\}$ is finite for any compact set $K$.
In particular, $\tilde m$-almost every orbit is escaping.

\begin{proposition} 
\label{prop:Gamma action is proper}
For any compact subset $K \subset M^{(2)} \times \Rb$ the set 
$$
\{ \gamma \in \Gamma : (K \cap Z) \cap \gamma(K \cap Z) \neq \emptyset \}
$$
is finite. 
In particular, the action of $\Gamma$ on $(M^{(2)}\times \R, \tilde m)$ is dissipative.
\end{proposition} 

\begin{proof} Suppose for a contradiction that there exist a compact set $K \subset M^{(2)} \times \Rb$ and a sequence $\{\gamma_n\}$ of distinct elements of $\Gamma$ such that 
$$
(K \cap Z) \cap \gamma_n(K \cap Z) \neq \emptyset
$$
for all $n$. Passing to a subsequence we can assume that $\gamma_n \rightarrow a \in M$ and $\gamma_n^{-1} \rightarrow b \in M$, i.e. $\gamma_n^{-1}|_{M \smallsetminus \{a\}}$ converges locally uniformly to $b$.

For each $n$, fix 
$$
(x_n, y_n, t_n) \in (K \cap Z) \cap \gamma_n(K \cap Z).
$$
Passing to a subsequence we can suppose that $x_n \rightarrow x$ and $y_n \rightarrow y\neq x$. 

Since $\{t_n\}$ is bounded and
$$
(\gamma_n^{-1}x_n, \gamma_n^{-1}y_n, t_n + \sigma_{\rm PS}(\gamma_n^{-1}, y_n))=\gamma_n^{-1}(x_n, y_n, t_n)  \in K \cap Z,
$$
we see that $\{\sigma_{\rm PS}(\gamma_n^{-1}, y_n)\}$ is bounded. Then Equation~\eqref{eqn:difference between cocycles} implies that $\{ \sigma(\gamma_n^{-1}, y_n)\}$ is bounded. Then Proposition~\ref{prop:basic properties}\eqref{item:a technical fact} implies that $\gamma_n^{-1} y_n \to b$. 
Since 
$$
\liminf_{n \to \infty} \dist(\gamma_n^{-1} x_n, \gamma_n^{-1} y_n) > 0
$$
and $\gamma_n^{-1}|_{M \smallsetminus \{a\}}$ converges locally uniformly to $b$, we must have $x=a$. Hence $y \neq a$. 
Since $\sigma$ is expanding, there exists $C' > 0$ such that 
$$
\sigma(\gamma_n^{-1}, y_n) \geq \norm{\gamma_n^{-1}}_\sigma-C'
$$
for all $n \geq 1$. By Proposition~\ref{prop:basic properties}\eqref{item:properness}, this quantity diverges to $+\infty$ as $n\to\infty$, which contradicts our earlier observation that $\{ \sigma(\gamma_n^{-1}, y_n)\}$ must be bounded.
\end{proof}

\subsection{The quotient flow space and quotient measure} \label{sec:quotient BMS}

In this section we show that the quotient $\Gamma\backslash M^{(2)}\times\R$ is a reasonable measure space, the flow descends to a measurable flow on the quotient, and the measure $\tilde m$ descends to a flow-invariant measure on the quotient.

Endow the quotient $\Gamma\backslash M^{(2)}\times\R$ with the quotient sigma-algebra (of the Borel sigma-algebra). 
By Proposition~\ref{prop:Gamma action is proper}, the action of $\Gamma$ is dissipative with respect to the measure $\tilde m=\nu\otimes dt=e^{\delta \tilde G}\bar\mu\otimes\mu\otimes dt|_{M^{(2)} \times \Rb}$.
Thus by the discussion in Section~\ref{sec:quotient measures} the space $\Gamma\backslash M^{(2)}\times\R$ admits a quotient measure $m$, which we also call a \emph{BMS measure} associated to $(\sigma,\bar\sigma, G, \mu,\bar\mu)$.

Recall that the flow $\psi^t(x,y,s) = (x,y,t+s)$  commutes with the $\Gamma$ action. So $\psi^t$ descends to a measurable flow on the quotient space $\Gamma\backslash M^{(2)}\times\R$,  which we also denote by $\psi^t$. Since $\tilde m$ is  $\psi^t$-invariant, the uniqueness of quotient measures, again see Section~\ref{sec:quotient measures}, implies that $m$ is $\psi^t$-invariant.

Finally, by the discussion in Section~\ref{sec:funddom}, $\Gamma\backslash M^{(2)}\times\R$ has a $\psi^t$-invariant full $m$-measure subset that is standard (i.e. measurably embeds into $[0,1]$).

\subsection{The continuous case}\label{sec:kappa=0 case} The construction above involves a number of choices, for instance a different choice of Patterson--Sullivan measure could lead to a different $\Gamma$ action on $M^{(2)} \times \Rb$ and hence a different quotient space. 

In this section we show that in the continuous case, some of the technicalities and all of the choices made in the above construction can be avoided. 

First suppose that $\sigma \colon \Gamma \times M \rightarrow \Rb$ is an expanding $0$-coarse-cocycle. Then, in the discussion above, can assume that  $\sigma_{\rm PS} = \sigma$, $M'=M$, and $Z=M^{(2)} \times \Rb$. Then Proposition~\ref{prop:Gamma action is proper} implies that $\Gamma$ acts properly discontinuously on $M^{(2)} \times \Rb$ and hence the quotient
$$
U_\Gamma : = \Gamma \backslash \Lambda(\Gamma)^{(2)}\times \Rb,
$$
is a metrizable locally compact topological space. Further the flow $\psi^t$ descends to a continuous flow, also called $\psi^t$, on $U_\Gamma$. 

Next we assume that $\sigma$ is part of a \emph{continuous GPS system} $(\sigma,\bar \sigma, G)$, that is  $\kappa=0$ and $G$ is continuous in Definition~\ref{def:GPS}. We also assume that $\delta : = \delta_\sigma(\Gamma) < +\infty$ and $\sum_{\gamma \in \Gamma} e^{-\delta\norm{\gamma}_\sigma} = +\infty$. By Theorems~\ref{thm:PS measures exist} and~\ref{thm:ergodicity on single M} there are unique probability measures $\mu$, $\bar \mu$ on $M$ which satisfy 
$$
\frac{d\gamma_* \mu}{d\mu} =e^{-\delta \sigma(\gamma^{-1}, \cdot)} \quad \text{and} \quad \frac{d\gamma_* \bar\mu}{d\bar\mu} =e^{-\delta \bar\sigma(\gamma^{-1}, \cdot)}.
$$
Then, since $(\sigma,\bar \sigma, G)$ is a continuous GPS system, the measure $\nu:=e^{\delta G} \bar \mu \otimes \mu|_{M^{(2)}}$ on $M^{(2)}$ is $\Gamma$-invariant. 
Note that $\nu$ is supported on $\Lambda(\Gamma)^{(2)}$, see Theorem~\ref{thm:ergodicity on single M}.

Finally, the measure $\tilde m: = e^{\delta G} d\bar\mu \otimes d\mu \otimes dt$ on $\Lambda(\Gamma)^{(2)} \times \Rb$ descends to a $\psi^t$-invariant Borel measure $m_\Gamma$ on $U_\Gamma$. In this construction, no choices were made and so we call $m$ \textbf{the} \emph{Bowen--Margulis--Sullivan (BMS) measure associated to $(\sigma, \bar\sigma, G)$} and denote it by $m_\Gamma$.


\section{Ergodicity of product measures} \label{sec:double ergo}


In this section we prove ergodicity of the product action for coarse GPS systems whose Poincar\'e series diverge at the critical exponent.

\begin{theorem}\label{thm:ergodicity on product} Suppose $(\sigma, \bar{\sigma}, G)$ is a coarse GPS system with $\delta:=\delta_\sigma(\Gamma) < +\infty$ and $\mu$, $\bar \mu$ are coarse Patterson--Sullivan measures of dimension $\delta$ for $\sigma$, $\bar{\sigma}$ respectively. If 
$$
\sum_{\gamma \in \Gamma} e^{-\delta \norm{\gamma}_\sigma} = + \infty,
$$
then $\Gamma$ acts ergodically on $(M^{(2)}, \bar\mu \otimes \mu)$.
\end{theorem} 

As described in Section~\ref{sec:kappa=0 case}, in the continuous case there is a canonical flow space and in this case our arguments will yield the following, see Section \ref{subsec:continuous case} for the proof. 

\begin{theorem}
\label{continuous ergodicity} If $(\sigma, \bar\sigma, G)$ is a continuous GPS system with $\delta:=\delta_\sigma(\Gamma) < +\infty$ and 
$$
\sum_{\gamma \in \Gamma} e^{-\delta \norm{\gamma}_\sigma} = + \infty,
$$
then the flow $\psi^t$ on $(U_\Gamma, m_\Gamma)$ is conservative and ergodic, where $m_\Gamma$ is the  BMS measure
associated to  $(\sigma, \bar\sigma, G)$ defined in Section~\ref{sec:kappa=0 case}.
\end{theorem}

The general strategy of the proof goes back to Sullivan's original work in real hyperbolic geometry~\cite{Sullivan1979}. In particular, we use the Hopf ratio ergodic theorem to prove ergodicity of the flow space introduced in Section~\ref{sec:flowspace}, which in turn will imply ergodicity of the action of $\Gamma$ on $M^{(2)}$. Some of our arguments also use ideas from work of Bader--Furman~\cite{BF2017}.

 \subsection{Notation}
 
 We will freely use the notation and objects introduced in Sections~\ref{sec: action on M2} and~\ref{sec:flowspace}, in particular: 
\begin{enumerate} 
\item the measurable cocycle $\sigma_{\rm PS}$ introduced in Section~\ref{sec: action on M2xR}, the associated action of $\Gamma$ on $M^{(2)} \times \Rb$ given by 
$$
\gamma \cdot (x,y,t) = (\gamma x, \gamma y, t + \sigma_{\rm PS}(\gamma, y)),
$$
and the associated measurable quotient $\Gamma \backslash M^{(2)} \times \Rb$; 
\item the $\Gamma$-invariant measure $\nu=e^{\delta \tilde G} \bar \mu \otimes \mu$ on $M^{(2)}$ constructed in Section~\ref{sec:BMS on M2};
\item the flow $\psi^t(x,y,s) = (x,y,t+s)$ on $M^{(2)} \times \Rb$ and the quotient flow, also denoted by $\psi^t$, on $\Gamma \backslash M^{(2)} \times \Rb$; 
\item the flow-invariant measure $\tilde m = \nu \otimes dt$ on $M^{(2)} \times \Rb$ and the associated flow-invariant quotient measure $m$ on $\Gamma \backslash M^{(2)} \times \Rb$ described in Section~\ref{sec:quotient BMS};
\item the set $M'\subset M$ defined in Equation~\eqref{eqn:defn of M'}, which is $\Gamma$-invariant, has full $\mu$-measure, and where 
\begin{equation}\label{eqn:difference between cocycles 2}
C : = \sup_{\gamma\in\Gamma, \ y\in M'}\abs{\sigma(\gamma,y)-\sigma_{\rm PS}(\gamma,y)} < +\infty. 
\end{equation}
\end{enumerate}

We will also use the following notation from Section~\ref{sec:quotient measures}. For $f \in L^1(M^{(2)}\times\R,\tilde m)$, let $\tilde P(f)$ be the $\tilde m$-almost everywhere defined function on $M^{(2)} \times \Rb$ given by 
$$
\tilde P(f)(v) = \sum_{\gamma \in \Gamma} f(\gamma \cdot v)
$$
and let $P(f)$ be the $m$-almost everywhere defined function on the quotient given by $P(f)([v]) = \tilde P(f)(v)$. By Equation~\eqref{eqn:defining property of quotient measures}, 
\begin{equation}
\label{eqn:integral of P}
\int P(f) dm = \int f d\tilde m
\end{equation}
for all $f \in L^1(M^{(2)} \times \Rb, \tilde m)$ and the map 
$$
P \colon L^1(M^{(2)} \times \Rb, \tilde m) \to L^1(\Gamma \backslash M^{(2)} \times \Rb,m)
$$
is continuous. We also observe that
\begin{equation}
\label{eqn: flow upstairs and downstairs}
\tilde P(f)(\psi^t(v)) = P(f)(\psi^t([v]))
\end{equation}
whenever both sides are defined. 

Finally, given $\theta \in L^1(\Rb)$ and $f \in L^1(M^{(2)}, \nu)$, let $f \otimes \theta \in L^1(M^{(2)} \times \Rb, \tilde m)$ denote the function 
$$
(f \otimes \theta)(x,y,t) = f(x,y) \, \theta(t). 
$$
Notice that with $a,b$ fixed, the map
$$
f \in L^1(M^{(2)},\nu) \mapsto f\otimes 1_{[a,b]} \in L^1(M^{(2)}\times\R,\tilde m)
$$
is a continuous operator.

\subsection{Constructing a weight function for the Hopf ratio ergodic theorem} 
\label{sec:weight}
In this section we construct a weight function to use in the Hopf ratio ergodic theorem.  

We begin by relating conical limit points to recurrence properties of the flow. To that end, fix a compatible metric $\dist$ on $\Gamma \sqcup M$. Then given $\epsilon>0$, let
$$
K_{\epsilon} : = \{ (x,y,0) : \dist(x,y) \geq \epsilon\}. 
$$

\begin{proposition}\label{prop:quantitative recurrence} Fix $0 < \epsilon' < \epsilon$ and $y \in M'$. 
\begin{enumerate}
\item If $y\in \Lambda^{\rm con}_\epsilon(\Gamma)$, then there exists a sequence of distinct elements $\{\gamma_n\} \subset \Gamma$  such that: for any 
$x \in M \smallsetminus\{y\}$, there is a sequence $\{t_n\} \subset \Rb$ with $t_n \to +\infty$ so that $(x,y,t_n) \in \gamma_n( K_{\epsilon'}) $ for  $n$ sufficiently large.
\item  If there exist a sequence of distinct elements $\{\gamma_n\} \subset \Gamma$, $x \in M \smallsetminus \{y\}$, and a sequence $\{t_n\} \subset \Rb$ so that $\{t_n\}$ is bounded below and $(x,y,t_n) \in \gamma_n K_\epsilon $ for all $n$, then $y \in \Lambda^{\rm con}_\epsilon(\Gamma)$. 
\end{enumerate}

\end{proposition} 

\begin{proof} (1) If $y\in \Lambda^{\rm con}_\epsilon(\Gamma)$, there exist $a,b \in M$ so that $\dist(a,b)\ge \epsilon$ and a sequence $\{\gamma_n\} \subset \Gamma$ 
such that  $\gamma_n^{-1} y \to a$ and $\gamma_n^{-1} x \to b$ for all $x \in M \smallsetminus \{y\}$. In particular, $\gamma_n^{-1} \to b$ and $\gamma_n \to y$.

Fix $x \in M \smallsetminus \{y\}$. Then $(\gamma_n^{-1} x, \gamma_n^{-1} y) \rightarrow (b,a)$ and so $(\gamma_n^{-1} x, \gamma_n^{-1} y,0) \in K_{\epsilon'}$ for  all
sufficiently large $n$. Then 
$$
(x,y, \sigma_{\rm PS}(\gamma_n, \gamma_n^{-1}y))=\gamma_n( \gamma_n^{-1} x, \gamma_n^{-1} y, 0)   \in \gamma_n (K_{\epsilon'}) 
$$
for sufficiently large $n$. Since $\gamma_n^{-1}y \to a$, $\gamma_n^{-1} \to b$, and $a \neq b$,
Equation~\eqref{eqn:difference between cocycles 2} and the expanding property for $\sigma$ imply that 
$$
\lim_{n \rightarrow \infty} \sigma_{\rm PS}(\gamma_n, \gamma_n^{-1}y) \geq -C + \lim_{n \rightarrow \infty} \sigma(\gamma_n, \gamma_n^{-1}y)=+\infty.
$$
Hence if $t_n := \sigma_{\rm PS}(\gamma_n, \gamma_n^{-1}y)$, then $t_n \rightarrow +\infty$ and $(x,y,t_n) \in \gamma_n(K_{\epsilon'})$ for  all sufficiently large $n$.

(2) Now suppose there exist a sequence of distinct elements $\{\gamma_n\} \subset \Gamma$, $x \in M \smallsetminus \{y\}$, and a sequence $\{t_n\} \subset \Rb$ 
so that $\{t_n\}$ is bounded below and $(x,y,t_n) \in \gamma_n (K_\epsilon) $ for all $n$.
Passing to subsequence, we may assume that $\gamma_n^{-1} y \to a \in M$ and $\gamma_n^{\pm 1} \rightarrow b^\pm$. 

Notice that $\gamma_n^{-1}(x,y,t_n) \in K_\epsilon$ and so 
$$
\{\sigma_{PS}(\gamma_n^{-1}, y) + t_n \}
$$
is bounded. Since $\{t_n\}$ is bounded below, $\{\sigma_{PS}(\gamma_n^{-1}, y)\}$ is bounded above. Then Equation~\eqref{eqn:difference between cocycles 2} implies that $\{\sigma(\gamma_n^{-1}, y)\}$ is bounded above. Then, since $\sigma$ is expanding, Proposition~\ref{prop:basic properties}\eqref{item:properness} implies that $\dist(\gamma_n, y) \rightarrow 0$. Thus $\liminf_{n \rightarrow \infty} \dist(x,\gamma_n) > 0$ and so by the convergence group property, $\gamma_n^{-1} x \rightarrow b^-$. By assumption, $\dist(\gamma_n^{-1} x, \gamma_n^{-1} y)\ge \epsilon$ for all $n$, so 
$\dist(a,b^-)\ge \epsilon$, and hence $y\in \Lambda^{\rm con}_\epsilon(\Gamma)$. 
\end{proof} 

Given $v \in M^{(2)} \times \Rb$ we let 
$$v^\pm \in M$$
 denote the associated ``forward/backward endpoints'' of $v$, that is $v = (v^-, v^+, t)$ for some $t \in \Rb$. 

Using Theorem~\ref{thm:ergodicity on single M}, we can fix $\epsilon_0>0$ sufficiently small so that  
$$
\mu(\Lambda_{\epsilon_0}^{\rm con}(\Gamma)) = 1.
$$
Then by Proposition~\ref{prop:quantitative recurrence}, there exists a compact subset $K \subset M^{(2)} \times \Rb$ such that for every $v\in M^{(2)} \times \Rb$ with $v^+ \in \Lambda_{\epsilon_0}^{\rm con}(\Gamma)\cap M'$ there exist sequences $\{\gamma_n\} \subset \Gamma$ and  $\{t_n\} \subset [0,\infty)$ where $t_n \rightarrow +\infty$ and 
$$
\psi^{t_n}(v) \in \gamma_n( K)
$$
for all $n \geq 1$. 
Then fix a non-negative $\rho_0 \in C_c(M^{(2)})$ and $R > 0$ such that 
$$
\rho_0 \otimes 1_{[-R,R]} \geq 1
$$
on $\bigcup_{t \in [0,1]}\psi^t(K)$. 
Then let 
$$
\tilde\rho:= \tilde P(\rho_0 \otimes 1_{[-R,R]}) \quad \text{and} \quad \rho:=P(\rho_0 \otimes 1_{[-R,R]}).
$$
Notice that $\rho \in L^1( \Gamma \backslash M^{(2)} \times \Rb, m)$.

\begin{lemma}\label{lem:average value of rho} If $v \in M^{(2)} \times \Rb$ and $v^+ \in \Lambda_{\epsilon_0}^{\rm con}(\Gamma)\cap M'$, then 
$$
\lim_{T \rightarrow \infty} \int_0^T \tilde\rho(\psi^t(v)) dt = +\infty. 
$$
In particular, 
$$
\lim_{T \rightarrow \infty} \int_0^T \rho(\psi^t(v)) dt = +\infty
$$
for $m$-almost every $v \in \Gamma \backslash M^{(2)} \times \Rb$ and so the quotient flow $\psi^t \colon \Gamma \backslash M^{(2)} \times \Rb \to  \Gamma \backslash M^{(2)} \times \Rb$ is conservative (see Fact~\ref{fact:Hopf decompo}).
\end{lemma} 

The fact that $\psi^t$ is conservative can also be deduced from Corollary~\ref{cor:conservativity dissipativity dichotomy} and \cite[Fact 2.29]{BlayacPS}.

\begin{proof} Fix  $v \in M^{(2)} \times \Rb$ with $v^+ \in \Lambda_{\epsilon_0}^{\rm con}(\Gamma) \cap M'$.
By our choice 
of $K$, there exist $\{\gamma_n\} \subset \Gamma$ and  $\{t_n\} \subset [0,\infty)$ where $t_n +1 < t_{n+1}$ and 
$$
\psi^{t_n}(v) \in \gamma_n( K)
$$
for all $n \geq 1$. Then 
$$
\liminf_{T \rightarrow \infty} \int_0^T \tilde\rho(\psi^t(v)) dt \geq \sum_{n=1}^\infty \int_{t_n}^{t_n+1} \tilde\rho(\psi^t(v)) dt = +\infty,
$$
since $\tilde\rho(\psi^t(v)) \geq 1$ for any  $t\in [t_n, t_n+1]$. 

The ``in particular'' statement then follows from Equation~\eqref{eqn: flow upstairs and downstairs}. 
\end{proof} 

\subsection{Applying the Hopf ratio ergodic theorem}
\label{subsec:Hopf ratio} 
Next we apply the Hopf ratio ergodic theorem to the conservative flow $\psi^t$ on $(\Gamma\backslash M^{(2)} \times \Rb,m)$. 
This theorem was first proved by Stepanoff \cite{Stepanoff} and Hopf \cite{Hopf37}.
For a modern reference: Krengel states the result for discrete actions of $\Z_{\geq1}$~\cite[Th.\,2.7 \& 3.4]{Krengel} and explains how to then deduce the result for flows~\cite[\S2 p.10]{Krengel}.

 In our setting, the Hopf ratio ergodic theorem  implies that if $\rho$ is the weight function defined in Section \ref{sec:weight} and
$f \in L^1(\Gamma\backslash M^{(2)} \times \Rb, m)$, then
$$
\Phi(f)(v) = \lim_{T \rightarrow \infty} \frac{ \int_0^T f(\psi^t(v)) dt}{ \int_0^T \rho(\psi^t(v)) dt} 
$$
exists for every $v$ in a $\psi^t$-invariant set of $m$-full measure. 
Further, the $m$-almost everywhere defined function $\Phi(f)$ is measurable and $\psi^t$-invariant, and  $ \Phi(f) \rho$ is integrable with
\begin{equation}\label{eqn:Phi is the conditional expectation}
\int_A \Phi(f) \rho\ dm = \int_{A} f \ dm
\end{equation}
for any $\psi^t$-invariant subset $A \subset \Gamma\backslash M^{(2)} \times \Rb$. Since $\abs{\Phi(f)} \leq \Phi(\abs{f})$, Equation~\eqref{eqn:Phi is the conditional expectation} implies that 
$$\Phi \colon L^1(\Gamma\backslash M^{(2)} \times \Rb, m) \to L^1(\Gamma\backslash M^{(2)} \times \Rb, \rho m)$$
is continuous. 

We will also let $\tilde{\Phi}(f) \colon M^{(2)} \times \Rb \to \Rb$ denote the lift of  $\Phi(f)$, which is  $\tilde m$-almost everywhere defined, $\Gamma$-invariant and $\psi^t$-invariant.

Using a Hopf Lemma type argument, we will deduce the following.

\begin{proposition}\label{prop:ergodic averages} If  $ f  \in C_c(M^{(2)})$ and $c<d$, then $\Phi\circ P(f \otimes 1_{[c,d]})$ is constant $m$-almost everywhere. 
\end{proposition} 

This proposition will be proved by first showing $\tilde{\Phi}\circ P(f \otimes 1_{[c,d]})$ is almost surely constant along ``weak stable manifolds'' of the form $M\times\{y\}\times\R$, 
which are parametrized by $y\in M$. Thus $\tilde{\Phi}\circ P(f \otimes 1_{[c,d]})$ induces a $\Gamma$-invariant  function on $M$ defined by 
$$
y \mapsto \tilde{\Phi}\circ P(f \otimes 1_{[c,d]})(*,y,*). 
$$
Then Theorem~\ref{thm:ergodicity on single M}, which says that $\Gamma$ acts ergodically on $(M,\mu)$, will imply  that this function is constant. 

Delaying the proof of Proposition~\ref{prop:ergodic averages}, we deduce Theorem~\ref{thm:ergodicity on product}.

\begin{lemma} $\Gamma$ acts ergodically on $(M^{(2)},\nu)$ and hence also on $(M^{(2)},\bar\mu \otimes \mu)$. \end{lemma}

\begin{proof} 
Suppose for a contradiction that there exists a $\Gamma$-invariant measurable set $A \subset M^{(2)}$ where $\nu(A) > 0$ and $\nu(A^c)>0$. 
By inner regularity, there exists a compact subset $K \subset A^c$ with $\nu(K) > 0$. 

Let $\{f_n\}$ be a sequence of compactly supported continuous functions on $M^{(2)}$ converging to $1_K$ in $L^1(M^{(2)},\nu)$. Since $\Phi$, $P$ and $\cdot\otimes 1_{[0,1]}$ are continuous operators, we have 
$$
\Phi\circ P(f_n\otimes 1_{[0,1]})\to \Phi\circ P(1_K \otimes 1_{[0,1]})=\Phi\circ P(1_{K\times[0,1]})
$$
in $ L^1(\Gamma\backslash M^{(2)} \times \Rb, \rho m)$.

By Proposition~\ref{prop:ergodic averages}, each $\Phi\circ P(f_n\otimes 1_{[0,1]})$ is constant $m$-almost everywhere and hence constant $\rho m$-almost everywhere. Hence  the limit $ \Phi\circ P(1_{K\times[0,1]})$ is  constant $\rho m$-almost everywhere (since the convergence is in $L^1(\Gamma\backslash M^{(2)} \times \Rb, \rho m)$). So there exists $c \in \Rb$ such that 
$$
\Phi\circ P(1_{K\times[0,1]})(v) = c
$$
for $m$-almost every $v \in \{ \rho \neq 0\}$. However, by Lemma~\ref{lem:average value of rho} for $m$-almost every $v \in \Gamma \backslash M^{(2)} \times \Rb$ the orbit $\psi^t(v)$ intersects $\{\rho \neq 0\}$. Since $\Phi\circ P(1_{K\times[0,1]})$ is $\psi^t$-invariant, then
$$
\Phi\circ P(1_{K\times[0,1]})(v) =c
$$
for $m$-almost every $v \in \Gamma \backslash M^{(2)} \times \Rb$. 

By definition, $K \subset A^c$ and so $ \Phi\circ P(1_{K\times[0,1]})$ is well-defined and equal to zero on $\Gamma\backslash A\times \R$.
Moreover, $\Gamma\backslash A\times \R$ has positive $m$-measure (see Remark~\ref{rem:quotient is sigmafinite}).  So $c = 0$ and $\Phi\circ P(1_{K\times[0,1]})=0$ $\rho m$-almost everywhere. Hence 
\begin{equation*}
\int\Phi\circ P(1_{K\times[0,1]})\rho \,dm =0.
\end{equation*} 
However, by Equations~\eqref{eqn:Phi is the conditional expectation} and~\eqref{eqn:integral of P},
\begin{equation*}
\int\Phi\circ P(1_{K\times[0,1]})\rho \,dm = \int P(1_{K \times [0,1]}) \,dm = \int_{M^{(2)} \times \Rb} 1_{K \times [0,1]} \, d\tilde m > 0.
\end{equation*} 
So we have a contradiction.
\end{proof}

\subsection{Proof of Proposition~\ref{prop:ergodic averages} } 

We start with a technical lemma similar to \cite[Lem.\,2.6]{BF2017}. The statement of the lemma is somewhat opaque, but can be interpreted as a boundary version of the assertion that the flow $\psi^t \colon M^{(2)} \times \Rb\to M^{(2)} \times \Rb$ has ``weak stable manifolds'' of the form $M\times\{y\}\times\R$. In the case when the GPS system is continuous, this assertion about  ``weak stable manifolds''  can be made precise, see \cite[\S3]{papertwo}.

Recall that $\dist$ is a compatible metric on $\Gamma \sqcup M$. 

\begin{lemma}\label{lem:contraction} For any $\epsilon, r > 0$ and  $B \in \Rb$ there exists a finite subset $F \subset \Gamma$ such that: if $x_1, x_2, y \in M$, 
$\gamma \in \Gamma \smallsetminus F$, $\sigma(\gamma, y) \leq B$ and
$$
\min\{ \dist(x_1,y), \dist(x_2,y) \} \geq r,
$$
then   
$$
\dist(\gamma x_1, \gamma x_2) < \epsilon. 
$$
\end{lemma} 

\begin{proof} Suppose not. Then there exists a sequence $\{\gamma_n\} \subset \Gamma$ of distinct elements such that for every $n \geq 1$ there are $x_{1,n}, x_{2,n}, y_n \in M$ where
\begin{align*}
\min & \{ \dist(x_{1,n},y_n), \dist(x_{2,n},y_n) \} \geq r, \quad \sigma(\gamma_n, y_n) \leq B, \quad \text{and} \quad \dist(\gamma_n x_{1,n}, \gamma_n x_{2,n}) \geq \epsilon.
\end{align*}
Passing to a subsequence we can suppose that $\gamma_n^{\pm 1} \rightarrow a^\pm \in M$. Since $\{ \gamma_n\}$ are distinct, Proposition~\ref{prop:basic properties}\eqref{item:properness} implies that $\norm{\gamma_n}_\sigma \rightarrow +\infty$. Then since 
$$
\sigma(\gamma_n, y_n) \leq B
$$
and $\sigma$ is expanding, we must have $y_n \rightarrow a^-$. Then since 
$$
\min\{ \dist(x_{1,n},y_n), \dist(x_{2,n},y_n) \} \geq r,
$$
we have $\gamma_n x_{1,n} \rightarrow a^+$ and $\gamma_n x_{2,n} \rightarrow a^+$. So 
$$
\lim_{n \rightarrow \infty} \dist(\gamma_n x_{1,n}, \gamma_n x_{2,n})= 0
$$
and we have a contradiction. 
\end{proof} 

We now begin our investigation of functions of the form $P(f \otimes 1_{[c,d]})$.

\begin{lemma}
\label{gh bounds} 
Suppose $ f  \in C_c(M^{(2)})$, $c < d$,  $g := \tilde P(f \otimes 1_{[c,d]})$, and $h:=\tilde P(1_{{\rm supp}(f)} \otimes 1_{[c,d]})$. 
If $v,w \in M^{(2)}\times \R$ satisfy $v^+=w^+\in M'$ and $\epsilon > 0$, then  there exists $C=C(f,c,d,v,w,\epsilon)>0$ such that 
\begin{align*}
 \abs{ \int_0^T g  (\psi^t(v)) dt - \int_0^T g (\psi^t(w)) dt}  \leq C + \epsilon \left(  \int_0^T h (\psi^t(v)) dt+  \int_0^T h  (\psi^t(w)) dt\right)
\end{align*}  
for all $T \geq 0$. 
\end{lemma} 

\begin{proof} Let $v = (x_1, y, s)$ and $w=(x_2, y, s^\prime)$. 
 Let $Q$ be the support of $f\otimes  1_{[c,d]}$. Since
$Q$ is compact,
Proposition~\ref{prop:Gamma action is proper} implies that
$$N:=\#\{\gamma\in\Gamma :  (Q\cap Z)\cap \gamma(Q\cap Z)\ne\emptyset\}<+\infty$$
where $Z=\{(x,y,t) \in M^{(2)} \times \mb R : y \in M'\}$. So, if $u\in Z$, then
$|g(u)|\leq N\norm f_\infty$. 
Therefore,
\begin{align*}
 \abs{ \int_0^T g  (\psi^t(x_2,y,s^\prime)) dt - \int_0^T g (\psi^t(x_2,y,s)) dt} \leq 2N\norm f_\infty \abs{s-s^\prime}.
\end{align*} 
Since our constant $C$ depends on $v$ and $w$, and hence on $s$ and $s^\prime$, we can assume that $s=s^\prime$. 

Let $\lambda$ be the Lebesgue measure on $\Rb$ and let
$$
L_{\gamma}(T) : =  \lambda\Big( [0,T] \cap [c-s-\sigma_{\rm PS}(\gamma,y), d-s-\sigma_{\rm PS}(\gamma, y)]\Big).
$$
Then
\begin{align*}
\int_0^T g  (\psi^t(v)) dt - \int_0^T g (\psi^t(w)) dt = \sum_{\gamma \in \Gamma} \left( f(\gamma x_1, \gamma y) - f(\gamma x_2, \gamma y) \right) L_\gamma(T).
\end{align*} 
Notice that if $L_\gamma(T) \neq 0$, then Equation~\eqref{eqn:difference between cocycles 2} implies that 
$$
\sigma(\gamma, y) \leq C+\sigma_{\rm PS}(\gamma, y) \leq C+d-s.
$$
So by the uniform continuity of $f$ and Lemma~\ref{lem:contraction}, there exists a finite set $F \subset \Gamma$ such that: if $\gamma \in \Gamma \smallsetminus F$ and 
$L_\gamma(T) \neq 0$, then 
$$
\abs{ f(\gamma x_1, \gamma y) - f(\gamma x_2, \gamma y)} \leq \epsilon. 
$$
Then, writing $S := \supp(f)$, 
\begin{align*}
&  \abs{ \int_0^T g  (\psi^t(v)) dt - \int_0^T g (\psi^t(w)) dt} \\
& \quad \leq \sum_{\gamma \in F} 2\norm{f}_\infty (d-c) + \epsilon  \sum_{\gamma \in \Gamma \smallsetminus F}  \big(1_{S}(\gamma x_1, \gamma y) + 1_{S}(\gamma x_2, \gamma y) \big)L_\gamma(T) \\
& \quad \leq \sum_{\gamma \in F} 2\norm{f}_\infty (d-c) + \epsilon\left(  \int_0^T h (\psi^t(v)) dt+  \int_0^T h  (\psi^t(w)) dt\right). \qedhere
\end{align*}  
\end{proof} 

Recall that Lemma~\ref{lem:average value of rho} says that $\lim_{T \rightarrow \infty} \int_0^T \tilde \rho(\psi^t(v)) dt = +\infty$ for any $v$ with $v^+\in \Lambda^{\rm con}_{\epsilon_0}(\Gamma) \cap M'$. 
The next lemma shows that, on a full measure set, the convergence to infinity is asymptotically identical for flow lines with the same forward endpoint. 

\begin{lemma}\label{lem:rho value the same}
There is a full $\tilde m$-measure set $Y_\rho\subset M^{(2)}\times \R$ such that:
If $v,w \in Y_\rho$ and $v^+ = w^+\in  \Lambda^{\rm con}_{\epsilon_0}(\Gamma) \cap M'$, then 
$$
\lim_{T \rightarrow \infty} \frac{ \int_0^T \tilde\rho(\psi^t(v)) dt }{ \int_0^T \tilde\rho(\psi^t(w)) dt } = 1. 
$$
\end{lemma}

\begin{proof} 
Recall that $\tilde \rho = \tilde P(\rho_0 \otimes 1_{[-R,R]})$ where $\rho_0 \in C_c(M^{(2)})$.
Let $h := P(1_{{\rm supp}(\rho_0)} \otimes 1_{[-R,R]})$.
By  the Hopf ratio ergodic theorem, there is a full measure subset $Y_\rho$ where $\tilde{\Phi}(h)$ exists.

Fix $v,w \in Y_\rho$ with $v^+ = w^+\in  \Lambda^{\rm con}_{\epsilon_0}(\Gamma) \cap M'$ and let 
$$
r_T : = \frac{ \int_0^T \tilde\rho(\psi^t(v)) dt }{ \int_0^T \tilde\rho(\psi^t(w)) dt }.
$$
By Lemma~\ref{lem:average value of rho}, there exists $T_0 > 0$ such that $r_T \in (0,\infty)$ for all $T \geq T_0$. 

Suppose for a contradiction that 
 $$
\lim_{T \rightarrow \infty} r_T \neq 1. 
$$
Then after possibly relabeling $v$ and $w$ there exists $T_n \rightarrow \infty$ such that 
$$
r_\infty : = \lim_{n \rightarrow \infty}r_{T_n} \in (1,+\infty].
$$

By Lemma \ref{gh bounds}, for any $\epsilon > 0$ there exists $C_\epsilon > 0$ such that 
\begin{align*}
r_{T} & = 1 + \frac{ \int_0^T \tilde\rho(\psi^t(v)) dt - \int_0^T \tilde\rho(\psi^t(w)) dt}{ \int_0^T \tilde\rho(\psi^t(w)) dt }\\
& \leq 1 + \frac{C_\epsilon} { \int_0^{T} \tilde\rho(\psi^t(w)) dt }+  \epsilon \left(  r_T\frac{\int_0^T h (\psi^t(v)) dt}{ \int_0^T \tilde\rho(\psi^t(v)) dt }+  \frac{\int_0^T h  (\psi^t(w)) dt}{ \int_0^T \tilde\rho(\psi^t(w)) dt }\right)
 \end{align*}
 for all $T \geq T_0$. Hence, 
 $$
\left( 1- \epsilon\frac{\int_0^T h (\psi^t(v)) dt}{ \int_0^T \tilde\rho(\psi^t(v)) dt }\right) r_{T} \leq 1 + \frac{C_\epsilon} { \int_0^{T} \tilde\rho(\psi^t(w)) dt }+  \epsilon  \frac{\int_0^T h  (\psi^t(w)) dt}{ \int_0^T \tilde\rho(\psi^t(w)) dt }
 $$
 for all $T \geq T_0$. Lemma~\ref{lem:average value of rho} implies that 
$$
\lim_{T \rightarrow \infty} \int_0^T \tilde \rho(\psi^t(w)) dt = +\infty.
$$
So  for any $\epsilon > 0$ we have 
$$
\left(1 - \epsilon \tilde{\Phi}(h)(v)\right) r_\infty \leq 1 + \epsilon \tilde{\Phi}(h)(w) 
$$
Since $\epsilon > 0$ is arbitrary, we have $r_\infty \leq 1$ which is a contradiction. 
 \end{proof} 
 
We are now ready to finish the proof of Proposition~\ref{prop:ergodic averages}. 
Fix $f \in C_c(M^{(2)})$ and $c<d$. Let $g := \tilde P(f \otimes 1_{[c,d]})$ and  $h := \tilde P(1_{{\rm supp}(f)} \otimes1_{[c,d]})$. 
By the Hopf ratio ergodic theorem, there is a full $\tilde m$-measure set  $X$  where $\tilde \Phi(g)$ and $\tilde \Phi(h)$ both exist.
Let $Y_\rho$ be the full $\tilde m$-measure set given by Lemma~\ref{lem:rho value the same}. Then, $Y:=X\cap Y_\rho$ is 
a full $\tilde m$-measure set.

We claim that 
\begin{equation}\label{eqn:Hopf type lemma}
\tilde\Phi(g)(v) = \tilde\Phi(g)(w)
\end{equation} 
when $v,w \in Y$ and $v^+ = w^+\in  \Lambda^{\rm con}_{\epsilon_0}(\Gamma) \cap M'$.
Indeed, Lemma~\ref{lem:average value of rho} implies that 
$$
\lim_{T \rightarrow \infty} \int_0^T \tilde\rho(\psi^t(v)) dt = +\infty.
$$
So by Lemmas~\ref{gh bounds} and ~\ref{lem:rho value the same}, for any $\epsilon > 0$ we have
$$
\abs{\tilde\Phi(g)(v) - \tilde\Phi(g)(w)} \leq \epsilon \Big( \tilde\Phi(h)(v) + \tilde\Phi(h)(w)\Big).
$$
So $\tilde\Phi(g)(v) = \tilde\Phi(g)(w)$. 

Now suppose for a contradiction that $\tilde{\Phi}(g)$ is not constant $\tilde m$-almost everywhere. Then there exists a measurable set $A \subset \Rb$ such that the sets $\{ v : \tilde{\Phi}(g)(v) \in A\}$ and $\{ v : \tilde{\Phi}(g)(v) \in A^c\}$ both have positive $\tilde m$-measure. As before, let $\lambda$ denote the Lebesgue measure on $\Rb$. 

Then consider the following sets
\begin{align*}
A_1 &:= \{ y \in M : \tilde\Phi(g)(x,y,t) \in A \text{ for $\mu \otimes \lambda$-a.e. } (x,t) \},\\
A_2 &:= \{ y \in M : \tilde\Phi(g)(x,y,t) \in A \text{ for a positive $\mu \otimes \lambda$-measure set of } (x,t) \},\\
B &:=  \{ y \in M'\cap\Lambda^{\rm con}_{\epsilon_0}(\Gamma) : (x,y,t) \in Y \text{ for $\mu \otimes \lambda$-a.e. } (x,t) \}.
\end{align*}
Note that $A_1\subset A_2$, 
that these sets are $\Gamma$-invariant (since $\tilde\Phi(g)$ is $\Gamma$-invariant),
and that they are measurable by Fubini's Theorem (for instance $A_1$ is the set of $y$ such that $\int 1_{A^c}\circ\tilde\Phi(x,y,t)d\mu(x)dt=0$).
Moreover $A_2$ has positive measure, by definition of $A$, and $B$ has full measure.
Finally Equation~\eqref{eqn:Hopf type lemma} implies that $A_2\cap B\subset A_1$, so $\mu(A_1) > 0$.
One also checks that $M\smallsetminus A_1$, the set of $y$ such that $\tilde\Phi(g)(x,y,t) \in A^c$ for a positive $\mu \otimes \lambda$-measure set of  $(x,t)$, satisfies $\mu(M\smallsetminus A_1)>0$. 

However, this contradicts the fact that
$\Gamma$ acts ergodically on $(M, \mu)$, see Theorem~\ref{thm:ergodicity on single M}.

Thus $\tilde{\Phi}(g)$ is constant $\tilde m$-almost everywhere, which implies that $\Phi\circ P(f \otimes 1_{[c,d]})$ is constant $m$-almost everywhere.
\qed

\subsection{Proof of Theorem \ref{continuous ergodicity}: The continuous case}
\label{subsec:continuous case}

In this section, we observe that the arguments we have just given immediately establish Theorem \ref{continuous ergodicity}. We will freely use the objects introduced in Section~\ref{sec:kappa=0 case}. 

First notice that the flow $\psi^t \colon (U_\Gamma, m_\Gamma) \to (U_\Gamma, m_\Gamma)$ introduced in Section~\ref{sec:kappa=0 case} coincides with the flow $\psi^t \colon (\Gamma \backslash M^{(2)} \times \Rb, m) \to (\Gamma \backslash M^{(2)} \times \Rb, m)$ considered in the proof of Theorem~\ref{thm:ergodicity on product}. So Lemma \ref{lem:average value of rho} implies immediately that $\psi^t$ is conservative on $(U_\Gamma,m_\Gamma)$. 

If $\psi_t$ is not ergodic on $(U_\Gamma,m_\Gamma)$, then there exists a flow-invariant subset $A$ of $U_\Gamma$ so that $m_\Gamma(A)>0$ and $m_\Gamma(A^c) > 0$. Then
$A$ lifts to a flow-invariant, $\Gamma$-invariant subset $\tilde A$ of $M^{(2)}\times\Rb$ of the form $B\times \Rb$. Then, $(\bar\mu \otimes\mu)(B)>0$
and  $(\bar\mu\otimes\mu)(B^c)>0$, which contradicts the ergodicity of the action of $\Gamma$ on $(M^{(2)},\bar\mu \otimes\mu)$. Therefore,
$\psi^t$ is ergodic on $(U_\Gamma,m_\Gamma)$.

We also note that in this case the proof of ergodicity of the geodesic flow can be simplified by using \cite{Coudene_Hopf}, see \cite[\S6.6]{BlayacPS} and \cite[\S4]{papertwo}.

\section{Proof of dichotomy}\label{sec:proof of dichotomy}

In this section we complete the proof of Theorem~\ref{our dichotomy}.
Suppose $(\sigma, \bar\sigma, G)$ is a coarse GPS system and $\delta_\sigma(\Gamma) < +\infty$. Let $\mu$ and $\bar \mu$ be coarse Patterson--Sullivan measures of dimension $\delta$ for $\sigma$ and $\bar{\sigma}$ respectively. 
By Lemma~\ref{lem:finding the good G}, there exists a measurable function $\tilde G$ on $M^{(2)}$ such that 
$\nu:=e^{\delta \tilde G}\bar\mu\otimes\mu$ is $\Gamma$-invariant.

We have already completed most of the proof. There is one lemma left to prove:

\begin{lemma} \label{lem:ergodic implies conservative}
If the action of $\Gamma$ on $(M^{(2)}, \nu)$ is ergodic, then $\nu$ has no atoms, and hence the $\Gamma$ action on $(M^{(2)}, \nu)$ is also conservative. 
\end{lemma}

\begin{proof}
We argue by contradiction: suppose the $\Gamma$ action on $(M^{(2)}, \nu)$ is ergodic and $(\xi,\eta) \in M^{(2)}$ is an atom of $\nu$. Then $\mc O := \Gamma \cdot (\xi,\eta)$ must have full $\nu$-measure by ergodicity. 
Pick $\gamma \in \Gamma$ such that $(\xi, \gamma \eta) \in M^{(2)} - \mc O$. Since $\nu=e^{\delta \tilde G}\bar\mu\otimes\mu$, we see that $(\xi,\gamma\eta)$ is also an atom for $\nu$, which contradicts the fact that $\mc O$ has full measure.

Conservativity of the $\Gamma$ action then follows by \cite[Prop.\ 1.6.6]{Aaronson} (see also \cite{MO:330015}).
\end{proof}

\subsection{Divergent case}

First suppose $\sum_{\gamma \in \Gamma} e^{-\delta \norm{\gamma}_\sigma} = +\infty$.

(a) By the definition of the critical exponent, $\delta \leq \delta_\sigma(\Gamma)$. By Proposition \ref{prop:some consequences of the shadow lemma}\eqref{item:dim>=critexp}, $\delta \geq \delta_\sigma(\Gamma)$. Hence $\delta = \delta_\sigma(\Gamma)$.

(b) By Proposition~\ref{prop:uniform conical has full measure}, $\mu\left( \Lambda^{\rm con}(\Gamma) \right) = 1$.

(c) By Theorem~\ref{thm:ergodicity on product}, the action  of $\Gamma$ on $(M^{(2)}, \nu)$ is ergodic. Conservativity of the action can be seen from Corollary~\ref{cor:conservativity dissipativity dichotomy}, or from Lemma~\ref{lem:ergodic implies conservative}.

\subsection{Convergent case}

Now suppose $\sum_{\gamma \in \Gamma} e^{-\delta \norm{\gamma}_\sigma} < +\infty$.

(a) By the definition of the critical exponent, $\delta \geq \delta_\sigma(\Gamma)$. 

(b) By Proposition~\ref{prop:some consequences of the shadow lemma}\eqref{item:convergent => conical limset null}, $\mu\left( \Lambda^{\rm con}(\Gamma) \right) = 0$.

(c) The $\Gamma$ action on $(M^{(2)}, \bar\mu \otimes \mu)$ is dissipative by Corollary~\ref{cor:conservativity dissipativity dichotomy}. Non-ergodicity of the action then follows from Lemma~\ref{lem:ergodic implies conservative}.


\part{Applications, Examples, and other Remarks}


\section{Rigidity of Patterson--Sullivan measures}\label{sec:rigidity of PS measures}


In this section, we prove that in the divergent case Patterson--Sullivan measures are either absolutely continuous or mutually singular. Furthermore, we  characterize the absolutely continuous case in terms of rough similarity between magnitudes. 

For the rest of the section, suppose $\Gamma \subset \mathsf{Homeo}(M)$ is a convergence group and $\sigma_1, \sigma_2 \colon \Gamma \times M \to \Rb$ are two expanding coarse-cocycles. For $i=1,2$, assume  $\mu_i$ is a coarse $\sigma_i$-Patterson--Sullivan measure of dimension $\delta_i$.

\begin{proposition} \label{prop:abscts or singular}
If  $\sum_{\gamma\in\Gamma} e^{-\delta_1 \norm{\gamma}_{\sigma_1}} = +\infty$, then either: 
\begin{enumerate}
\item $\mu_1 \ll \mu_2$ and $\mu_2 \ll \mu_1$, or 
\item $\mu_1 \perp \mu_2$.
\end{enumerate} 
\end{proposition}

\begin{proof}By the Lebesgue decomposition theorem, we can write 
$$
d\mu_1 = d\lambda + f d\mu_2
$$ 
where $\lambda \perp \mu_2$ and $f$ is a non-negative measurable function. 

Since $\lambda \perp \mu_2$, we can fix a decomposition $M = A \cup B$ where $A$ has full $\lambda$-measure, $B$ has full $\mu_2$-measure, and $\lambda(B) = \mu_2(A) = 0$. 
Since $\mu_2$ is $\Gamma$-quasi-invariant, by replacing $A$ by $\bigcup_{\gamma \in \Gamma} \gamma A$, we can assume that $A$ is $\Gamma$-invariant. 
Then, since $\Gamma$ acts ergodicallly on $(M,\mu_1)$, see Theorem~\ref{thm:ergodicity on single M}, $A$ either has zero or full $\mu_1$-measure.
If $A$ has zero $\mu_1$-measure, then $\mu_1 = f \mu_2$ and $\mu_1 \ll \mu_2$.  If $A$ has full $\mu_1$-measure, then $\mu_1 = \lambda$ and $\mu_1 \perp \mu_2$. 

It remains to show that  if $\mu_1 \ll \mu_2$, then $\mu_2 \ll \mu_1$. Since $\sum_{\gamma\in\Gamma} e^{-\delta_1 \norm{\gamma}_{\sigma_1}} = +\infty$, 
Theorem~\ref{thm:ergodicity on single M} implies that $\mu_1(\Lambda^{\rm con}(\Gamma)) = 1$. Then since $\mu_1 \ll \mu_2$, this implies that $\mu_2(\Lambda^{\rm con}(\Gamma)) > 0$.
Then Proposition~\ref{prop:some consequences of the shadow lemma} implies that $\sum_{\gamma\in\Gamma} e^{-\delta_2 \norm{\gamma}_{\sigma_2}} = +\infty$.  Then we can repeat the argument in the first two paragraphs to see that $\mu_2 \ll \mu_1$. 
\end{proof}

We can also characterize when the measures are absolutely continuous. 

\begin{proposition} \label{prop:characterisation of abscts case}
If  $\sum_{\gamma\in\Gamma} e^{-\delta_1 \norm{\gamma}_{\sigma_1}} = +\infty$, then the following are equivalent: 
\begin{enumerate}
\item \label{item:measures abscts} $\mu_1 \ll \mu_2$.
\item \label{item:measures abscts 2} $\mu_1 \ll \mu_2$ and $\mu_2 \ll \mu_1$.
\item \label{item:magnitudes similar} $\sup_{\gamma \in \Gamma}\abs{ \delta_1\norm{\gamma}_{\sigma_1} - \delta_2\norm{\gamma}_{\sigma_2} } < +\infty$.
\item \label{item:measures biLip} There exist $C>0$ such that $C^{-1}\mu_1 \leq  \mu_2 \leq C \mu_1$.
\end{enumerate}
\end{proposition}

Before starting the proof of Proposition~\ref{prop:characterisation of abscts case}, we first prove a lemma which states that two expanding coarse-cocycles have coarsely the same magnitudes if and only if they have coarsely the same periods.

\begin{lemma}\label{prop:magnitude versus period} If $\Gamma \subset \mathsf{Homeo}(M)$ is a convergence group and $\sigma_1,\sigma_2 \colon \Gamma \times M \rightarrow \Rb$ are two expanding $\kappa$-coarse-cocycles, then the following are equivalent: 
\begin{enumerate} 
\item $\sup\limits_{\gamma \in \Gamma,\ x\in \Lambda(\Gamma)} \abs{\sigma_1(\gamma,x)-\sigma_2(\gamma,x)} < +\infty$,
\item $\sup\limits_{\gamma \in \Gamma} \abs{\norm{\gamma}_{\sigma_1} - \norm{\gamma}_{\sigma_2}} < +\infty$,
\item $\sup\limits_{\substack{\gamma \in \Gamma  \\ \gamma \text{ \rm loxodromic}} } \abs{\sigma_1(\gamma,\gamma^+) - \sigma_2(\gamma, \gamma^+)} \leq 2\kappa$,
\item $\sup\limits_{\substack{\gamma \in \Gamma  \\ \gamma \text{ \rm loxodromic}} } \abs{\sigma_1(\gamma,\gamma^+) - \sigma_2(\gamma, \gamma^+)} < +\infty$.
\end{enumerate}
\end{lemma}

\begin{proof}Notice that the implication $(1)\Rightarrow(2)$  follows from the expanding property, the implication $(2)\Rightarrow(1)$  follows from Lemma~\ref{lem:nice magnitude}, the implication $(2)\Rightarrow(3)$ follows from Proposition~\ref{prop:basic properties}\eqref{item:loxodromic periods}, and the implication $(3)\Rightarrow(4)$  is clear. 

It only remains to  show that $(4)\Rightarrow(2)$. To that end, suppose that 
$$
C_1:=\sup_{\substack{\gamma \in \Gamma  \\ \gamma \text{ \rm loxodromic}} } \abs{\sigma_1(\gamma,\gamma^+) - \sigma_2(\gamma, \gamma^+)} 
$$
is finite. Fix a compatible metric $\dist$ on $\Gamma \sqcup M$. Then fix a finite set $F \subset \Gamma$ and $\epsilon > 0$ as in Lemma~\ref{lem:AMS_Gromov_hyp}. By the expanding property, there exists $C_2 > 0$ such that 
$$
\norm{\gamma}_{\sigma_j} - C_2 \leq \sigma_j(\gamma, x) \leq \norm{\gamma}_{\sigma_j}+C_2
$$
for $j\in \{1,2\}$,  $\gamma \in \Gamma$, and $x \in M$ with $\dist(x,\gamma^{-1}) \geq \epsilon$. By Proposition~\ref{prop:basic properties}\eqref{item:tri inequality} there exists $C_3 > 0$ such that 
$$
\norm{\gamma}_{\sigma_i}-C_3\le \norm{\gamma f}_{\sigma_i}\le \norm{\gamma}_{\sigma_i}+C_3
$$
for all $\gamma \in \Gamma$ and $f \in F$. 

Given $\gamma \in \Gamma$, choose $f \in F$ as in Lemma~\ref{lem:AMS_Gromov_hyp}. Then
\begin{align*}
 \abs{\norm{\gamma}_{\sigma_1} - \norm{\gamma}_{\sigma_2}} & \leq  2C_3 + \abs{\norm{\gamma f}_{\sigma_1} - \norm{\gamma f}_{\sigma_2}} \\
& \leq   2C_3 +2C_2 + \abs{\sigma_1(\gamma f,(\gamma f)^+) - \sigma_2(\gamma f, (\gamma f)^+)}\\
& \leq  2C_3+2C_2 + C_1. 
\end{align*}
Thus (2) is true. 
\end{proof}

We may combine this lemma with the earlier statements to obtain a further characterization of the absolutely continuous case when our cocycles are continuous.

\begin{corollary} \label{cor:PS measures and length spectra}
Suppose in addition that $\sigma_1, \sigma_2$ are continuous expanding cocycles (i.e. $\kappa=0$ in Definition~\ref{defn:expanding}) with $\sum_{\gamma\in\Gamma} e^{-\delta_1\norm{\gamma}_{\sigma_1}} = +\infty$. Then $\mu_1, \mu_2$ are absolutely continuous if and only if $\delta_1\sigma_1(\gamma,\gamma^+) = \delta_2\sigma_2(\gamma,\gamma^+)$ for all loxodromic $\gamma \in \Gamma$, and mutually singular otherwise.
\end{corollary}

\begin{proof}
By Proposition \ref{prop:abscts or singular}, $\mu_1$ and $\mu_2$ are either mutually singular or absolutely continuous. By Proposition \ref{prop:characterisation of abscts case}, they are absolutely continuous if and only if $\sup_{\gamma\in\Gamma} \abs{ \delta_1 \norm{\gamma}_{\sigma_1} - \delta_2 \norm{\gamma}_{\sigma_2}} < \infty$. By Lemma \ref{prop:magnitude versus period}, this last condition is equivalent, when $\sigma_1, \sigma_2$ are continuous cocycles, to the condition that $\delta_1\sigma_1(\gamma,\gamma^+) = \delta_2 \sigma_2(\gamma,\gamma^+)$ for all loxodromic $\gamma \in \Gamma$.
\end{proof}

The proof of Proposition~\ref{prop:characterisation of abscts case} occupies the rest of the section.
Many of the implications in the proposition follow immediately. Proposition~\ref{prop:abscts or singular} implies that \eqref{item:measures abscts} $\Leftrightarrow$ \eqref{item:measures abscts 2}. 
By definition \eqref{item:measures biLip}  $\Rightarrow$ \eqref{item:measures abscts 2}.
By the Shadow Lemma (Theorem~\ref{thm:shadow lemma}), \eqref{item:measures biLip}  $\Rightarrow$ \eqref{item:magnitudes similar}. 

The implication \eqref{item:magnitudes similar} $\Rightarrow$  \eqref{item:measures biLip} is a consequence of Lemma~\ref{prop:magnitude versus period} and Theorem~\ref{thm:ergodicity on single M}. Indeed, if \eqref{item:magnitudes similar} holds, then Lemma~\ref{prop:magnitude versus period} implies that 
$$
\sup_{\gamma \in \Gamma,\, x\in \Lambda(\Gamma)}\abs{\delta_1\sigma_1(\gamma,x)-\delta_2\sigma_2(\gamma,x)}<+\infty.
$$
Further, $\sum_{\gamma\in\Gamma} e^{-\delta_2 \norm{\gamma}_{\sigma_2}} = +\infty$ and so Theorem~\ref{thm:ergodicity on single M} implies that $\mu_2$ is supported on $\Lambda(\Gamma)$. Hence $\mu_2$ is a coarse $\sigma_1$-Patterson--Sullivan measure of dimension $\delta_1$. Then by the coarse uniqueness statement in Theorem~\ref{thm:ergodicity on single M} we see that  \eqref{item:measures biLip} is true.

The rest of the section is devoted to the proof that  \eqref{item:measures abscts} $\Rightarrow$ \eqref{item:measures biLip}. So suppose that $\mu_1 \ll \mu_2$.

By Theorem~\ref{thm:ergodicity on single M} there exists $\epsilon_0 > 0$ such that $\mu_1(\Lambda^{\rm con}_{2\epsilon_0}(\Gamma)) = 1$. Then
since $\mu_1 \ll \mu_2$,  we must have $\mu_2(\Lambda^{\rm con}_{2\epsilon_0}(\Gamma)) >0$. Then Proposition~\ref{prop:some consequences of the shadow lemma} implies that $\sum_{\gamma\in\Gamma} e^{-\delta_2 \norm{\gamma}_{\sigma_2}} = +\infty$. So by Theorem~\ref{thm:ergodicity on single M}, $\Gamma$ acts ergodically on $(M,\mu_2)$. Since $\Lambda^{\rm con}_{2\epsilon_0}(\Gamma)$ is $\Gamma$-invariant, see Observation~\ref{obs:invariance of uniform conical limit sets}, we then have  $\mu_2(\Lambda^{\rm con}_{2\epsilon_0}(\Gamma)) = 1$.
By shrinking $\epsilon_0 > 0$,  if necessary, we may also assume that it satisfies the hypothesis of the Shadow Lemma (Theorem~\ref{thm:shadow lemma}) for $\mu_1$ and $\mu_2$. 

Let $f := \frac{d\mu_1}{d\mu_2}$. Notice that for every $\gamma \in \Gamma$,
\begin{equation}\label{eqn:change of variable for f} 
f \circ \gamma = \frac{d\gamma_*^{-1}\mu_1}{d\gamma^{-1}_*\mu_2}= \frac{d\gamma_*^{-1}\mu_1}{d\mu_1}\frac{d\mu_2}{d\gamma_*^{-1}\mu_2} f =e^{\delta_2 \sigma_2(\gamma, \cdot ) - \delta_1 \sigma_1(\gamma, \cdot)}f
\end{equation} 
$\mu_2$-almost everywhere. Since $\mu_2$ and $\mu_1 = f\mu_2$ are probability measures, there exists $D_0 > 1$ such that the set 
$$
E:=\{ D_0^{-1} \leq f \leq D_0\}
$$
has positive $\mu_2$-measure. 

Using Lemmas~\ref{lem:Leb diff theorem} and~\ref{lem:zooming in on density points} we can fix $x_0 \in E \cap \Lambda^{\rm con}_{2\epsilon_0}(\Gamma)$ such that: if $0 < \epsilon \leq \epsilon_0$ and $\{\gamma_n\} \subset \Gamma$ is an escaping sequence with $x_0 \in \bigcap_{n \geq 1} \mc S_{\epsilon}(\gamma_n)$, then 
\begin{equation}\label{eqn:zooming in on E in proof of abs cont}
1 = \lim_{n \rightarrow \infty} \mu_2(\gamma_n^{-1}E) 
\end{equation} 
and 
\begin{equation}\label{eqn:limit 1 in proof of abs cont}
f(x_0) = \lim_{n \rightarrow \infty} \frac{1}{\mu_2(\mc S_{\epsilon}(\gamma_n))} \int_{\mc S_{\epsilon}(\gamma_n)} f(y) d\mu_2(y) = \lim_{n \rightarrow \infty} \frac{\mu_1(\mc S_{\epsilon}(\gamma_n))}{\mu_2(\mc S_{\epsilon}(\gamma_n))}.
\end{equation}

We will construct two sequences of group elements to use in the above limits. 
Since $x_0 \in \Lambda^{\rm con}_{2\epsilon_0}(\Gamma)$, Lemma~\ref{lem:shadows versus uniform conical limit sets} implies that there exists an escaping sequence $\{\gamma_{1,n}\}$ such that 
$$
x_0 \in \bigcap_{n \geq 1} \mc S_{\epsilon_0}(\gamma_{1,n}).
$$
Passing to a subsequence we can suppose that $\gamma_{1,n}^{-1}(x_0) \rightarrow a_1 \in M$ and $\gamma_{1,n}^{-1} \rightarrow b_1 \in M$. 
Then fix $\alpha \in \Gamma$ such that $\alpha^{-1}b_1 \neq b_1$.
Then let $\gamma_{2,n} : = \gamma_{1,n} \alpha$ and  
$$
b_2 : = \alpha^{-1} b_1 = \lim_{n \rightarrow \infty} \gamma_{2,n}^{-1}. 
$$

\begin{lemma} After passing to a subsequence, there exists $\epsilon> 0$ such that $$x_0  \in \bigcap_{n \geq 1} \mc S_{\epsilon}(\gamma_{2,n}).$$ \end{lemma} 

\begin{proof} Note that $\gamma_{1,n}^{-1}(x_0) \in\gamma_{1,n}^{-1} \mc S_{\epsilon_0}(\gamma_{1,n})=M-B_{\epsilon_0}(\gamma_{1,n}^{-1})$. 
So $a_1 \neq b_1$.
Hence $\epsilon:=2\dist(\alpha^{-1}a_1, \alpha^{-1}b_1)$ is positive.
Since $(\gamma_{1,n}\alpha)^{-1}x_0 \rightarrow \alpha^{-1}a_1$ and $(\gamma_{1,n}\alpha)^{-1} \rightarrow \alpha^{-1}b_1$, after passing to a subsequence we can suppose that $\dist((\gamma_{1,n}\alpha)^{-1}x_0, (\gamma_{1,n}\alpha)^{-1} ) > \epsilon$ for all $n$. 
Then $x_0 \in \mc S_{\epsilon}(\gamma_{1,n} \alpha)=\mc S_{\epsilon}(\gamma_{2,n})$ for all $n$. 
\end{proof} 

Shrinking $\epsilon > 0$ we can assume 
$$
\big( M - B_{2\epsilon}(b_1) \big) \cup \big( M - B_{2\epsilon}(b_2) \big) = M
$$
and that  $\epsilon \leq \epsilon_0$. 
Since $\gamma_{i,n}^{-1} \rightarrow b_i$, passing to a subsequence we can also suppose that 
$$
M - B_{2\epsilon}(b_i) \subset M-B_\epsilon(\gamma_{i,n}^{-1})
$$
for all $n \geq 1$ and $i=1,2$.

\begin{lemma} There exists $D_1> 1$ such that 
$$
D_1^{-1} \leq f(x) \leq  D_1
$$
for $\mu_2$-almost every point $x\in M$. 
\end{lemma} 

\begin{proof} 
Since $\epsilon \leq \epsilon_0$, Equation~\eqref{eqn:limit 1 in proof of abs cont} implies 
$$
f(x_0) =\lim_{n \rightarrow \infty} \frac{\mu_1(\mc S_\epsilon(\gamma_{i,n}))}{\mu_2(\mc S_\epsilon(\gamma_{i,n}))}.
$$
Since $x_0 \in E$, we have $f(x_0) \in [D_0^{-1}, D_0]$.
So by the Shadow Lemma (Theorem~\ref{thm:shadow lemma}), there exists $C_1 > 0$ such that 
$$
\abs{ \delta_1\norm{\gamma_{i,n}}_{\sigma_1} -  \delta_2\norm{\gamma_{i,n}}_{\sigma_2}} \leq C_1
$$
for every $n \geq 1$.
Then, since the cocycles are expanding, this implies that there exists $C_2 > 0$ such that 
$$
\abs{\delta_1 \sigma_1(\gamma_{i,n}, x) - \delta_2 \sigma_2(\gamma_{i,n}, x) } \leq C_2
$$
when $x \in M- B_{2\epsilon}(b_i) \subset M-B_\epsilon(\gamma_{i,n}^{-1})$.
So Equation~\eqref{eqn:change of variable for f}  implies that 
$$
D_0^{-1} e^{-C_2} \leq f(x) \leq  D_0 e^{C_2}
$$
$\mu_2$-almost everywhere on $\gamma_{i,n}^{-1}E \cap (M- B_{2\epsilon}(b_i))$. Then Equation~\eqref{eqn:zooming in on E in proof of abs cont} implies that 
$$
D_0^{-1} e^{-C_2} \leq f(x) \leq  D_0 e^{C_2}
$$
$\mu_2$-almost everywhere on $M- B_{2\epsilon}(b_i)$. Since $\big( M - B_{2\epsilon}(b_1) \big) \cup \big( M - B_{2\epsilon}(b_2) \big) = M$, this completes the proof. 
\end{proof}


\section{Convexity of critical exponent}\label{sec:convexity of critical exponent}


In this section we prove Theorem~\ref{thm:convexity of entropy in intro}, which we restate below. For the rest of the section fix a convergence group $\Gamma \subset \mathsf{Homeo}(M)$ and two expanding coarse-cocycles $\sigma_0, \sigma_1 \colon \Gamma \times M \to \Rb$ 
such that $\delta_{\sigma_0}(\Gamma)= 1=\delta_{\sigma_1}(\Gamma)$.
For $0 < \lambda < 1$, notice that 
$$
\sigma_\lambda= (1-\lambda)\sigma_0 + \lambda \sigma_1
$$
is also a coarse-cocycle. 

\begin{theorem}\label{thm:convexity of entropy} If $0 < \lambda < 1$, then 
$$\delta_{\sigma_\lambda}(\Gamma) \leq 1.$$
Moreover, if $\sum_{\gamma \in \Gamma} e^{-\delta_{\sigma_\lambda}(\Gamma) \norm{\gamma}_{\sigma_\lambda}} = +\infty$, then the following are equivalent: 
\begin{enumerate}
\item $\delta_{\sigma_\lambda}(\Gamma) = 1$.
\item $\sup\limits_{\gamma \in \Gamma} \abs{ \norm{\gamma}_{\sigma_0} - \norm{\gamma}_{\sigma_1}} < +\infty$.
\end{enumerate}
\end{theorem}

We start that by observing that the magnitudes of group elements behave nicely under convex combinations of cocycles. 

\begin{lemma}\label{lem:magnitudes under convex hulls}  $\sigma_\lambda$ is expanding and we may assume that 
$$
 \norm{\gamma}_{\sigma_\lambda} = (1-\lambda)\norm{\gamma}_{\sigma_0}+\lambda\norm{\gamma}_{\sigma_1}
$$
for all $\gamma \in \Gamma$. 
\end{lemma} 

\begin{proof} This follows immediately from the definition. 
\end{proof} 

\begin{lemma} \label{lem:delta<1}
$\delta_{\sigma_\lambda}(\Gamma) \leq 1$. 
\end{lemma} 

\begin{proof}
By H\"older's inequality and the previous lemma,
\begin{equation*}
\sum_{\gamma \in \Gamma} e^{-s \norm{\gamma}_{\sigma_\lambda}} \leq \left( \sum_{\gamma \in \Gamma} e^{-s \norm{\gamma}_{\sigma_0}} \right)^{\lambda} \left( \sum_{\gamma \in \Gamma} e^{-s \norm{\gamma}_{\sigma_1}} \right)^{1-\lambda}.
\end{equation*}
Hence $\delta_{\sigma_\lambda}(\Gamma) \leq 1$. 
\end{proof}

\begin{remark}
Using that $\delta_{t\sigma}(\Gamma)=\delta_\sigma(\Gamma)/t$ for all cocycles $\sigma$ and $t>0$, Lemma~\ref{lem:delta<1} implies that if $\sigma$ and $\sigma'$ are expanding cocycles  with finite critical exponents and $\lambda\in[0,1]$, then
\[\delta_{\lambda\sigma+(1-\lambda)\sigma'}(\Gamma) \leq \left( \lambda\delta_\sigma(\Gamma)^{-1} + (1-\lambda)\delta_{\sigma'}(\Gamma)^{-1} \right)^{-1}.\]
\end{remark}

We now consider the ``moreover'' part of the theorem. So fix $\lambda \in (0,1)$ where   
$$
\sum_{\gamma \in \Gamma} e^{-\delta_{\sigma_\lambda}(\Gamma) \norm{\gamma}_{\sigma_\lambda}} = +\infty.
$$
It is clear that $(2)\Rightarrow (1)$. The proof that $(1)\Rightarrow (2)$ is much more complicated and will occupy the rest of the section. 

To that end, suppose that $\delta_{\sigma_\lambda}(\Gamma) = 1$.
For $t \in \{0, \lambda, 1\}$, let $\mu_t$ be a coarse $\sigma_t$-Patterson-Sullivan measure of dimension 1 (which exists by Theorem~\ref{thm:PS measures exist}).

The key step in the proof is to show that  $\mu_\lambda$ is absolutely continuous to $\mu_0 + \mu_1$.

\begin{proposition}\label{prop:mu_lambda abs cont to mu_0 plus mu_1} $\mu_\lambda \ll \mu_0 + \mu_1$. \end{proposition}  

Assuming Proposition~\ref{prop:mu_lambda abs cont to mu_0 plus mu_1} for a moment we finish the proof that $(1) \Rightarrow (2)$. Since $\mu_\lambda \ll \mu_0 + \mu_1$, at least one of $\mu_0$ or $\mu_1$ is not singular to $\mu_\lambda$. So by relabelling we can assume that $\mu_\lambda$ is not singular to $\mu_0$. Then Proposition~\ref{prop:characterisation of abscts case} implies that 
$$
\sup\limits_{\gamma \in \Gamma} \abs{ \norm{\gamma}_{\sigma_0} - \norm{\gamma}_{\sigma_\lambda}} < +\infty.
$$
Then Lemma~\ref{lem:magnitudes under convex hulls}  implies that  
$$
\sup\limits_{\gamma \in \Gamma} \abs{ \norm{\gamma}_{\sigma_0} - \norm{\gamma}_{\sigma_1}} < +\infty.
$$

\subsection{Proof of Proposition~\ref{prop:mu_lambda abs cont to mu_0 plus mu_1}} 

The idea is to use the Shadow Lemma to relate the measures. 

Fix a compatible metric on $\Gamma \sqcup M$ and let $\mc S_\epsilon(\gamma)$ denote the associated shadows. By the Shadow Lemma (Theorem~\ref{thm:shadow lemma}), there exists $\epsilon_0 > 0$ such that for every $0<\epsilon \leq \epsilon_0$ there is constant $C_0(\epsilon) > 1$ where 
$$
\frac{1}{C_0(\epsilon)} e^{-\norm{\gamma}_{\sigma_t}} \leq \mu_t\Big( \mc S_\epsilon(\gamma) \Big) \leq C_0(\epsilon) e^{-\norm{\gamma}_{\sigma_t}}
$$
for all $\gamma \in \Gamma$ and all $t \in \{0, \lambda, 1\}$.

We first establish bounds for the measure of shadows and then extend these bounds to arbitrary sets using the covering result in Proposition~\ref{prop:properties of fake shadows}\eqref{item:Vish covering lemma}. 

\begin{lemma} For any $0 < \epsilon \leq \epsilon_0$ there exists $C_1 =C_1(\epsilon)> 1$ such that: if $\gamma \in \Gamma$, then
$$
\mu_\lambda\left( \mc S_\epsilon(\gamma) \right) \leq C_1(\mu_0 + \mu_1)\left( \mc S_\epsilon(\gamma) \right).
$$
\end{lemma} 

\begin{proof} By the Shadow Lemma and Lemma~\ref{lem:magnitudes under convex hulls}, 
\begin{align*}
\mu_\lambda&\left( \mc S_\epsilon(\gamma) \right) \leq C_0(\epsilon) e^{ - \norm{\gamma_2}_{\sigma_\lambda} } = C_0 (\epsilon)  e^{-(1-\lambda)\norm{\gamma_2}_{\sigma_0}  -\lambda  \norm{\gamma_2}_{\sigma_1} } \\
& \leq C_0(\epsilon)^2  \mu_0 \left( \mc S_\epsilon(\gamma) \right)^{1-\lambda}\mu_1 \left( \mc S_\epsilon(\gamma) \right)^\lambda. 
\end{align*}
Then the desired estimate follows from the weighted arithmetic-geometric mean inequality. 
\end{proof}

\begin{lemma} There exists $C_2 > 1$ such that: if $A \subset M$ is Borel measurable, then
$$
\mu_\lambda( A  ) \leq C_2( \mu_0 + \mu_1)( A).
$$
Hence $\mu_\lambda \ll \mu_0 + \mu_1$. 
\end{lemma} 

\begin{proof} Fix $\eta > 0$. By outer regularity we can find an open set $U \subset M$ such that $A \subset U$ and 
$$
( \mu_0 + \mu_1)( U)< \eta + ( \mu_0 + \mu_1)( A).
$$

Using Theorem~\ref{thm:ergodicity on single M} and possibly shrinking $\epsilon_0>0$, we may assume that
\begin{equation}\label{eqn:full mass at epsilon0}
\mu_\lambda(\Lambda_{\epsilon_0}^{\rm con}(\Gamma)) = 1.
\end{equation}
Now fix $0 < \epsilon < \epsilon_0$ and let 
$$
I := \{ \gamma \in \Gamma : \mc S_{\epsilon}(\gamma) \subset U\}.
$$
By Lemma \ref{lem:shadows versus uniform conical limit sets}, for each $x \in \Lambda_{\epsilon_0}^{\rm con}(\Gamma)$ there exists an escaping sequence $\{\gamma_n\}$ such that 
$$
x \in \bigcap_{n =1}^\infty \mc S_{\epsilon}(\gamma_n). 
$$
Moreover, for each such sequence, Proposition~\ref{prop:properties of fake shadows}\eqref{item:collapsing fake shadows} implies that ${\rm diam} \, \mc S_{\epsilon}(\gamma_n) \to 0$ as $n\to\infty$. Hence
$$
U \cap \Lambda_{\epsilon_0}^{\rm con}(\Gamma) \subset  \bigcup_{\gamma \in I} \mc S_{\epsilon}(\gamma) \subset U. 
$$
So by Equation~\eqref{eqn:full mass at epsilon0},
$$
\mu_\lambda\left( U \right) = \mu_\lambda\left( \bigcup_{\gamma \in I} \mc S_{\epsilon}(\gamma) \right).
$$

Let $J \subset I$ and $\epsilon' < \epsilon$ satisfy Proposition~\ref{prop:properties of fake shadows}\eqref{item:Vish covering lemma}. Then repeatedly using the Shadow Lemma,
\begin{align*}
\mu_\lambda& ( A) \leq \mu_\lambda(U)=  \mu_\lambda\left( \bigcup_{\gamma \in I} \mc S_{\epsilon}(\gamma) \right)  \leq \sum_{\gamma \in J} \mu_\lambda\left( \mc S_{\epsilon^\prime}(\gamma)  \right) \\
& \leq C_0(\epsilon) C_0(\epsilon^\prime)\sum_{\gamma \in J} \mu_\lambda\left( \mc S_{\epsilon}(\gamma)  \right) \leq C_0(\epsilon) C_0(\epsilon^\prime)C_1(\epsilon) \sum_{\gamma \in J}( \mu_0 +  \mu_1)\left( \mc S_{\epsilon}(\gamma)  \right) \\
& \leq C_0(\epsilon) C_0(\epsilon^\prime)C_1(\epsilon) ( \mu_0 +  \mu_1)(U)  \leq C_0(\epsilon) C_0(\epsilon^\prime)C_1(\epsilon) \Big( ( \mu_0 + \mu_1)(A) + \eta \Big). 
\end{align*}
Since $\eta > 0$ was arbitrary, this completes the proof. 
\end{proof} 


\section{Symmetric coarse-cocycles}\label{sec:symmetric cocycles}


In this section we consider the case when an expanding coarse-cocycle is ``coarsely-symmetric.'' For the rest of the section, fix a convergence group $\Gamma \subset \mathsf{Homeo}(M)$. A coarse-cycle $\sigma : \Gamma \times M \rightarrow \Rb$ is \emph{coarsely-symmetric} if 
$$
\sup_{\substack{\gamma \in \Gamma  \\ \gamma \text{ \rm loxodromic}} } \abs{\sigma(\gamma,\gamma^+) - \sigma(\gamma^{-1}, \gamma^-)} < +\infty.
$$

The next result shows that coarsely-symmetric can also be defined using magnitudes and that expanding coarsely-symmetric coarse-cocycles are always contained in a coarse GPS system.

\begin{proposition}\label{prop:symmetric implies GPS} Suppose $\sigma \colon \Gamma \times M \rightarrow \Rb$ is an expanding coarse-cocycle. Then $\sigma$ is coarsely-symmetric if and only if 
$$
\sup_{\gamma \in \Gamma} \abs{\norm{\gamma}_\sigma - \norm{\gamma^{-1}}_\sigma} < + \infty.
$$
Moreover, if $\sigma$ is coarsely-symmetric, then $(\sigma, \sigma, G)$ is a coarse GPS system for $\Gamma$ acting on $\Lambda(\Gamma) \subset M$, where $G \colon \Lambda(\Gamma)^{(2)} \to \Rb$ is defined by 
\begin{align*}
G(x,y)  = \limsup_{\alpha \to x, \beta \to y} \norm{\alpha}_\sigma +\norm{\beta^{-1}}_\sigma-\norm{\beta^{-1}\alpha}_\sigma.
\end{align*}
\end{proposition} 

\begin{remark} One could also define a Gromov product using a limit infimum instead of a limit supremum. \end{remark}

As a corollary to Proposition~\ref{prop:symmetric implies GPS} and Theorem~\ref{thm:ergodicity on product}, we have the following. 

\begin{corollary} Suppose $\sigma \colon \Gamma \times M \rightarrow \Rb$ is a coarsely-symmetric expanding coarse-cocycle with $\delta : = \delta_\sigma(\Gamma) < +\infty$ and $\mu$ is a coarse $\sigma$-Patterson--Sullivan measure of dimension $\delta$. If 
$$
\sum_{\gamma \in \Gamma} e^{-\delta \norm{\gamma}_\sigma} = +\infty, 
$$
then $\Gamma$ acts ergodically on $(M^{(2)}, \mu \otimes \mu)$. 
\end{corollary} 

We prove the proposition via a series of lemmas. Fix, for the rest of the section, an expanding $\kappa$-coarse-cocycle $\sigma \colon \Gamma \times M \to \Rb$ and a compatible distance $\dist$ on $\Gamma \sqcup M$.

\begin{lemma} $\sigma$ is coarsely-symmetric if and only if 
$$
\sup_{\gamma \in \Gamma} \abs{\norm{\gamma}_\sigma - \norm{\gamma^{-1}}_\sigma} < + \infty.
$$
\end{lemma} 

\begin{proof} $(\Leftarrow)$: This follows from Proposition~\ref{prop:basic properties}\eqref{item:loxodromic periods}.

$(\Rightarrow)$: Fix $\epsilon > 0$ and a finite set $F \subset \Gamma$ satisfying Lemma~\ref{lem:AMS_Gromov_hyp}. By Proposition~\ref{prop:basic properties}\eqref{item:tri inequality} there exists $C_1 > 0$ such that 
$$
\norm{\gamma}_\sigma - C_1 \leq \norm{\gamma f}_\sigma  \leq \norm{\gamma}_\sigma + C_1 \quad \text{and} \quad \norm{\gamma}_\sigma - C_1 \leq \norm{f^{-1}\gamma}_\sigma  \leq \norm{\gamma}_\sigma + C_1 
$$
for all $\gamma \in \Gamma$ and $f \in F$. Further, by the expanding property, there exists $C_2 > 0$ such that: if $\gamma \in \Gamma$ and $\dist(x, \gamma^{-1}) > \epsilon$, then 
$$
\norm{\gamma}_\sigma -C_2 \leq \sigma(\gamma, x) \leq \norm{\gamma}_\sigma +C_2. 
$$

Now fix $\gamma \in \Gamma$. Then there exists $f \in F$ such that $\gamma f$ is loxodromic and 
$$
\min\left\{  \dist( \gamma f, (\gamma f)^-), \,  \dist( (\gamma f)^+, (\gamma f)^{-1}) \right\} > \epsilon.
$$
Then 
\begin{align*}
\abs{\norm{\gamma}_\sigma - \norm{\gamma^{-1}}_\sigma} & \leq 2C_1 + \abs{\norm{\gamma f}_\sigma - \norm{(\gamma f)^{-1}}_\sigma} \\
& \leq 2C_1 + 2C_2 + \abs{\sigma(\gamma f,(\gamma f)^+) - \sigma((\gamma f)^{-1}, (\gamma f)^-)}.
\end{align*}
Then, since $\sigma$ is coarsely-symmetric, 
\begin{align*}
\sup_{\gamma \in \Gamma} \abs{\norm{\gamma}_\sigma - \norm{\gamma^{-1}}_\sigma} & < + \infty.  
\qedhere
\end{align*}
\end{proof}

\begin{lemma}\label{lem:symmetric implies GPS} If $\sigma$ is coarsely-symmetric, then $(\sigma, \sigma, G)$ is a coarse GPS system on $\Lambda(\Gamma)$. \end{lemma} 

\begin{proof} Notice that $G$ is locally bounded by Proposition~\ref{prop:basic properties}\eqref{item:multiplicative estimate}. Also, by Lemma~\ref{lem:nice magnitude}, there exists $C_1 > 0$ such that for every $x \in \Lambda(\Gamma)$ and $\gamma \in \Gamma$ we have
$$
 \limsup_{\alpha \rightarrow x}\abs{ \sigma(\gamma, x) - (\norm{\gamma \alpha}_\sigma - \norm{\alpha}_\sigma)} \leq C_1.
$$
Then by the previous lemma, there exists $C_2 > 0$ such that 
$$
 \limsup_{\alpha \rightarrow x} \abs{ \sigma(\gamma, x)-(\norm{ \alpha^{-1}\gamma^{-1}}_\sigma - \norm{\alpha^{-1}}_\sigma)} \leq  C_2.
 $$
 
Fix $(x,y) \in \Lambda(\Gamma)^{(2)}$ and $\gamma \in \Gamma$. By definition there exist $\alpha_n \rightarrow x$ and $\beta_n \rightarrow y$ such that 
$$
G(\gamma x, \gamma y) = \lim_{n \rightarrow \infty} \norm{\gamma \alpha_n}_\sigma +\norm{\beta_n^{-1}\gamma^{-1}}_\sigma-\norm{\beta_n^{-1}\alpha_n}.
$$
Then, by the definition of $G$, 
$$
G(x, y) \geq  \limsup_{n \rightarrow \infty} \norm{\alpha_n}_\sigma +\norm{\beta_n^{-1}}_\sigma-\norm{\beta_n^{-1}\alpha_n}.
$$
So 
\begin{align*}
G(\gamma x, \gamma y) - G(x,y) & \leq \liminf_{n \rightarrow \infty} \norm{\gamma \alpha_n}_\sigma- \norm{\alpha_n}_\sigma + \norm{\beta_n^{-1}\gamma^{-1}}_\sigma-\norm{\beta_n^{-1}}_\sigma \\
& \leq C_1 + C_2 + \sigma(\gamma, x) + \sigma(\gamma, y).
\end{align*}
Using the definition of $G$ again,  there exist $\hat\alpha_n \rightarrow x$ and $\hat\beta_n \rightarrow y$ such that 
$$
G(x, y) = \lim_{n \rightarrow \infty} \norm{\hat\alpha_n}_\sigma +\norm{\hat\beta_n^{-1}}_\sigma-\norm{\hat\beta_n^{-1}\hat\alpha_n}.
$$
Then 
$$
G(\gamma x, \gamma y) \geq \limsup_{n \rightarrow \infty} \norm{\gamma \hat\alpha_n}_\sigma +\norm{\hat\beta_n^{-1}\gamma^{-1}}_\sigma-\norm{\hat\beta_n^{-1}\hat\alpha_n}.
$$
So 
\begin{align*}
G(\gamma x, \gamma y) - G(x,y) & \geq \limsup_{n \rightarrow \infty} \norm{\gamma \hat\alpha_n}_\sigma- \norm{\hat\alpha_n}_\sigma + \norm{\hat\beta_n^{-1}\gamma^{-1}}_\sigma-\norm{\hat\beta_n^{-1}}_\sigma \\
& \geq  - C_1 - C_2 + \sigma(\gamma, x) + \sigma(\gamma, y).
\end{align*}
Thus 
$$
\abs{ \Big( \sigma(\gamma, x) + \sigma(\gamma, y) \Big) - \Big( G(\gamma x, \gamma y) - G(x,y) \Big)} \leq C_1+C_2
$$
and hence $(\sigma, \sigma, G)$ is a coarse GPS system. 
\end{proof} 


\section{Potentials on Gromov hyperbolic spaces}\label{sec:potentials}


For the rest of the section let $(X,\dist_X)$ be a proper geodesic Gromov hyperbolic metric space and fix a basepoint $o \in X$. Also, let $\Gamma \subset \mathsf{Isom}(X)$ be a discrete group. Then $\Gamma$ acts on the Gromov boundary $\partial_\infty X$ as a convergence group. 

In this section we consider coarsely additive potentials on $X$, as defined in Definition~\ref{defn:coarsely additive potential}, and prove Theorems~\ref{thm: coarsely additive potentials introduction} and~\ref{thm:hyperbolic2_potential introduction} (which we restate here).

\begin{theorem}\label{thm:hyperbolic1_potential} 
Suppose $\psi$ is a $\Gamma$-invariant coarsely additive potential. 
Define functions $\sigma_{\psi}, \bar\sigma_{\psi} \colon \Gamma \times \partial_\infty X \to [-\infty,\infty]$ and $G_{\psi} \colon \partial_\infty X^{(2)} \to  [-\infty,\infty]$ by 
\begin{align*}
\sigma_{\psi}(\gamma,x) & = \limsup_{p \rightarrow x} \, \psi(\gamma^{-1}o, p) - \psi(o,p), \\ 
\bar\sigma_{\psi}(\gamma,x) & = \limsup_{p \rightarrow x} \, \psi(p,\gamma^{-1}o) - \psi(p,o), \\
G_{\psi}(x,y) & = \limsup_{p \rightarrow x, q \rightarrow y}\,  \psi(p,o) +\psi(o,q) - \psi(p,q).
\end{align*}
Then these quantities are finite, with $G_\psi$ bounded below, $(\sigma_{\psi}, \bar\sigma_{\psi}, G_{\psi})$ is a coarse GPS-system, and one can choose 
$$
\norm{\gamma}_{\sigma_{\psi}} = \psi(o,\gamma o).
$$
\end{theorem}

\begin{remark} As mentioned in the introduction, most of Theorem~\ref{thm:hyperbolic1_potential} can be derived from results in~\cite{DT2023}, but for the reader's convenience we provide a complete proof. 
\end{remark}

\begin{theorem}\label{thm:hyperbolic2_potential} Suppose $\Gamma$ acts co-compactly on $X$ and $\sigma \colon \Gamma \times \partial_\infty X \to \Rb$ is an expanding coarse-cocycle. Then there exists a $\Gamma$-invariant coarsely additive potential where
$$
\sup_{\gamma \in \Gamma, x \in \partial_\infty X} \abs{\sigma_{\psi}(\gamma,x)-\sigma(\gamma,x) } < + \infty.
$$
In particular, $\sigma$ is contained in a coarse GPS system. 
\end{theorem} 

\subsection{Metric perspective}\label{sec:metric perspective} If $\psi: X \times X \rightarrow \Rb$ is a $\Gamma$-invariant coarsely additive potential, then by Lemma~\ref{lem:basic properties of potentials} below there exists a constant $C_0 > 0$ such that the function 
$$
\dist_\psi(p,q) = \begin{cases} \psi(p,q) + C_0 & \text{if } p \neq q \\ 0 & \text{if } p=q \end{cases}
$$
is a $\Gamma$-invariant quasimetric, that is a function that satisfies all the axioms of a metric except for the symmetry property. Using properties~\eqref{item:potential proper} and~\eqref{item:potential coarsely additive} in Definition~\ref{defn:coarsely additive potential} one can show that $(X,\dist_\psi)$ is quasi-isometric to $(X,\dist_X)$, and that $(X,\dist_\psi)$ is coarsely-geodesic, i.e. there is some $C > 0$ such that every two points in $X$ are joined by a $(1,C)$-quasi-geodesic with respect to the quasimetric $\dist_\psi$. 

Conversely, given a $\Gamma$-invariant coarsely-geodesic quasimetric $\dist$ on $X$ which is quasi-isometric to $(X,\dist_X)$, the Morse lemma implies that the function $\psi(x,y) = \dist(x,y)$ is a $\Gamma$-invariant coarsely additive potential. 

Hence Theorems~\ref{thm:hyperbolic1_potential} and~\ref{thm:hyperbolic2_potential} could be instead be stated in terms of $\Gamma$-invariant coarsely-geodesic quasimetrics which are quasi-isometric to $(X,\dist_X)$.

\subsection{Proof of Theorem~\ref{thm:hyperbolic1_potential}} Suppose $\psi \colon X \times X \to \Rb$ is a $\Gamma$-invariant coarsely additive potential and $\kappa \colon [0,\infty) \to [0,\infty)$ is the function in property~\eqref{item:potential coarsely additive} of Definition~\ref{defn:coarsely additive potential}.

Since $(X,\dist_X)$ is Gromov hyperbolic, there exists $\delta > 0$ such that every geodesic triangle in $(X,\dist_X)$ is $\delta$-slim.

We first show that $\psi$ satisfies a coarse version of the triangle inequality. 

\begin{lemma}\label{lem:basic properties of potentials} \ 
\begin{enumerate}
\item For every $r > 0$ there exists $C(r) > 0$ such that: 
$$
\abs{\psi(p,q)-\psi(p',q')} \leq C(r)
$$
when $\dist_X(p,p')$, $\dist_X(q,q') \leq r$. 
\item There exists $\kappa_1 > 0$ such that: 
$$
\psi(p_1,p_2) \leq \psi(p_1,q) + \psi(q,p_2) + \kappa_1
$$
for all $p_1,p_2,q \in X$.
\end{enumerate}
\end{lemma} 

\begin{proof} (1). Notice that $p'$ is in the $(r+1)$-neighborhood of any geodesic joining $p$ to $q$ and $q'$ is in the $(r+1)$-neighborhood of any geodesic joining $p'$ to $q$. So 
$$
\abs{ \psi(p,q) - \big( \psi(p,p') + \psi(p',q') + \psi(q',q) \big)} \leq 2\kappa(r+1). 
$$
Hence
$$
\abs{\psi(p,q)-\psi(p',q')} \leq 2\kappa(r+1)+2 \sup_{\dist_X(u,v) \leq r} \abs{\psi(u,v)},
$$
which is finite by property~\eqref{item:potential bounded} of Definition~\ref{defn:coarsely additive potential}.

(2). Let $m : = \inf_{p,q \in X} \psi(p,q)$, which is finite by properties~\eqref{item:potential proper} and \ref{item:potential bounded} of Definition~\ref{defn:coarsely additive potential}.

Fix $p_1,p_2,q \in X$ and a geodesic triangle $[p_1,p_2] \cup [p_2, q] \cup [q,p_1]$ in $X$ with vertices $p_1,p_2,q$. Since every geodesic triangle is $\delta$-slim, there exists  $u \in [p_1,p_2]$, $p_1' \in [q,p_1]$ and $p_2' \in [p_2,q]$ such that 
$$
\dist_X(p_1',u), \dist_X(p_2',u) < \delta. 
$$
Then 
\begin{align*}
\psi(p_1,p_2) & \leq \psi(p_1,u) +\psi(u,p_2)+\kappa(0)  \leq \psi(p_1,p_1') + \psi(p_2',p_2) + 2C(\delta) + \kappa(0)  \\
& \leq \psi(p_1,q) -\psi(p_1',q)+ \psi(q,p_2)-\psi(q,p_2') +2C(\delta) + 3\kappa(0) \\
& \leq \psi(p_1,q) + \psi(q,p_2)-2m +2C(\delta) + 3\kappa(0). \qedhere
\end{align*} 
\end{proof} 

The next lemma states that it coarsely doesn't matter what sequence we use to define $\sigma_\psi$. 

\begin{lemma}\label{lem:cocycles well defined potentials} There exists $\kappa_2 > 0$ such that: if $x \in \partial_\infty X$ and $\gamma \in \Gamma$, then 
$$
\limsup_{p,q \rightarrow x} \abs{\big(  \psi(\gamma^{-1}o, p) - \psi(o,p)\big) - \big( \psi(\gamma^{-1}o, q) - \psi(o,q) \big) } \leq \kappa_2. 
$$
\end{lemma} 

\begin{proof} Since geodesic triangles are $\delta$-slim, for any two geodesic rays $r_1, r_2 \colon [0,\infty) \to (X,\dist_X)$ with $\lim_{t \to \infty} r_1(t) = \lim_{t \to \infty} r_2(t)$ there exists $T > 0$ such that $r_1([T,\infty))$ is contained in the $2\delta$-neighborhood of  $r_2$.  

Fix $x \in \partial_\infty \Gamma$ and $\gamma \in \Gamma$. Then fix sequences $\{p_n\}, \{q_n\} \subset \Gamma$ converging to $x$ where 
$$
L:=\lim_{n \rightarrow \infty} \abs{\big(  \psi(\gamma^{-1}o, p_n) - \psi(o,p_n)\big) - \big( \psi(\gamma^{-1}o, q_n) - \psi(o,q_n) \big) }
$$
equals the limit supremum in the lemma statement. Using the fact mentioned above, after passing to a subsequence, we can find $u \in X$ such that $u$ is contained in the $(2\delta+1)$-neighborhood of any geodesic joining $o$ to either $p_n$ or $q_n$, and $u$ is contained in the $(2\delta+1)$-neighborhood of any geodesic joining $\gamma^{-1} o$ to either $p_n$ or $q_n$.  Then 
\begin{align*}
& \abs{\big(\psi(\gamma^{-1}o, p_n) - \psi(o,p_n)\big)  - \big(\psi(\gamma^{-1}o, u) - \psi(o,u)\big) } \\
& \quad =  \abs{  \big(\psi(\gamma^{-1}o, p_n)- \psi(\gamma^{-1}o, u)-\psi(u, p_n)\big) - \big( \psi(o,p_n)-\psi(o,u)-\psi(u,p_n) \big)} \\
& \quad \leq  2 \kappa(2\delta+1).
\end{align*}
Likewise, 
\begin{align*}
\abs{\big(\psi(\gamma^{-1}o, q_n) - \psi(o,q_n)\big)  - \big(\psi(\gamma^{-1}o, u) - \psi(o,u)\big) }   \leq 2 \kappa(2\delta+1). 
\end{align*}
So $L \leq 4 \kappa(2\delta+1)$. 

\end{proof}

\begin{lemma}\label{lem:buseman function is a coarse-cocycle} $\sigma_{\psi} \colon \Gamma \times \partial_\infty X  \to \Rb$ is a coarse-cocyle.
 \end{lemma} 
 
 \begin{proof} Fix $\gamma_1, \gamma_2 \in \Gamma$ and a sequence $\{p_n\} \subset X$ converging to $x \in \partial_\infty X$. Then by Lemma~\ref{lem:cocycles well defined potentials}, 
 \begin{align*}
&  \abs{ \sigma_{\psi}(\gamma_1\gamma_2, x) - \sigma_{\psi}(\gamma_1, \gamma_2 x) - \sigma_\psi(\gamma_2, x)}\\
& \quad\leq3\kappa_2+ \limsup_{n \rightarrow \infty} \big| \psi(\gamma_2^{-1}\gamma_1^{-1}o, p_n) - \psi(o,p_n) -\psi(\gamma_1^{-1}o, \gamma_2p_n) + \psi(o,\gamma_2 p_n)\\
& \quad \quad \quad \quad\quad \quad\quad \quad \quad \quad-\psi(\gamma_2^{-1}o, p_n) + \psi(o,p_n) \big|\\
& \quad = 3\kappa_2. 
 \end{align*}
 
 Next fix $\gamma \in \Gamma$ and $\{x_n\} \subset \partial_\infty X$ converging to $x$. Then we can fix $\{p_{n,j}\} \subset X$ such that $\lim_{j \rightarrow \infty} p_{n,j} = x_n$. Then using Lemma~\ref{lem:cocycles well defined potentials}, we can fix $\{j_n\}$ such that 
 $$
 \sup_{n \geq 1} \abs{ \sigma_{\psi}(\gamma, x_n) - \psi(\gamma^{-1}o, p_{n,j_n}) + \psi(o,p_{n,j_n})} \leq 2\kappa_2
 $$
 and $p_{n,j_n} \rightarrow x$. Using Lemma~\ref{lem:cocycles well defined potentials} again,
 \begin{align*}
 \limsup_{n \rightarrow \infty} \abs{\sigma_{\psi}(\gamma, x) - \sigma_{\psi}(\gamma,x_n)} & \leq 2\kappa_2+ \limsup_{n \rightarrow \infty} \abs{ \sigma_{\psi}(\gamma, x) - \psi(\gamma^{-1}o, p_{n,j_n}) +\psi(o,p_{n,j_n})} \\
 & \leq 3\kappa_2.
 \end{align*}
 Thus $\sigma_\psi$ is a $3\kappa_2$-coarse-cocycle. 
 \end{proof}
 
 Next we verify that $\sigma_\psi$ is expanding. To that end, fix a compatible distance $\dist$ on $\Gamma \sqcup \partial_\infty X$. Notice that if $x \in \partial_\infty X$, then $\dist(\gamma_n,x) \rightarrow 0$ if and only if $\gamma_np \rightarrow x$ in the topology on $X \sqcup \partial_\infty X$ for all $p \in X$. 

\begin{lemma}\label{lem:buseman function is a expanding coarse-cocycle}  For every $\epsilon > 0$ there exists $C > 0$ such that: if $\gamma \in \Gamma$, $x \in \partial_\infty X$ and $\dist(x, \gamma^{-1}) > \epsilon$, then 
$$
\psi(o,\gamma o) - C \leq \sigma_\psi(\gamma, x) \leq \psi(o,\gamma o) + C. 
$$
Hence $\sigma_\psi$ is an expanding coarse-cocycle and we can choose $\norm{\gamma}_\psi = \psi(o, \gamma o)$. 
\end{lemma} 

\begin{proof} We first note that there exists $r =r(\epsilon)> 0$ such that: if $\gamma \in \Gamma$, $x \in \partial_\infty X$, $\dist(x, \gamma^{-1}) > \epsilon$, and $p_n \rightarrow x$, then for $n$ sufficiently large any geodesic segment joining $\gamma^{-1}o$ to $p_n$ intersects $B_r(o)$. 

Then fix $\gamma \in \Gamma$ and $x \in \partial_\infty X$ with $\dist(x, \gamma^{-1}) > 0$. Then 
\begin{align*}
\abs{\sigma_{\psi}(\gamma,x) -\psi(o, \gamma o)} & = \abs{\sigma_{\psi}(\gamma,x) -\psi(\gamma^{-1}o, o)} \\
& = \abs{\limsup_{p \to x}  \psi(\gamma^{-1}o, p) - \psi(o,p)- \psi(\gamma^{-1} o, o)} \leq \kappa(r). 
\end{align*}
Thus the first assertion is true. 

{Property~\eqref{item:potential proper} implies that $ \psi(o, \gamma_n o) \rightarrow + \infty$ for any escaping sequence $\{\gamma_n\}$. So Proposition~\ref{prop:different definition of expanding} implies $\sigma_\psi$ is expanding and we can choose $\norm{\gamma}_\psi = \psi(o, \gamma o)$. }
\end{proof}

The next lemma states that it coarsely doesn't matter what sequence we use to define $G_\psi$. 

\begin{lemma}\label{lem:formula for Gromov product} There exists $\kappa_3 > 0$ such that: if $\ell$ is a geodesic line in $(X,\dist_X)$ with endpoints $x,y \in \partial_\infty X$, then 
$$
\limsup_{p\rightarrow x, q\rightarrow y} \abs{\psi(p,o) + \psi(o,q) - \psi(p,q) - \inf_{u \in \ell} \psi(u,o) + \psi(o,u) } \leq \kappa_3. 
$$
\end{lemma} 

\begin{proof} Fix a geodesic line $\ell$ with endpoints $x,y \in \partial_\infty X$. Then fix sequences $\{p_n\}, \{q_n\} \subset X$ converging to $x,y$ which realize the  limit supremum in the lemma statement. Passing to a subsequence we can suppose that 
$$
L:=\lim_{n \rightarrow \infty} \psi(p_n,o) + \psi(o,q_n) - \psi(p_n,q_n)
$$
exists in $[-\infty, +\infty]$. Then Lemma~\ref{lem:basic properties of potentials} implies that $L \in [-\kappa_1, +\infty]$. 

Fix a geodesic $[p_n,q_n]$ joining $p_n$ to $q_n$. Passing to a subsequence we can suppose that $[p_n,q_n]$ converges to a geodesic line $\hat\ell$ with endpoints $x,y$. Since every geodesic triangle is $\delta$-slim, $\ell$ must be contained in a $2\delta$-neighborhood of $\hat\ell$ and $\hat\ell$ must be contained in a $2\delta$-neighborhood of $\ell$.

First suppose that $u \in \ell$. Then for $n$ sufficiently large, $u$ is in the $(2\delta+1)$-neighborhood of $[p_n,q_n]$. Hence 
$$
\psi(p_n,q_n) \geq \psi(p_n,u) + \psi(u,q_n) - \kappa(2\delta+1),
$$
which implies 
\begin{align*}
L & \leq \kappa(2\delta+1)+ \limsup_{n \rightarrow \infty} \psi(p_n,o) -\psi(p_n,u)+ \psi(o,q_n) - \psi(u,q_n) \\
& \leq \kappa(2\delta+1) + 2\kappa_1 + \psi(o,u) + \psi(u,o). 
\end{align*}
Thus 
$$
L \leq \kappa(2\delta+1) + 2\kappa_1 + \inf_{u \in \ell} \psi(u,o) + \psi(o,u).
$$
Notice that this implies that $L < +\infty$. 

For each $n$, let $[o,p_n] \cup [p_n,q_n] \cup [q_n, o]$ be a geodesic triangle with vertices $o,p_n,q_n$. Since every geodesic triangle is $\delta$-slim, exist $u_n \in [p_n,q_n]$, $p_n' \in [0,p_n]$, and $q_n' \in [0,q_n]$ such that 
$$
\dist_X(p_n',u_n), \dist_X(q_n', u_n) < \delta. 
$$
Then 
\begin{align*}
\psi & (p_n,o)  + \psi(o,q_n) - \psi(p_n,q_n) \\
&  \geq \psi(p_n,u_n)+\psi(u_n,o) + \psi(o,u_n)+\psi(u_n,q_n) - \psi(p_n,u_n) - \psi(u_n,q_n)-3\kappa(\delta) \\
& = \psi(u_n,o) + \psi(o,u_n) - 3\kappa(\delta).
\end{align*} 
Hence 
$$
L \geq - 3\kappa(\delta) + \limsup_{n \rightarrow \infty}  \psi(u_n,o) + \psi(o,u_n) .
$$
Since  $L < +\infty$, property~\eqref{item:potential proper} of Definition~\ref{defn:coarsely additive potential} implies that $\{u_n\}$ is relatively compact in $X$. Since any limit point of $\{u_n\}$ is contained in $\hat\ell$, for $n$ sufficiently large $u_n$ is contained in a $(2\delta+1)$-neighborhood of $\ell$. Thus by Lemma~\ref{lem:basic properties of potentials},
\begin{equation*}
L \geq - 3\kappa(\delta) -2C(2\delta+1)+ \inf_{u \in \ell} \psi(u,o) + \psi(o,u). \qedhere
\end{equation*}
\end{proof}

\begin{lemma} $(\sigma_{\psi}, \bar\sigma_{\psi}, G_{\psi})$ is a coarse GPS system. \end{lemma} 

\begin{proof} Notice that $\bar\psi(p,q) = \psi(q,p)$ defines a  $\Gamma$-invariant coarsely additive potential. So  Lemma~\ref{lem:buseman function is a expanding coarse-cocycle} implies that $\sigma_\psi$ and $\bar \sigma_{\psi} = \sigma_{\bar \psi}$ are both expanding coarse-cocycles. Hence, by definition, they are proper coarse-cocycles.

Next we show that $G_{\psi}$ is locally finite. Fix a compact set $K \subset \partial_\infty X^{(2)}$. Then there exists $r > 0$ such that any geodesic line in $(X,\dist_X)$ joining points in $K$ intersects the ball of radius $r > 0$ centered at $o$. Then Lemma~\ref{lem:formula for Gromov product}  and property~\eqref{item:potential bounded} of Definition~\ref{defn:coarsely additive potential} imply that 
$$
\sup_{(x,y) \in K} G_\psi(x,y) < +\infty. 
$$
Hence $G_{\psi}$ is locally finite.

Finally, arguing exactly as in the proof of Lemma~\ref{lem:symmetric implies GPS} there exists a constant $C >0$ such that 
$$
\abs{ G_{\psi}(\gamma x, \gamma y) - G_{\psi}(x,y) - \bar\sigma_{\psi}(\gamma, x) - \sigma_{\psi}(\gamma,y)} \leq C
$$
for all $\gamma \in \Gamma$ and $(x,y) \in \partial_\infty X^{(2)}$. 

Hence $(\sigma_{\psi}, \bar\sigma_{\psi}, G_{\psi})$ is a coarse GPS system.
\end{proof} 

 This completes the proof of Theorem \ref{thm:hyperbolic1_potential}.

\subsection{Proof of Theorem~\ref{thm:hyperbolic2_potential}} Suppose $\sigma \colon \Gamma \times \partial_\infty X \rightarrow \Rb$ is an expanding $\kappa$-coarse-cocycle and $\Gamma$ acts co-compactly on $X$. 
Fix  $s > 0$ such that $X = \Gamma \cdot B_s(o)$.

For $p \in X$ let $A_p := \{ \gamma \in \Gamma : \dist_X(p,\gamma(o)) < s\}$. 
Then define $\psi \colon X \times X \to \Rb$ by 
$$
\psi(p,q) = \frac{1}{\#A_p \# A_q} \sum_{\gamma_1 \in A_p, \gamma_2 \in A_q} \norm{\gamma_1^{-1}\gamma_2}_\sigma. 
$$
We will show that $\psi$ is a $\Gamma$-invariant coarsely additive potential.

Since $\gamma A_p = A_{\gamma p}$, the function $\psi$ is $\Gamma$-invariant. 
By Proposition~\ref{prop:basic properties}\eqref{item:properness}, 
$$
\lim_{r \rightarrow \infty} \inf_{\dist_X(p,q) \geq r} \psi(p,q) = +\infty.
$$
Since $\Gamma$ acts properly on $X$, for any $r > 0$ we have 
$$
\sup_{\dist_X(p,q) \leq r}\abs{ \psi(p,q)} < +\infty.
$$ 

\begin{lemma} There exists $C > 0$ such that: If $\alpha \in A_x$ and $\beta \in A_y$, then 
\begin{equation}\label{eqn:comparing two magnitudes 1} 
\abs{\norm{\alpha^{-1}\beta}_\sigma - \psi(x,y)} \leq C.
\end{equation} 
\end{lemma} 

\begin{proof} Let $B : = \{ \gamma \in \Gamma : \dist_X(o,\gamma(o)) <  2s\}$. Notice that if $\alpha, \hat\alpha \in A_x$ and $\beta, \hat\beta \in A_y$, then there exists $f_1, f_2 \in B$ such that $\hat\alpha = \alpha f_1$ and $\hat\beta=\beta f_2$. So by Proposition~\ref{prop:basic properties}\eqref{item:tri inequality} there exists $C > 0$ (which only depends on $B$) such that 
$$
\abs{ \norm{\alpha^{-1}\beta}_\sigma-\norm{\hat\alpha^{-1}\hat\beta}_\sigma} \leq C. 
$$
Then, by definition, 
\begin{equation*}
\abs{\norm{\alpha^{-1}\beta}_\sigma - \psi(x,y)} \leq C
\end{equation*}
whenever $\alpha \in A_x$ and $\beta \in A_y$. 
\end{proof}

\begin{lemma} For every $r > 0$ there exists $\kappa=\kappa(r) > 0$ such that: if $u$ is contained in the $r$-neighborhood of a geodesic in $(X,\dist_X)$ joining $p$ to $q$, then 
$$
\abs{\psi(p,q) - \big(\psi(p,u) + \psi(u,q)\big) } \leq \kappa.
$$
\end{lemma} 

\begin{proof} Fix $r > 0$ and suppose no such $\kappa(r) > 0$ exists.  Without loss of generality $r\ge s$.
Then for each $n \geq 1$ we can find $p_n, q_n, u_n \in X$ such that $u_n$ is contained in the $r$-neighborhood of a geodesic joining $p_n$ to $q_n$ and 
$$
\abs{\psi(p_n,q_n) - \big(\psi(p_n,u_n) + \psi(u_n,q_n)\big) } \geq n. 
$$
Translating by $\Gamma$, we can assume that $p_n \in B_s(o)$, which implies that $\id \in A_{p_n}$. Fix $\alpha_n \in A_{u_n}$ and $\beta_n \in A_{q_n}$. Then Equation~\eqref{eqn:comparing two magnitudes 1} implies that 
$$
\abs{ \norm{\beta_n}_\sigma - \norm{\alpha_n}_\sigma - \norm{\alpha_n^{-1} \beta_n}_\sigma} \geq n-3C. 
$$
However, Proposition~\ref{prop:basic properties}\eqref{item:multiplicative estimate} implies that there exists $C' > 0$ such that 
$$
\abs{\norm{\beta_n}_\sigma - \norm{\alpha_n}_\sigma - \norm{\alpha_n^{-1} \beta_n}_\sigma} \leq C'
$$
and so we have a contradiction. 
\end{proof} 

Thus $\psi$ is a $\Gamma$-invariant coarsely additive potential. Notice that Equation~\eqref{eqn:comparing two magnitudes 1}  implies that
\begin{equation}\label{eqn:comparing two magnitudes 2} 
\abs{\norm{\gamma}_\sigma - \psi(o,\gamma o)} \leq C
\end{equation} 
for all $\gamma \in \Gamma$. So, by the definition of $\sigma_\psi$, Equation~\eqref{eqn:comparing two magnitudes 2}, and Lemma~\ref{lem:nice magnitude} we have
$$
\sup_{\gamma \in \Gamma, x \in \partial_\infty X} \abs{\sigma_{\psi}(\gamma,x)-\sigma(\gamma,x) } < + \infty.
$$
This completes the proof of Theorem~\ref{thm:hyperbolic2_potential}.


\appendix

\section{Conservativity, dissipativity and quotient measures}\label{appendix:conservative and dissipative}

In this appendix we define the notions of conservativity, dissipativity and Hopf decompositions for a general group action, check that it coincides with several other definitions in the literature \cite{Kaimanovich,Aaronson,roblin}, 
and also that it is consistent with the classical theory of Hopf decompositions for actions of $\Z$.
This expands on the discussion in \cite{BlayacPS}. 
We also prove that quotient measures exist when the action is dissipative.

We include this appendix because the references we found on this topic were not entirely suitable for this paper: some sources \cite{roblin,BlayacPS} are missing details, while others \cite{Kaimanovich,Aaronson} only apply to free actions (while here we allow actions which are not free).

For the rest of the section fix a measurable space $X$, a unimodular, locally compact second-countable group $\ms G$ acting measurably on $X$, and an $\mathsf G$-invariant sigma-finite measure $m$.
Recall: every discrete group is unimodular.
Denote by $dg$ a fixed choice of Haar measure on $\ms G$. Since $\ms G$ is unimodular, this measure is invariant under both left and right multiplication, and under the involution $g\mapsto g^{-1}$.

\subsection{The Hopf decomposition}\label{sec:hopf decompo}

There are several reasonable definitions of wandering sets, which generalize in different ways the classical notion of wandering sets for actions of $\Z$, $\R$ and $\Z_{\geq1}$.
We use the following:
\begin{definition}
A measurable subset $W\subset X$ is called \emph{wandering} (\resp \emph{exactly wandering}) if $\{g\in \ms G: gx\in W\}$ is relatively compact for $m$-almost any (\resp for any) $x\in W$.
\end{definition}

When $\ms G$ is discrete, in particular if $\ms G=\Z$, then a set $W$ is sometimes called wandering if it satisfies the stronger property that $W\cap g W=\emptyset$ for any $g\in\ms G$, or $m(W\cap g W)=0$, see \cite{Aaronson,Kaimanovich}. 
We will see, in Section~\ref{sec:funddom}, the link between this stronger definition and ours. 
Roblin defines $W$ to be wandering if it satisfies the weaker property that $\int 1_W(gx)dg<+\infty$ for almost any $x\in W$~\cite[p.17]{roblin}. 
This gives the same notions of conservativity and dissipativity, as explained below.

\begin{definition}
The action of $\ms G$ on $(X,m)$ is called \emph{conservative} if every wandering set has measure zero. The action is called \emph{dissipative} if $X$ is a countable union of wandering sets.

A \emph{Hopf decomposition} of $X$ is a decomposition $X=C\sqcup D$ into disjoint $\ms G$-invariant measurable sets such that the action on $C$ is conservative and the action on $D$ is dissipative. 
\end{definition}

Notice that if $X = C \sqcup D$ and $X=C'\sqcup D'$ are both Hopf decompositions, then $C'\cap D$ is a countable union of wandering sets. 
Since every wandering set in $C'$ has measure zero, we see that $m(C'\cap D)=0$, and similarly $m(C\cap D')=0$. 
So, up to a set of measure zero, there is a unique Hopf decomposition.

There is another classical characterization of conservativity, dissipativity and Hopf decompositions in terms of integrable functions. This characterization also proves the existence of Hopf decompositions.

\begin{fact}[{\cite[Fact 2.27]{BlayacPS}}] \label{fact:Hopf decompo}
For any positive integrable function $f$ on $X$, the sets $C:=\{x:\int_{\ms G} f(gx) \,dg=+\infty\}$ and $D:=\{x:\int_{\ms G} f(gx) \,dg<+\infty\}$ form a Hopf decomposition.

In particular, the action of $\ms G$ is conservative (resp.\ dissipative) if and only if for any/some positive integrable function $f$ on $X$, we have $\int_{\ms G} f(gx) \,dg=+\infty$ (resp.\ $<+\infty$) for $m$-almost any $x\in X$.
\end{fact}

This result implies that our notion of Hopf decomposition coincides with that of Roblin~\cite[p.17]{roblin}.
If not, there would be $W\subset C$ with positive measure such that $h(x)=\int 1_W(gx) \,dg$ is finite for $m$-almost every $x \in W$, and up to reducing $W$ we can assume there exists $R$ such that $\{h(x) \leq R\}$ has full $m$-measure in $W$.
Then $\int_W h(x)f(x) \,dx<+\infty$, but this quantity equals  $\int_{x\in W}\int f(gx)\,dg\,dm(x)$, which is infinite since $\int f(gx)\,dg=+\infty$ on $C$.

A consequence of Fact~\ref{fact:Hopf decompo} is that a group $\ms G$ acting on $X$ has the same Hopf decomposition as any lattice of $\ms G$ acting on $X$ (if $\ms G$ has lattices), see e.g.\ \cite[Th.\,1.6.4]{Aaronson} in the case of free actions.

\subsection{The case of discrete groups}\label{sec:funddom}

In this section we suppose that $\ms G$ is discrete and $X$ is standard, 
 i.e.\ $X$ is measurably isomorphic to a complete separable metric space equipped with its Borel sigma-algebra.
The goal is to construct a measurable fundamental domain for the action of $\ms G$ on the dissipative part.
This will allow us to check that our definition of Hopf decomposition agree with other definitions \cite{Aaronson,Kaimanovich} when $\ms G$ is torsion-free.

\begin{lemma}\label{lem:funddom}
 Let $A\subset X$ be a $\ms G$-invariant measurable subset which can be written as a countable union of exactly wandering sets.
 Then there exists a measurable subset $\mc F\subset A$ such that every orbit $\ms G\cdot x$ intersects $\mc F$ at exactly one point.
 In particular, $\mc F$ is measurably isomorphic to $\ms G \backslash A$ endowed with the quotient sigma-algebra, and this quotient is hence standard (can be measurably embedded in $[0,1]$).
 
Moreover, whenever $\ms G$ is  a discrete group acting measurably on a standard space  $X$,
  there exists a Hopf decomposition $X=(X-A)\sqcup A$ so that $A$ is a countable union of exactly wandering sets.

\end{lemma}

\begin{proof}
 Let $\{W_n\}$ be a sequence of exactly wandering sets with $A=\bigcup_n W_n$.
 Let $W_1':=W_1$, and let $W_n':=W_n-\ms G(W_1\cup\dots \cup W_{n-1})$ for all $n > 1$. Then $\{W'_n\}$ is a sequence of exactly wandering sets such that the orbits $\ms G\cdot W_n'$ form a partition of $A$.
 To conclude the proof, it suffices to find a fundamental domain in each $W_n'$, i.e.\ to select in a measurable way one representative for each orbit which intersects $W_n'$.
 
 Let $\phi\colon X\to [0,1]$ be a measurable embedding.
 Now for each $n\geq 1$, we select in each orbit $\ms G\cdot x$ the point $y\in W_n'$ whose image under $\phi$ is the smallest, i.e.\ $y=\phi^{-1}(\min \phi(\ms G\cdot x\cap W_n'))$. So
 \[ \mc F_n:=\{x\in W_n':\phi(x)\leq \phi(gx) \quad \text{for any }g\in \ms G\text{ with } gx\in W_n'\}. \]
This set is measurable since $x\in \mc F_n$ if and only if $\phi(x)\leq \phi(gx)+1_{X-W_n'}(gx)$ for any $g\in\ms G$, which are countably many measurable conditions.
 Then $\mc F:=\bigcup_n\mc F_n$ is a measurable fundamental domain.
 
 To construct $A$ satisfying the ``moreover'' statement, consider a Hopf decomposition $X=C\sqcup D$, write $D$ as a countable union of wandering sets $\{W_n\}$, and then let $W_n'$ be the set of $x\in W_n$ such that $\{g:gx\in W_n\}$ is finite, so that $W_n'$ is exactly wandering and has full measure in $W_n$.
 Finally set $A:=\bigcup_n\ms G\cdot W_n'$.
\end{proof}

If $\ms G$ is torsion-free, then the action of $\ms G$ on the set $A\subset X$ constructed in Lemma \ref{lem:funddom}  is free.
As a corollary, any Hopf decomposition $X=C\sqcup D$ for our definition in Section~\ref{sec:hopf decompo} is also a Hopf decomposition in the sense of Aaronson \cite[\S1.6]{Aaronson} and Kaimanovich \cite{Kaimanovich}: every positive measure subset $B\subset C$ is \emph{recurrent}, meaning that for almost any $x\in B$ the orbit eventually returns to $B$ (because $B$ is not wandering), and $D$ admits a subset $\mc F$ such that $\{g\mc F\}_{g\in \ms G}$ are pairwise disjoint and $\ms G\cdot \mc F$ has full measure in $D$.

\subsection{Quotient measures}\label{sec:quotient measures}

In this section we assume the action of $\ms G$ on $X$ is dissipative.
Let $\pi\colon X\to \ms G\backslash X$ denote the projection map associated to the action and endow $\ms G\backslash X$ with the quotient sigma-algebra.

For any non-negative measurable function $f\colon X\to [0,+\infty]$, the function $$\tilde P(f)(x):= \int_{g\in\ms G}f(gx) \,dg$$ is measurable and $\ms G$-invariant, hence it descends to a measurable function on $\ms G\backslash X$ which we denote by $P(f)$.

We say that a measure $m'$ on $\ms G\backslash X$ is a \emph{quotient measure for $m$} if for any non-negative measurable function $f : X \rightarrow [0,+\infty]$ we have
\begin{equation}\label{eqn:defining property of quotient measures}
\int_{x\in X} f(x) \,dm(x) = \int_{q\in \ms G\backslash X} P(f(q) \,dm'(q). 
\end{equation} 

For instance, if $X$ is a smooth manifold, $\ms G$ is discrete and acts freely and properly discontinuously on $X$, and $m$ comes from a smooth $\ms G$-invariant volume form $\alpha$, then $\ms G\backslash X$ is a manifold and the quotient measure is induced by the volume form $\pi_*\alpha$.

We will show that quotient measures exist and are unique.

\begin{remark}\label{rem:quotient is sigmafinite} 
First we make some observations.
\begin{enumerate}
\item The quotient measure $m'$ is automatically sigma-finite, since $P(f)$ is a positive function in $L^1(\ms G \backslash X, m')$ whenever $f$ is a positive function in $L^1(X,m)$.
\item A $\ms G$-invariant measurable subset $A\subset X$ has zero $m$-measure if and only if its projection $\pi(A)$ (which is measurable) has zero $m'$-measure.
Indeed, let $f$ be a positive integrable function on $X$.
Then $P(f1_A)=P(f)1_{\pi(A)}$.
If $m(A)=0$ then $\int P(f)1_{\pi(A)}dm'=\int f1_Adm=0$ so $ P(f)1_{\pi(A)}=0$ almost everywhere, so $m'(\pi(A))=0$.
Conversely, if $m'(\pi(A))=0$ then $\int f1_Adm=0$ so $m(A)=0$. 
\item If $f \in L^1(X,m)$, then $\tilde P(f)(x)= \int_{g\in\ms G}f(gx) \,dg$ is an $m$-almost everywhere defined measurable function and hence it descends to a measurable $m'$-almost everywhere defined function on $\ms G\backslash X$ which we denote by $P(f)$. Equation~\eqref{eqn:defining property of quotient measures} implies that 
$$
\int P(f) dm' = \int f dm
$$
for all $f \in L^1(X,m)$. Since $\abs{P(f)} \leq P(\abs{f})$, Equation~\eqref{eqn:defining property of quotient measures} also  implies that
$$
P : L^1(X, m) \rightarrow L^1(\ms G \backslash X, m')
$$ 
is continuous.
\end{enumerate} 
\end{remark}

\begin{fact}
 There is a unique quotient measure on $\ms G\backslash X$, and it is given by the formula
 \[ m' = \frac1{P(f_0)} \pi_*(f_0m), \]
 where $f_0$ is any integrable positive function on $X$.
 Moreover, for any $\chi \colon \ms G\backslash X \to \R_{\geq 0}$, if
 $f=\tfrac{f_0}{\tilde P(f_0)}\chi\circ \pi $ then $P(f)=\chi$ $m'$-almost surely.
\end{fact}
\begin{proof}
 Since the action of $\ms G$ on $X$ is dissipative, $P(f_0)$ is finite $m$-almost surely by Fact~\ref{fact:Hopf decompo}.

 Let us prove uniqueness: let $m_1,m_2$ be quotient measures.
 Fix $\chi\colon \ms G\backslash X \to [0,+\infty]$ measurable.
 Set $f(x)=\chi(\pi(x))\tfrac{f_0(x)}{\tilde P(f_0)(x)}$ and observe that $P(f)(q)=\chi(q)$ for any $q\in \ms G\backslash X$ such that $P(f_0)(q)<\infty$, which occurs $m_i$-almost surely since $\int P(f_0)dm_i=\int f_0dm<\infty$, for any $i=1,2$.
 Thus $\int \chi dm_1=\int fdm=\int \chi dm_2$, which implies $m_1=m_2$ since $\chi$ was an arbitrary non-negative measurable function.
 
 Let us now check that  $ m' = P(f_0)^{-1} \pi_*(f_0m)$ is a quotient measure.
 Fix $f\colon X\to [0,+\infty]$.
 Then by Fubini, the $\ms G$-invariance of $m$, and the invariance of the Haar measure under $g\mapsto g^{-1}$, we have
 \begin{align*}
  \int_{q\in \ms G\backslash X} P(f)(q) \,dm'(q) & = \int_{q\in \ms G\backslash X} \frac{P(f)(q)}{P(f_0)(q)} \,d\pi_*(f_0m)(q)\\
  & = \int_{\ms G} \int_X f(gx) \frac{f_0(x)}{P(f_0)(gx)} \,dm(x) \,dg\\
  & = \int_X f(y) \int_{\ms G} \frac{f_0(g^{-1}y)}{P(f_0)(y)} \,dg \,dm(y)
    = \int_X f(y) \,dm(y).\qedhere
 \end{align*}
\end{proof}
 
 If $\ms G$ is discrete, then one can use the existence of a fundamental domain from Section~\ref{sec:funddom} to give a more concrete description of the quotient measure.
 
 \begin{fact}
  Suppose $X$ is standard and $\ms G$ is discrete.
  Let $\mc F\subset X$ be a measurable subset that intersects every $\Gamma$-orbit at most once and such that $\Gamma\cdot \mc F$ has full measure (as in the ``moreover'' part of Lemma~\ref{lem:funddom}).
  
  Then $\pi_*({f_0}m_{|\mc F})$ is the quotient measure, where $f_0(x)=\tfrac{1}{\#\{\gamma\in\Gamma:\gamma x=x\}}$.
 \end{fact}
 
\begin{proof}
 Let $f$ be a measurable non-negative function on $X$.
 
 For any finite subgroup $K\subset \ms G$, let $\mc F_K$ be the set of $x\in\mc F$ whose stabilizer is $K$.
 Then $\mc F$ is the disjoint countable union of the $\mc F_K$'s, and we have $\sum_\gamma 1_{\mc F_K}\circ\gamma=\#K\cdot 1_{\ms G \mc F_K}$ and $f_0(x)=(\#K)^{-1}$ for any $x\in \mc F_K$. So
 \begin{align*}
  \int P(f)d \pi_*(f_0m_{|\mc F}) 
  & = \int_{x\in \mc F}\sum_\gamma f(\gamma x)f_0(x) dm(x) \\
   & = \sum_{K\subset \ms G}\frac1{\#K} \int_{x\in \mc F} \sum_\gamma 1_{\mc F_K}(x)f(\gamma x)dm(x)\\
   & = \sum_{K\subset \ms G} \int_{y\in\ms G\mc F_K}f(y)dm(y)
   = \int f dm.\qedhere
 \end{align*}
\end{proof}

\subsection{The case $\ms G=\Z$}\label{sec:HopfZ}

In this section we consider the case when $\ms G=\Z$.
There is an abundant literature on the notions of conservativity, dissipativity and Hopf decomposition in this case, and more generally in the case of actions of the semigroup $\Z_{\geq 1}$.
We denote by $T^n$ the transformation of $X$ associated to an element $n$.

For any reasonable choice of definitions, it is obvious that conservativity of $\Z_{\geq 1}$ always implies conservativity of $\Z$ and that dissipativity of $\Z$ implies dissipativity of $\Z_{\geq 1}$.
It is well-known, although nontrivial, that the converses are also true.
We shall use Krengel\cite{Krengel} 
as a reference, and check that our definitions are consistent with the definitions there:
\begin{fact}
 Consider a decomposition $X=C\sqcup D$ which is a Hopf decomposition in the sense of Krengel \cite[Th.\,3.2]{Krengel}: $D$ admits a measurable subset $W_0$ such that $D=T^{\Z}W_0$ and $W_0\cap T^nW_0=\emptyset$ for any $n\neq 0$, and every subset $W\subset C$ with $W\cap T^nW=\emptyset$ for any $n\neq 0$ has measure zero.
 
 Then $X=C\sqcup D$ is a Hopf decomposition for our definition.
\end{fact}
\begin{proof}
 The action on $D$ is clearly dissipative for our definition.

If every subset $W\subset C$ with $W\cap T^nW=\emptyset$ for any $n\neq 0$ has measure zero, then for any measurable $A\subset C$, for almost any $x\in A$ there exist infinitely many $n$ such that $T^nx\in A$~\cite[Th.\,3.1]{Krengel}.
 This implies that the action on $C$ is conservative for our definition.
\end{proof}

\subsection{A topological Hopf decomposition}
Suppose the sigma-algebra of $X$ comes from a locally compact second-countable topology and the action of $\ms G$ is by homeomorphisms. 
In this case there is a natural Hopf decomposition that does not depend on $m$.

We say an orbit $\ms G\cdot x$ is \emph{escaping} if for any compact set $K$ the set $\{g:gx\in K\}$ is relatively compact, i.e.\ $gx\to\infty$ as $g\to\infty$.
Let $D\subset X$ be the set of $x$ such that $\ms G \cdot x$ is escaping, and $C=X-D$.
Note that $D$ is measurable because it is a countable intersection of closed sets of the form 
\[
\{x:(\ms G-L)\cdot x\subset X-\mr{int}(K)\}
=\bigcap_{g\in\ms G-L}g^{-1}(X-\mr{int}(K))
\]
for some compact sets $K\subset X$ and $L\subset \ms G$.

\begin{lemma}\label{lem:topHopf}
 $X=C\sqcup D$ is a Hopf decomposition for any $ \ms G$-invariant locally finite measure.
\end{lemma}
\begin{proof}
 Our assumptions imply the existence of a positive continuous integrable function $f$.
 Then $\int f(gx) \,dx=+\infty$ for any $x\in C$.
 Indeed let $x\in C$.
 Then there is a compact subset $K\subset X$ such that $\{g:gx\in K\}$ is not relatively compact.
 Fix $U\subset \ms G$ a compact neighborhood of the identity.
 By continuity there is some $\epsilon>0$ such that  $f(uy)>\epsilon$ when $u \in U$ and $y \in K$. 
 Let $\{g_n\}\subset \ms G$ be an escaping sequence such that $
 Ug_n$ are pairwise disjoint and $g_nx\in K$ for any $n$.
 Then
 \[ \int_{\ms G} f(gx) \,dg \geq \sum_n \int_{u\in U} f(ug_nx) \,du \geq \sum_n \epsilon \cdot \mr{Haar}(U)=+\infty.\]
 
 It remains to prove that $D$ is a countable union of wandering sets.
 In fact it is a countable union of exactly wandering sets $W_n$ ($\{g:gx\in W_n\}$ is relatively compact for any $x\in W_n$).
 Indeed let $\{K_n\}$ be a sequence of compact sets covering $X$, and let $W_n=K_n\cap D$.
 Then $W_n$ is exactly wandering.
\end{proof}

\bibliographystyle{alpha}
\bibliography{geom}

\end{document}